\documentclass[12pt]{article}
\usepackage{amsthm}

\usepackage[ansinew]{inputenc}
\usepackage{graphicx}
\usepackage{color}
\usepackage[colorlinks]{hyperref}

\usepackage{enumerate,latexsym}
\usepackage{latexsym}
\usepackage{amsmath,amssymb}
\usepackage{graphicx}
\usepackage{times}
\usepackage{float}
\usepackage[colorlinks]{hyperref}
\usepackage{hyperref}
\hypersetup{%
	colorlinks = true,
	linkcolor  = black
}
\newfont{\bb}{msbm10 at 11pt}
\newfont{\bbsmall}{msbm8 at 8pt}

\def\rth{\mathbb{R}^3}
\def\R{\mathbb{R}}
\def\B{\mathbb{B}}
\def\N{\mathbb{N}}
\def\Z{\mathbb{Z}}
\def\S{\Sigma}
\def\C{\mathbb{C}}

\def\Pe{\mathbb{P}}

\def\D{\mathbb{D}}
\def\Pe{\mathbb{P}}
\def\esf{\mathbb{S}}

\newcommand{\la}{\looparrowright}

\newcommand{\ben}{\begin{enumerate}}
\newcommand{\bit}{\begin{itemize}}
\newcommand{\een}{\end{enumerate}}
\newcommand{\eit}{\end{itemize}}
\newcommand{\wh}{\widehat}
\newcommand{\ds}{\displaystyle}
\newcommand{\Int}{\mbox{\rm Int}}
\newcommand{\Ind}{\mbox{\rm Index}}
\newcommand{\Inj}{\mbox{\rm Inj}}

\newcommand{\wt}{\widetilde}

\newcommand{\ed}{\end{document}}
\newcommand{\ov}{\overline}

\def\a{{\alpha}}

\def\t{{\theta}}

\def\g{{\gamma}}
\def\G{{\Gamma}}
\def\l{{\lambda}}
\def\L{\Lambda}
\def\de{{\delta}}
\def\be{{\beta}}
\def\ve{{\varepsilon}}

\def\cA{\mathcal{A}}
\def\cU{\mathcal{U}}
\def\cP{\mathcal{P}}
\def\cS{\mathcal{S}}
\def\cY{\mathcal{Y}}

\def\cB{\mathcal{B}}

\def\cE{\mathcal{E}}
\def\cV{\mathcal{V}}
\def\cD{\mathcal{D}}
\def\cH{\mathcal{H}}

\def\cW{\mathcal{W}}

\let\8=\infty \let\0=\emptyset

\def\cte.{\mathop{\rm cte.}\nolimits}

\def\N{\mathbb{N}}

\def\B{\mathbb{B}}

\def\R{\mathbb{R}}
\def\Z{\mathbb{Z}}
\def\C{\mathbb{C}}
\def\A{\mathbb{A}}
\def\D{\mathbb{D}}

\newtheorem{theorem}{Theorem}[section]
\newtheorem{lemma}[theorem]{Lemma}
\newtheorem{proposition}[theorem]{Proposition}

\newtheorem{remark}[theorem]{Remark}
\newtheorem{corollary}[theorem]{Corollary}
\newtheorem{definition}[theorem]{Definition}

\newtheorem{question}[theorem]{Question}
\newtheorem{claim}[theorem]{Claim}

\newtheorem{example}[theorem]{Example}

\newcommand{\W}{W}

\numberwithin{equation}{section}

\usepackage{fullpage}

\usepackage{pdfsync}
\definecolor{pp}{rgb}{.5,0,.7}

\definecolor{rr}{rgb}{.8,0,.3}

\begin{document}

\begin{title}
{Hierarchy structures in finite index CMC surfaces}
\end{title}

\begin{author}
{William H. Meeks III\thanks{This research was partially supported
 by CNPq - Brazil, grant no. 400966/2014-0. }
 \and Joaqu\'\i n P\' erez\thanks{Research of both
 authors was partially supported by
MINECO/MICINN/FEDER grant no. PID2020-117868GB-I00,
regional grant P18-FR-4049,
and by the “Maria de Maeztu” Excellence Unit IMAG,
reference CEX2020-001105-M, funded by MCINN/AEI/10.13039/501100011033/ CEX2020-001105-M.
}}
\end{author}
\maketitle
\vspace{-.6cm}

\begin{abstract}
Given $\ve_0>0$,  $I\in \N\cup \{0\}$ and $K_0,H_0\geq0$,
let $X$ be a complete Riemannian $3$-manifold with
injectivity radius  $\Inj(X)\geq \ve_0$ and with the  supremum
of absolute sectional curvature  at most
$K_0$, and let $ M \la X$ be a complete
immersed surface of constant mean curvature
$H\in [0,H_0]$ with index at most $I$.
For such $ M \la X$, we prove Structure Theorem~\ref{mainStructure}
which describes how the
interesting ambient geometry of the immersion is organized
locally around at most $I$ points of
$M$ where the norm of the second fundamental form takes on
large local maximum values.
\vspace{.3cm}


\end{abstract}
\maketitle
\tableofcontents

\section{Introduction}  \label{sec:introduction}
Let  $X$ denote a complete Riemannian $3$-manifold with positive
injectivity radius $\Inj(X)$ and bounded absolute sectional curvature.
Let $M$ be a complete immersed surface in $X$ of constant mean curvature
$H\geq 0$, which we call
an {\it $H$-surface} in $X$. The Jacobi operator of $M$ is the
Schr\"{o}dinger operator
\[
L=\Delta +|A_M|^2+ \mbox{Ric}(N),
\]
where
$\Delta $ is the Laplace-Beltrami operator on $M$, $|A_M|^2$ is the square
of the norm of its second fundamental form and $\mbox{Ric}(N)$ denotes
the Ricci curvature of $X$ in the direction of the unit normal vector $N$
to $M$; the index of $M$ is the index of $L$,
\[
\mbox{Index}(M)=\lim _{R\to \infty }\mbox{Index}(B_M(p,R)),
\]
where $B_M(p,R)$ is the intrinsic metric ball in $M$ of
radius $R>0$ centered at a point $p\in M$, and
$\mbox{Index}(B_M(p,R))$ is the number of negative eigenvalues of $L$
on $B_M(p,R)$ with Dirichlet boundary conditions. Here, we have assumed
that the immersion is two-sided (this holds in particular if $H>0$).
In the case, $H=0$ and the immersion is one-sided, then the index is
defined in a similar manner using compactly supported variations
in the normal bundle; see Definition~\ref{DefIndexNO} for details.

The primary goal of this paper is to describe the structure of complete
immersed $H$-surfaces $F\colon M\la X$ (also called {\it $H$-immersions})
which have a fixed  bound $I\in \N\cup \{0\}$ on their
index and a fixed upper bound $H_0$ for their constant mean curvatures
$H$, in certain small intrinsic neighborhoods of points with sufficiently
large norm $|A_M|$ of their second fundamental forms, see
Theorem~\ref{mainStructure}.
When $M$ has non-empty boundary, we will assume,
after a choice of some $\ve_0\in (0,\Inj(X))$, that
there is an  upper bound $A_0$ of $|A_M|$ in
the intrinsic $\ve_0$-neighborhood of the boundary of $M$.
Theorem~\ref{mainStructure} plays an important theoretical
role in understanding global
properties of such surfaces in much the same way that the local structure
theorems of Colding-Minicozzi \cite{cm23,cm25}
(for embedded minimal surfaces) and of
Meeks-Tinaglia~\cite{mt15} (for embedded
$H$-surfaces with $H>0$)  play a fundamental
role in understanding global properties of complete embedded $H$-surfaces
of finite genus, especially in the case where $X=\rth$. However,
we point out that the results in this paper do not depend on the
results for embedded $H$-surfaces of Colding-Minicozzi and Meeks-Tinaglia;
for applications of Theorem~\ref{mainStructure}
to the global theory of finite index $H$-surfaces in Riemannian
3-manifolds, see~\cite{mpe19}.

In the sequel, we will denote by $B_X(x,r)$ (resp. $\ov{B}_X(x,r)$)
the open (resp. closed) metric ball centered at a point
$x\in X$ of radius $r>0$. For a Riemannian surface $M$ with smooth
compact boundary $\partial M$,
$\kappa(M)=\int_{\partial M}\kappa_g$ will stand for the 
total geodesic curvature of $\partial M$, where $\kappa _g$ denotes the
pointwise geodesic curvature of $\partial M$  with
respect to the inward pointing unit conormal vector of $M$ along $\partial M$.

\begin{definition}
\label{def:L}
{\em
For every $I\in \N\cup \{ 0\}$,   $\ve_0>0$, and $H_0,A_0,K_0\geq 0$, we denote by
\[
\L=\L(I, H_0,\ve_0,A_0,K_0)
\]
the space of all $H$-immersions $F\colon M\la X$
satisfying the following conditions:
\begin{enumerate}[({A}1)]
\item $X$ is a complete Riemannian 3-manifold with injectivity radius
$\Inj(X)\geq \ve_0$ and absolute sectional curvature bounded from above
by $K_0$.
\item $M$ is a complete surface with smooth boundary (possibly empty)
and when $\partial M\neq \varnothing$,
there are points in $M$ of distance greater than $\ve_0$ from
$\partial M$.
\item $H\in [0, H_0]$ and $F$ has index at most $I$.
\item  If $\partial M\neq \varnothing$, then for any
$\ve \in (0,\infty ]$ we let
$U(\partial M,\ve)=\{x\in M\mid d_M(x,\partial M)<\ve \}$
be the open intrinsic $\ve $-neighborhood of $\partial M$. Then,
$|A_M|$ is bounded from above by $A_0$ in $U(\partial M,\ve_0)$.
\end{enumerate}
}
\end{definition}

Suppose that $(F\colon M\la X)\in \L$ and $\partial M\neq \varnothing$.
For any positive $\ve_1 \leq \ve_2\in [0,\infty ]$, let
\[
U(\partial M,\ve_1,\ve_2 )=U(\partial M,\ve_2)\setminus \overline{U(\partial M,\ve_1)},\quad
\ov{U}(\partial M,\ve_1,\ve_2 )=\ov{U(\partial M,\ve_2)}\setminus U(\partial M,\ve_1).
\]
When $\partial M=\varnothing$, we  define $U(\partial M,\ve_1,\infty )=
\ov{U}(\partial M,\ve_1,\infty )$  as $M$.

In the next result we will make use of harmonic coordinates
$\varphi_x\colon U\to B_X(x,r)$ defined on an open subset $U$ of $\R^3$
containing the origin, taking values in a geodesic ball $B_X(x,r)$
centered at a point $x\in X$ of positive radius $r$ less
than the injectivity radius of $X$ at $x$
and with a $C^{1,\a}$ control of the ambient metric on $X$, see
Definition~\ref{defharm} for details.

\begin{theorem}[Structure Theorem for finite index $H$-surfaces]
\label{mainStructure}
Given $\ve_0>0$, $K_0, H_0, A_0\geq 0$, $I\in \N\cup \{0\}$ and
$\tau \in (0,\pi /10]$, there exist  $A_1\in [A_0,\infty)$,
$\de_1,\de\in (0,\ve_0/2]$ with $\de_1\leq \de/2$,
such that the following hold:
	
For any $(F\colon M\la  X)\in \L=\L(I, H_0,\ve_0,A_0,K_0)$,
there exists 
a (possibly empty) finite collection $\cP_F=\{p_1,\ldots,p_k\}\subset
U(\partial M,\ve_0 ,\infty )$ of points, $k\leq I$, and
numbers $r_F(1),\ldots ,{r_F}(k)\in [\de_1,\frac{\de}{2}]$
with $r_F(1)>4r_F(2)>\ldots >4^{k-1}r_F(k)$, satisfying the following:
\begin{enumerate}
\item
\label{it1}
\underline{Portions with concentrated curvature:}
Given $i=1,\ldots ,k$, let $\Delta_i$ be the component of
\newline
$F^{-1}(\ov{B}_X(F(p_i),r_F(i)))$ containing $p_i$.
Then:
\begin{enumerate}[a.]
\item $\Delta _i\subset \ov{B}_M(p_i,\frac{5}{4}r_F(i))$ (in particular,
$\Delta _i$ is compact).
\item $\Delta _i$ has smooth boundary
and $F(\partial \Delta_i)\subset \partial \ov{B}_X(F(p_i),r_F(i))$.
\item  $B_M(p_i,\frac{7}{5}r_F(i))\cap
B_M(p_j,\frac{7}{5}r_F(j))=\varnothing $ for $i\neq j$.
In particular, the intrinsic distance between $\Delta _i,\Delta _j$ is
greater than $\frac{3}{10}\de_1$ for every $i\neq j$.
\item $|A_M|(p_i)=\max_{\Delta_i}|A_M|= \max \{ |A_M|(p)\ :
\ p\in M\setminus \cup _{j=1}^{i-1}B_M(p_j,\frac54 r_F(j))\} \geq  A_1$,
see Figure~\ref{fig1}.
\begin{figure}
\begin{center}
\includegraphics[width=11cm]{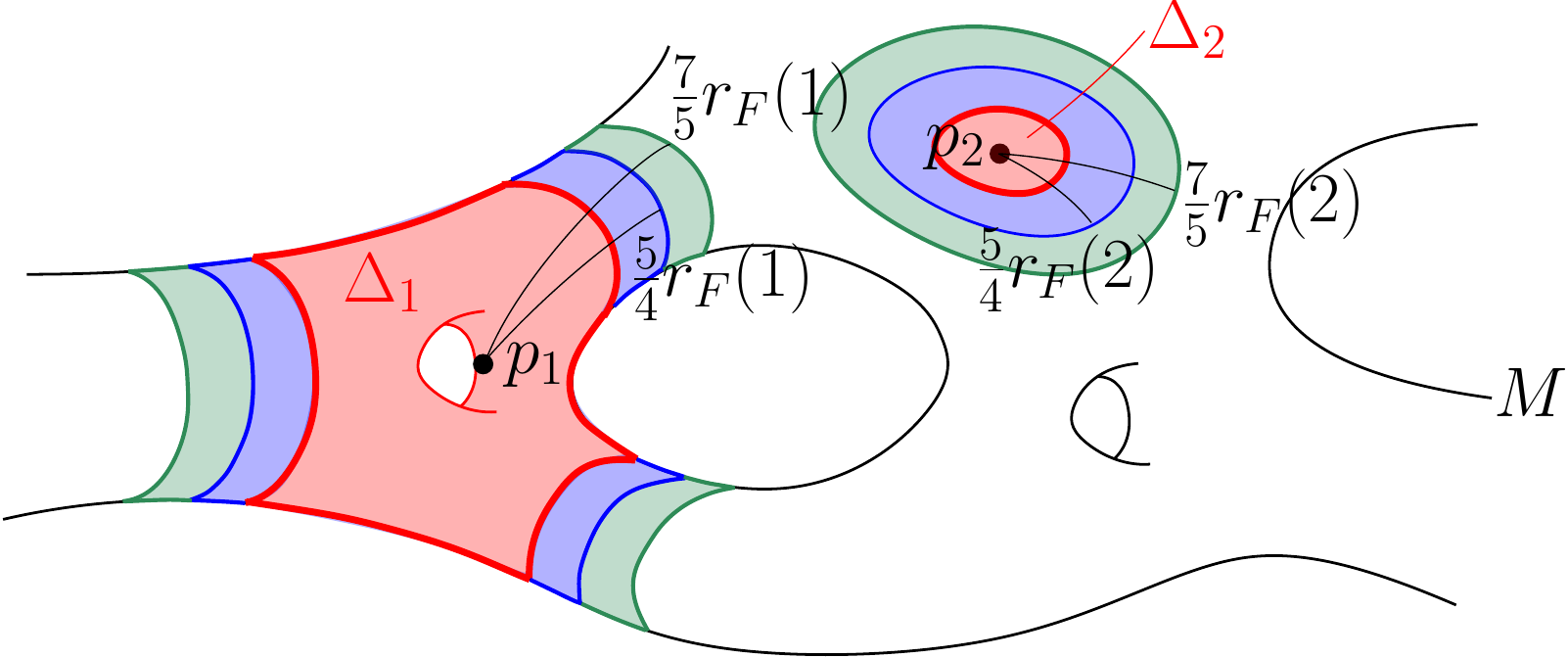}
\caption{The second fundamental form concentrates
inside the intrinsic compact regions $\Delta_i$ (in red),
each of which is mapped through the immersion $F$
to a surface inside the extrinsic ball in $X$ centered at
$F(p_i)$ of radius $r_F(i)>0$, with $F(\partial \Delta_i)\subset \partial
\ov{B}_X(F(p_i),r_F(i))$.
Although the boundary $\partial \Delta_i$ might not be at constant
intrinsic distance from the `center' $p_i$, $\Delta_i$ lies entirely
inside the intrinsic ball centered at $p_i$ of radius $\frac{5}{4}r_F(i)$.
The intrinsic open balls $B_M(p_i,\frac{7}{5}r_F(i))$ are pairwise
disjoint.}
\label{fig1}
\end{center}
\end{figure}
			
\item The index  $\Ind(\Delta_i)$ of $\Delta_i$ is positive.
\end{enumerate}
	
\item
\label{it3}
\underline{Transition annuli:}
For $i=1,\ldots ,k$ fixed, let $e(i)\in \N$ be the number of boundary
components of $\Delta _i$. Then, there exist planar disks $\D_{1},\ldots ,
\D_{e(i)}\subset T_{F(p_i)}X$ of radius $2r_F(i)$ centered at the origin
in $T_{F(p_i)}X$, such that if we denote by
\[
P_{i,h}=\varphi _{F(p_i)}(\D _h),\quad h\in \{ 1,\ldots ,e(i)\} ,
\]
(here $\varphi_{F(p_i)}$ denotes  a harmonic chart centered at $F(p_i)$,
see Definition~\ref{defharm}), then
\[
F(\Delta_i)\cap[\ov{B}_X(F(p_i),r_F(i))\setminus  B_X(F(p_i),r_F(i)/2)]
\]
consists of
$e(i)$ annular multi-graphs\footnote{See Definition~\ref{DefMulti} for
this notion of multi-graphs.} $G_{i,1},\ldots,G_{i,e(i)}$ over their
projections to $P_{i,1},\ldots ,P_{i,e(i)}$, with multiplicities
$m_{i,1},\ldots m_{i,e(i)}\in \N$ respectively, and whose related
graphing functions $u$ satisfy
\begin{equation}
\frac{|u(x)|}{|x|}+|\nabla u|(x)\leq \tau ,
\label{estimu}
\end{equation}
where we have taken coordinates $x$ in each of the $P_{i,h}$ and denoted
by $|x|$ the extrinsic distance to $F(p_i)$ in the ambient metric of $X$,
see Figure~\ref{fig2}.
\begin{figure}
\begin{center}
\includegraphics[width=10cm]{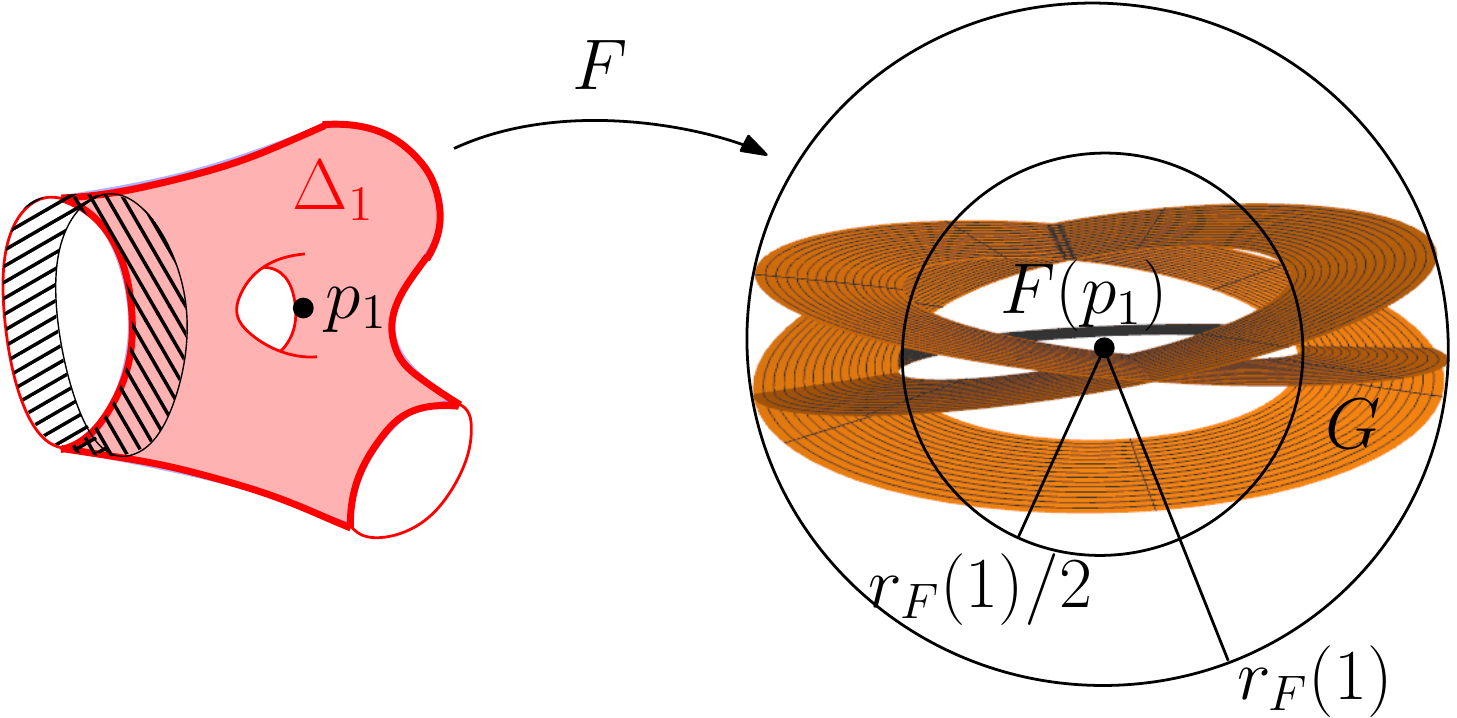}
\caption{The transition annuli: On the right, one has the extrinsic
representation in $X$ of one of the annular multi-graphs $G$ in
$F(\Delta_1)\cap[\ov{B}_X(F(p_1),r_F(1))\setminus  B_X(F(p_1),r_F(1)/2)]$;
in this case, the multiplicity of the multi-graph is 3.
On the left, one has the intrinsic representation of
the same annulus (shadowed); there is one such annular multi-graph for
each boundary component of $\Delta _i$.}
\label{fig2}
\end{center}
\end{figure}		
\item \label{it7}
\underline{Region with uniformly bounded curvature:} 
$|A_{M}|< A_1$  on $\wt{M}:=M\setminus \bigcup_{i=i}^k \Int(\Delta_i)$.
\end{enumerate}
	
\noindent
Moreover, the following additional properties hold:
\begin{enumerate}[A.]
\item  \label{it2}
$\sum_{i=1}^k I(\Delta_i)\leq I$, where $I(\Delta_i)=\Ind(\Delta_i)$.
		
\item \label{it4} \underline{Geometric and topological estimates:} Given
$i=1,\ldots ,k$, let $m(i):=\sum_{h=1}^{e(i)}m_{i,h}$ be the total
spinning of the boundary of $\Delta_i$, let $g(\Delta_i)$ denote the
genus of $\Delta_i$ (in the case $\Delta_i$ is non-orientable,
$g(\Delta_i)$ denotes the genus of its oriented
cover\footnote{If $\Sigma$ is a compact non-orientable surface
and	$\wh{\S}\stackrel{2:1}{\to}\S$ denotes the
oriented cover of $\S$, then the genus of $\wh{\S}$ plus 1 equals the
number of cross-caps in $\S$.}). Then, $m(i)\geq 2$ and the following
upper estimates hold:
\begin{enumerate}
\item If $I(\Delta_i)=1$, then $\Delta_i$ is
orientable, $g(\Delta_i)=0$, and $(e(i),m(i))
\in \{ (2,2),(1,3)\}$.

\item If   $\Delta_i$ is orientable and $I(\Delta_i)\geq 2$, then
$m(i)\leq 3I(\Delta_i)-1$,
$e(i)\leq 3I(\Delta_i)-2$,
and $g(\Delta_i)\leq 3I(\Delta_i)-4$.

\item If  $\Delta_i$ is non-orientable, then $I(\Delta_i)\geq 2$,
$m(i)\leq 3I(\Delta_i)-1$, $e(i)\leq 3I(\Delta_i)-2$  and
$g(\Delta_i)\leq 6I(\Delta_i)-8$.

\item
$\chi(\Delta_i)\geq -6I(\Delta_i) +2m(i)+e(i)$, and thus,
$\chi(\cup_{i=1}^k \Delta_i) \geq -6I +2S+e$, where
\[
e=\sum_{i=1}^ke(i),\qquad S=\sum_{i=1}^km(i).
\]

\item $|\kappa(\Delta_i)-2\pi m(i)|\leq \frac{\tau}{m(i)}$,
and so, the total geodesic curvature $\kappa(\widetilde{M})$
of $\widetilde{M}$ along $\partial \wt{M}\setminus \partial M$ satisfies
$\left| \kappa(\wt{M})+2\pi S\right| \leq
\frac{\tau}{2}k$, 
and so,
\begin{equation}
2\pi S-\frac{\tau}{2} k \leq\sum_{i=1}^k\kappa(\Delta_i)\leq
2\pi S+\frac{\tau}{2} k .
\label{2.3a}
\end{equation}
\item
$-\int _{\Delta_i}K >3\pi,$ and so,
\begin{equation} \label{2.4}
	-\int _{\cup_{i=1}^k\Delta_i}K
	=-2\pi\chi(\cup_{i=1}^k\Delta_i)+\int _{\cup_{i=1}^k\partial \Delta_i}\kappa_g>3k\pi.
\end{equation}
\end{enumerate}
\item
\label{it6}
\underline{Genus estimate outside the concentration of curvature:}
If $M$ is orientable, $k\geq1 $ and the genus $g(M)$ of $M$ is finite,
then the genus $g(\widetilde{M})$ of $\widetilde{M}$ satisfies
\[
0\leq g(M)-g(\wt{M})\leq 
3I-2. 
\]

\item
\label{it8}
\underline{Area estimate outside the concentration of curvature:}
If $k\geq 1$, then
\[
\mbox{\rm Area}(\widetilde{M})\geq 14\pi \sum_{i=1}^k m(i)r_F(i)^2
\geq 2\pi \sum_{i=1}^k m(i)r_F(i)^2
\geq\mbox{\rm Area}\left( \bigcup_{i=1}^k\Delta_i\right) \geq
k\pi \de_1^2.
\]

\item
\label{itF}
There exists a $C>0$,  depending   on $\ve_0,K_0,H_0$ and independent of
$I$, such that
\begin{equation}
\mbox{\rm Area}(M)\geq
\left\{
\begin{array}{ll}
{\displaystyle C\max\{1,\mbox{\rm Radius}(M)\} }&
\mbox{if $\partial M\neq \varnothing$,}
\\
{\displaystyle C\max\{1,\mbox{\rm Diameter}(M)\} } & \mbox{if $\partial
M=\varnothing$,}
\end{array}\right.
\label{Area}
\end{equation}
where
\begin{eqnarray}
\mbox{\rm Radius}(M)&=&\sup_{x\in M} d_M(x,\partial M)\in (0,\infty]
\quad \mbox{if }\partial M\neq \varnothing,\nonumber
\\
\mbox{\rm Diameter}(M)&=&\sup_{x,y\in M} d_M(x,y) \qquad \qquad \quad \;
\mbox{if }\partial M=\varnothing.\nonumber	
\end{eqnarray}
In particular, if $M$ has infinite radius
or if $M$ has empty boundary and it is non-compact, then
its area is infinite.
\end{enumerate}
\end{theorem}

The proof of the Structure Theorem~\ref{mainStructure} is carried out
in Section~\ref{sec3}, and it will be done by induction on the index
bound $I$. In the case $I=0$ , Theorem~\ref{mainStructure}
is obtained by using curvature estimates for stable $H$-surfaces,
and the arguments in this special case
generalize to the case where for a given $I\in \N$, there exists a uniform
curvature estimate for the immersions in $\L=\L(I, H_0,\ve_0,A_0,K_0)$,
see Section~\ref{subsecbdedA}.
A non-trivial step in the proof of Theorem~\ref{mainStructure}
involves an analysis of the local pictures on different scales for
a sequence of complete $H_n$-immersions $F_n\colon M_n\la X_n$ with
$H_n\in [0,H_0]$ and $\mbox{Index}({F_n})\leq I$,
such that $\{ \sup |A_{M_n}|\}_n$ is unbounded
(these local pictures are limits of the $F_n$ after blowing up on
certain scales).
Although non-trivial, this analysis is simpler in the case $I=1$
because in this case there is only one scale for the local pictures of
$F_n$; this case is studied in Section~\ref{sec2.6}.
The analysis of these local pictures
in the general case is carried out in Section~\ref{sec2.7},
and it is based on the lower bounds obtained by Chodosh
and Maximo~\cite{ChMa2} and Karpukhin~\cite{Kar1} for the
index of a possibly branched, complete
immersed minimal surface $\S$ in $\rth$ with finite total
curvature, in terms of its genus, total branching order,
and the number of its ends counted with multiplicity.
After coming back to the original scale, these complexity estimates will
give upper bounds for the total geodesic curvature of the boundary of the
portion $\wt{M}$ of $M$ defined in item~\ref{it7} of
Theorem~\ref{mainStructure}, as well as to give lower bounds in
item~\ref{it6} on the genus of $\wt{M}$ in terms of $I$ and the genus of
$M$ when $M$ is orientable. These geometric and topological bounds are
obtained in Sections~\ref{sec9.1} and~\ref{sec2.9}.
What this analysis demonstrates is
that there is an organized hierarchy-type structure in the geometry of a
complete, immersed $H$-surface $F\colon M\la X$ near points of large,
almost-maximal norm of the second fundamental form of the immersion,
from which the title of the paper is derived;
this hierarchy structure of $F$ around such special points
is described explicitly  in Section~\ref{sec9.1} and plays
an essential role in the proofs of our main results.

A key step in the proof of Theorem~\ref{mainStructure} is to obtain
curvature estimates for a large portion of the $H$-surface
$(F\colon M\la X)\in \Lambda$ in that theorem.
These curvature estimates are obtained in Section~\ref{sec5.2}
and they are based on related curvature estimates for stable
$H$-surfaces developed in Appendix A
(Section~\ref{sec:curvatureEst}).

Observe that \eqref{Area} is a lower bound for the area of an
$H$-surface in a Riemannian 3-manifold $X$, described
in terms of an upper bound for its absolute mean curvature function $|H|$,
a lower bound of the injectivity radius of $X$ and an upper bound of the
sectional curvature of $X$. This area estimate is proven in
Section~\ref{sec2.9} and follows from a more general area estimate
and an intrinsic monotonicity of volume formula
for $n$-dimensional submanifolds with bounded length of their mean
curvature vectors in $m$-dimensional Riemannian manifolds $X$
that have a lower bound for their injectivity radius and
an upper bound for the sectional curvature of $X$. Both of these
auxiliary results are proven in our paper~\cite{mpe20},
and we include their statements (without proofs) in this paper
for the sake of completeness, see Propositions~\ref{lemma8.4}
and~\ref{yau} in Appendix~(B) (Section~\ref{sec:summary}).
In Proposition~\ref{propos5.5} of the same appendix we state explicit
scale invariant weak chord-arc estimates for finitely branched minimal
surfaces of finite total curvature in $\rth$ in terms
of the index and total branching (also proven in~\cite{mpe20}); these
chord-arc estimates are applied in the proof of item~\ref{it1}
in Theorem~\ref{mainStructure}.

An important theoretical consequence of Structure Theorem~\ref{mainStructure}
is the existence of compactness results for $H$-surfaces of
bounded index in $X$. More specifically, in Section~\ref{sectionCompactness}
we state and prove some compactness
results for sequences of complete immersions with
constant mean curvature in 
$\Lambda=\Lambda (I,H_0,\ve_0,A_0,K_0)$, as described  in
Theorem~\ref{mainStructure}, in the particular case
that the immersions are defined on connected surfaces without boundary,
the ambient space $X$ is independent
of the element in the sequence and the image of each
immersion in the sequence intersects a fixed compact subdomain of $X$.
In this case, the limit object that we encounter
(after passing to a subsequence) is a complete, possibly finitely branched
immersion of constant mean curvature at most $H_0$ and index at most $I$.

In regards to the just mentioned
compactness results in Section~\ref{sectionCompactness},
it is worth mentioning the related paper~\cite{bst1} by Bourni, Sharp and Tinaglia,
where they give weak compactness results for a
sequence of embedded CMC hypersurfaces
in a compact Riemannian  manifold of dimension $m$ with $3\leq m\leq 7$, provided that their
areas and Morse indices are bounded. As they remark in~\cite{bst1}, their results
were motivated by the derivation of the genus-dependent 
area bounds  for triply periodic CMC surfaces
properly embedded in $\rth$  by Meeks and Tinaglia in \cite{mt12}; also the
results in~\cite{bst1} and in our present paper are motivated
other recent works \cite{AiHo1,abcs1,bst1,bush1,ChKeMa1,max1,sat1}, which
together help to describe
the geometry of finite index CMC surfaces $M$
embedded in closed Riemannian 3-manifolds
and  relationships between index, area and genus of such an $M$.

In \cite{mpe19} we give applications of
Theorem~\ref{mainStructure} to understand global properties of
immersed $H$-surfaces $M\la X$ of fixed finite index $I$,
especially results related to the area and diameter of such an $M$
when it is compact without boundary; in particular,
we deduce that the area of such an $M$ (resp. the diameter)
grows at least linearly (resp. logarithmically) with the genus.

{\em Acknowledgments}:
We thank Otis Chodosh for explaining to us his work~\cite{ChMa2}
with Davi Maximo on lower bounds for the index of a complete branched
minimal surface in $\rth$ of finite total curvature,
in terms of its genus and number of ends counted with multiplicity,
and how the analysis by Karpukhin~\cite{Kar1} can be
used to allow finitely many branch points.

\section{Index of $1$-sided $H$-immersions, harmonic coordinates and multi-valued graphs}
\label{Sec:harm}

In Theorem~\ref{mainStructure} we referred to the index of 1-sided
minimal immersions, harmonic coordinates and
finitely valued multi-graphs. We will devote this section to give some
details about these notions.
\begin{definition} \label{DefIndexNO}
{\rm
Given a $1$-sided minimal codimension-one immersion $F\colon M\la X$ in a
Riemannian manifold $X$, let $\wt{M}\to M$ be the two-sided cover of
$M$ and let $\tau \colon \wt{M}\to \wt{M}$ be the associated deck
transformation of order 2. Denote by $\wt{\Delta}$,
$|\wt{A}|^2$ the Laplacian and squared norm of the second fundamental
form of $\wt{M}$, and let $N\colon \wt{M}\to TX$
be a unitary normal vector  field.  The index of $F$ is defined as the
number of negative eigenvalues of the elliptic, self-adjoint operator
$\wt{\Delta} +|\wt{A}|^2+\mbox{Ric}(N,N)$ defined over the space of
compactly supported smooth functions $\phi \colon \wt{M}\to \R$
such that $\phi \circ \tau =-\phi $.
}
\end{definition}

\begin{definition} \label{defharm}
{\rm
Given a (smooth) Riemannian manifold $X$, a local chart $(x_1,\ldots x_n)$
defined on an open set $U$ of $X$
is called {\it harmonic} if $\Delta x_i=0$ for all $i=1,\ldots n$.
		
Following Definition 5 in~\cite{hehe}, we make the next definition.
Given $Q>1$ and $\alpha \in (0,1)$, we define the
$C^{1,\alpha }$-{\it harmonic radius} at a point $x_0\in X$
as the largest number $r=r(Q,\a)(x_0)$ so that in the geodesic ball
$B_X(x_0,r)$ of center $x_0$ and radius $r$,
there is a harmonic coordinate chart
such that the metric tensor $g$ of $X$ is $C^{1,\a}$-controlled in these
coordinates. Namely, if $g_{ij}$, $i,j=1,\ldots ,n$, are
the components of $g$ in these coordinates, then
\begin{enumerate}
\item $Q^{-1}\de _{ij}\leq g_{ij}\leq Q\, \de_{ij}$ as bilinear forms, and
\item ${\displaystyle \sum_{\be =1}^3 r \sup _{y}
|\frac{\partial g_{ij}}{\partial x_{\be}}(y)|
+\sum _{\be =1}^3r^{1+\a}\sup _{y\neq z}\frac{\left|
\frac{\partial g_{ij}}{\partial x_{\be}}(y)-
\frac{\partial g_{ij}}{\partial x_{\be}}(z)\right| }{d_X(y,z)^{\a}}}
\leq Q-1$.
\end{enumerate}
The $C^{1,\a}$-harmonic radius $r(Q,\a )(X)$ of $X$ is now defined by
\[
r(Q,\a)(X)=\inf_{x_0\in X}r(Q,\a)(x_0).
\]
If the absolute sectional curvature of $X$ is bounded by some
constant $K_0>0$ and Inj$(X)\geq \ve_0>0$, then Theorem~6 in~\cite{hehe}
implies that given $Q>1$ and $\a \in (0,1)$,
there exists $C=C(Q,\a ,\ve_0,K_0)$ (observe that $C$ does not depend
on $X$) such that $r(Q,\a)(X)\geq C$.
	}
\end{definition}

\begin{definition} \label{DefMulti}
{\rm Let $f\colon\S\la \R^3$ be an immersed annulus,  $P$
a plane passing through the origin and
$\Pi\colon \rth \to P$ the orthogonal projection.
Given $m\in \N$, let $\sigma_m\colon P_m \to P^*=P\setminus \{\vec{0}\}$
be the $m$-sheeted covering space of $P^*$.
We say that $\S$ is an $m$-{\it valued graph}
over $P$ if $\vec{0}\not \in (\Pi\circ f)(\S)$,
the induced map $(\Pi\circ f)_*\colon H_1(\S)=\Z \to H_1(P^*)=\Z$
satisfies $|(\Pi\circ f)_*(1)|=m$ and $\Pi\circ f\colon \S \to P^*$
has a smooth injective lift $\wt{f} \colon  \S \to P_m$
through $\sigma_m$; in this case, we say that $\Sigma $
has {\it multiplicity $m$} as a multi-graph.

Given $Q>1$ and $\a \in (0,1)$, let $X$ be a Riemannian 3-manifold
and $(x_1,x_2,x_3)$ be a harmonic chart for $X$ defined on $B_X(x_0,r)$,
$x_0\in X$, $r>0$, where the metric tensor $g$ of $X$ is
$C^{1,\a}$-controlled in the sense of Definition~\ref{defharm}.
Let $P\subset B_X(x_0,r)$ be the image by this harmonic chart of the
intersection of a plane in $\R^3$ passing through the origin with the
domain of the chart. In this setting, the notion of $m$-valued graph
over $P$ generalizes naturally to an immersed annulus
\[
f\colon\S\la B_X(x_0,r),
\]
where the projection $\Pi$ refers to the harmonic coordinates.
If $f\colon\S\la B_X(x_0,r)$ is an $m$-valued graph over $P$
and $u$ is the corresponding graphing function that
expresses $f(\Sigma)$, we can consider the gradient $\nabla u$ with
respect to the metric on $P$ induced by the ambient metric of $X$.
Both $u$ and $|\nabla u|$ depend on the choice of harmonic coordinates
around $x_0$ (and they also depend on $Q$),
but if $\frac{|u(x)|}{|x|}+|\nabla u|(x)\leq \tau $ for some
$\tau \in (0,\pi /10]$ and $Q>1$ sufficiently close to $1$, then
$\frac{|u(x)|}{|x|}+|\nabla u|(x)<2\tau$ for any other choice of harmonic
chart around $x_0$ with this restriction of $Q$.
}
\end{definition}

\section{Finitely branched minimal surfaces in $\R^3$ of finite index}
\label{secMinBr}
In the process of finding local pictures of $H$-immersions as
in  the proof of Theorem~\ref{mainStructure}, we will find complete, non-flat,
finitely branched minimal surfaces in $\R^3$. We will devote this
section to obtain some properties of these surfaces which will be useful
in the sequel.

\begin{definition}
{\rm
Let $\Sigma$ be a smooth surface endowed with a conformal class of
metrics. We say that a harmonic map $f\colon \Sigma\to \R^3$
is a (possibly non-orientable) branched minimal surface if it
is a conformal immersion outside of a locally finite
set of points $\cB_{\Sigma}\subset \Sigma$,
where $f$ fails to be an immersion.
Points in $\cB_{\Sigma}$ are called branch points of $f$.
It is well-known (see e.g. Micallef and White~\cite[Theorem~1.4]{miwh1})
that given $p\in \cB_{\Sigma}$, there exist a conformal coordinate
$(\ov{\D},z)$ for $\Sigma$ centered at $p$ (here $\ov{\D}$
is the closed unit disk in the plane),
a diffeomorphism $u$ of $\ov{\D}$ and a rotation
$\phi$ of $\R^3$ such that $\phi \circ f\circ u$ has the form
\[
z\mapsto (z^q,x(z))\in \C\times \R\sim \R^3
\]
for $z$ near $0$, where $q\in \N$, $q\geq 2$, $x$
is of class $C^2$, and $x(z)=o(|z|^q)$.
The branching order $B(p)\in \N$
is defined to be $q-1$. The total branching
order of $f$ is
\[
B(\Sigma):=\sum_{p\in \cB_{\Sigma}}B(p).
\]
}
\end{definition}

The next result is a generalization of the well-known
Jorge-Meeks formula~\cite{jm1} to the case of a possibly branched
and non-orientable complete minimal surface $\S\la \R^3$
of finite total curvature and finite branching order.
It is well-known that each of the (finitely many) ends of $\S$ is
a multi-graph of finite multiplicity over the exterior
of a disk in the plane passing through the
origin and orthogonal to the extended
value of the unoriented Gauss map of $\S$.
We will use the language {\it the total spinning of} $\S$ to describe the
sum of these multiplicities; for instance,
the classical Henneberg  and Enneper surfaces each have one end and
total spinning equal to three.
\begin{proposition}
\label{lem6.3}
Let $\S\la \R^3$ be a complete, finitely connected and finitely
branched minimal surface with finite
total curvature, $e$ ends with total spinning $S$, and
total branching order $B(\Sigma)$. Then:
\begin{equation}
\label{lem6.3a}
\frac{1}{2\pi }\int_{\S}K+S-B(\Sigma)=\chi (\ov{\S})-e=\chi (\S),
\end{equation}
where $K\colon \S\setminus \mathcal{B}_{\S}\to (-\infty,0]$
is the Gaussian curvature function and $\ov{\S}$
denotes the conformal\footnote{Observe that
$\S$ admits an atlas whose changes of coordinates are
conformal or anti-conformal.}
compactification of $\S$. Furthermore, if $G\colon
\ov{\S}\to \Pe^2$ denotes the extended unoriented Gauss map
of $\S$, then the degree of $G$ satisfies
\begin{equation}
\label{lem6.3b}
\deg(G)=\frac{1}{2\pi }\int_{\S}K\equiv \chi (\ov{\S})	\mod 2.	
\end{equation}
In particular,
\begin{equation} \label{lem6.3c}
S-B(\Sigma)\equiv e \mod 2.
\end{equation}
\end{proposition}
\begin{proof}
To prove each of the statements in
the above proposition  it suffices to consider the special case that
$\S$ is connected, which we will assume holds for the remainder
of the proof.
	
We first prove~\eqref{lem6.3b}. Note that the
total curvature $\int_{\S}K$ equals $2\pi
\deg(G)$. First consider the case that $\deg (G)\neq 0$.
By Theorem~1 in~\cite{me7}, $\deg (G)\equiv \chi
(\ov{\S}) \mod 2$, which proves that \eqref{lem6.3b} holds in this case.
If the degree of the Gauss map is zero, then the image
of the branched immersion is a flat plane, and we can view
$\ov{\S}$ as a	connected, finitely branched covering of
the sphere. Hence, $\ov{\S}$ is orientable with even Euler characteristic.
Thus, \eqref{lem6.3b} holds in all cases.
	
Using Gauss-Bonnet in the compact portion
of $\S$ obtained by removing pairwise disjoint disks around
its ends (viewed as points in $\ov{\S}$)
and the branch points of $\S$, and taking the radii
of the removed disks going to zero, we obtain
equation~\eqref{lem6.3a}. Taking classes mod 2
in~\eqref{lem6.3a} and using (\ref{lem6.3b}), we obtain (\ref{lem6.3c}).
\end{proof}

We next recall a fundamental lower bound for the index
$I(f)$ of a connected, complete, possibly finitely branched
minimal surface $f\colon \S \la \R^3$ with  finite total
curvature, which is due to Chodosh and Maximo~\cite{ChMa2}
and to Karpukhin~\cite{Kar1}:
\begin{equation}
\label{eq:CMindex1}
3I(f)\geq \left\{ \begin{array}{ll}
{\displaystyle 2g(\S)+2\sum_{j=1}^{e}(d_j+1)-2B-5}
& \mbox{if $\S$ is orientable,}
\\
{\displaystyle g(\wt{\S})+2\sum_{j=1}^{e}(d_j+1)-2B-4}
& \mbox{if $\S$ is non-orientable,}
\end{array}\right.
\end{equation}
where $g(\Sigma)$ is the genus of $\S$ if $\S $ is
orientable (resp. $g(\wt{\S})$ is the genus of
the orientable cover $\wt{\S}$ of $\S$ if $\S $ is not
orientable), $e$ and $B$ are respectively
the number of ends and the total branching order of $\Sigma$,
and for each end $E_j$ of $\S$, $d_j$ is the multiplicity
of $E_j$ as a multi-graph over the limiting tangent plane of $E_j$.

Inequality \eqref{eq:CMindex1} has not been explicitly stated
in the literature, so an explanation is in order.
Ros~\cite{ros9} proved that $3I(f)\geq 2g(\S)$
using harmonic square integrable $1$-forms on $\S$ for a minimal immersion
$f\colon \S \la \R^3$ with  finite total curvature,
in order to produce test functions for the index operator of $f$.
Chodosh and Maximo~\cite[Theorem 1]{ChMa2}
improved Ros' technique with an enlarged space of harmonic $1$-forms which
admit certain singularities at the ends of $\S$ that take care of
the spinning (multiplicity) of each end of such an immersion $f$,
obtaining a simplified version of~\eqref{eq:CMindex1}
without the term $-2B$. Finally,
Karpukhin~\cite[Proposition 2.3 and Remark 2.4]{Kar1}
included the study of branch points although he made use
of the original space of $L^2(\S)$
harmonic $1$-forms considered by Ros. Formula~\eqref{eq:CMindex1}
is the combined inequality that one can deduce
from~\cite{ChMa2} and~\cite{Kar1}.

\begin{remark}\label{rem3.3}
{\rm
\begin{enumerate}
\item If $\Sigma$ is orientable and the index of $f$ is even,
then all summands in \eqref{eq:CMindex1} except for the $-5$
in the RHS are even. Therefore, the inequality still holds
after adding $1$ to the RHS of~\eqref{eq:CMindex1}.

\item Inequality~\eqref{eq:CMindex1} can be expressed in
a unified way regardless of the orientability character of $\Sigma$,
if we use the Euler characteristic.
Recall that if $\S$ is orientable, then its Euler characteristic is
$\chi=\chi(\S)=2-2g(\S)-e$, while if $\S$ is non-orientable,
the Euler characteristic of its orientable cover is
$\chi(\wt{\S})=2-2g(\wt{\S})-2e$ where $e$ is the number of
ends of $\S$, and so, $\chi=\chi(\S)=1-g(\S)-e$.
Thus, \eqref{eq:CMindex1} reduces to
\begin{equation}
\label{eq:CM1}
3I(f)\geq -\chi +2S+e-2B-3,
\end{equation}
where $S=\sum_{j=1}^{r}d_j$ is the total spinning of
the ends of $f$ (sometimes we will refer to $S$ as the
{\it total spinning} of $f$).
\end{enumerate}
}
\end{remark}

\begin{lemma}
\label{lema3.6}
Let $f\colon \Sigma \la \R^3$ be a complete, connected,
non-flat, finitely branched minimal surface with branch point set
$\cB_{\Sigma}\subset \Sigma$. Then:
\begin{enumerate}
\item If $f$ is stable,  then $\Sigma $ is non-orientable and
$f(\cB_\S)$ contains more than 1 point.

\item If $\S$  non-orientable
with $f(\cB_{\S})$ consisting of at most one point in $\rth$,
then $I(f)\geq 2$;
in particular, if $\S$ has exactly one branch point, then $I(f)\geq 2$.
\end{enumerate}
\end{lemma}
\begin{proof}
Assume that $f\colon \Sigma \la \R^3$ is stable.
Also, suppose for the moment that $\Sigma $ is orientable. Let $g\colon
\ov{\S}\to \esf^2$ be the Gauss map extended to the
conformal compactification $\ov{\S}=\S \cup \mathcal{E}$ of
$\S$ obtained after adding the set $\mathcal{E}$ of its ends.
Let $\mathcal{C}\subset \ov{\S}$ be the set of branch points of $g$.
Let $\Omega(\varepsilon) \subset \esf^2$ be the complement of
the union of a pairwise disjoint collection of open
$\varepsilon$-disks around the points in the finite set
$g(\mathcal{E}\cup \mathcal{B}_\S\cup \mathcal{C})$.
For $\varepsilon >0$ sufficiently small, the
Schr\"{o}dinger operator $\Delta +2$ has negative first
Dirichlet eigenvalue on $\bf \Omega(\varepsilon)$, where $\Delta$ is
the spherical Laplacian. Since $g|_{g^{-1}(\Omega (\ve))}
\colon g^{-1}(\Omega (\ve))\to \Omega (\ve)$ is a finite
covering, then each component of $g^{-1}(\Omega (\ve))$
is a smooth, compact unstable domain. This contradicts that
$f$ is stable, which proves that $\S$ is non-orientable.
	
Since $\Sigma $ is non-orientable and $f$ is stable, then the main result
in Ros~\cite{ros9} implies that $\mathcal{B}_\S\neq \varnothing$.
To finish the proof of item~1, suppose that $f(\mathcal{B}_\S)$ is a
single point in	$\R^3$ (say the origin) and we will find a
contradiction. The area density of $\Sigma$ at the origin is at
least $B(\S)+l$, where $B(\S)$ is the total branching
order of $f$ and $l$ is the
cardinality of $\mathcal{B}_\S$. Using the monotonicity
formula for minimal surfaces, the total spinning
$S$ of the ends of $f$ is at least $B(\S)+l+1$. Using
(\ref{eq:CMindex1}), since $g(\wt{\S})\geq 0$ and $e\geq 1$,
we have
\begin{align}
3I(f)&\geq g(\wt{\S})+2\sum_{j=1}^{e}(d_j+1)-2B(\S)-4
\nonumber
\\
&\geq 2S+2e-2B(\S)-4
\geq 2S-2B(\S)-2\geq 2l>0,
\label{countingI1}
\end{align}
which contradicts that $\S$ is stable and proves item~1 of
the lemma.
	
To prove item~2, assume that $\S$ is non-orientable and
$f(\cB_\S)$ contains  at most
one  point. If  $f$ is unbranched, then, by Theorem~1.8
in~\cite{ChMa2}, the index of $f$ is at least 2. So assume
that $\mathcal{B}_\S \neq \varnothing$.  If $I(f)=1$, then the calculation
in~\eqref{countingI1} implies $l=e=1$ and the
total spinning $S$ of
the ends of $f$ is $B(\S)+l+1=B(\S)+2$. But this
implies that $S-B(\S)$ is even and $e$ is odd, which
contradicts the last statement of Proposition~\ref{lem6.3}.
Hence, by item~1 of the lemma,  $I(f)\geq 2$.
\end{proof}

\section{Almost flat annular $H$-multi-graphs of bounded multiplicity}

For the next lemma we will need the following notation. For $0<R_1<R_2$, we let
\[
\A (R_1,R_2)=
\{ x\in  \R^3\ | \  R_1\leq |x|\leq R_2\} .
\]
Observe that the statement of the
next Lemma~\ref{lemma5.3} is invariant under homotheties
centered at the origin.
\begin{lemma}
\label{lemma5.3}
Given $\tau \in (0,\pi /10]$ and $L_0>0$, there exists $\a_1\in (0,\tau]$
such that the following property holds. Take $\a \in (0,\a_1]$,
$0<R_1\leq R_2/2$, and a compact immersed annulus
$\Sigma \subset \A (R_1,R_2)$ with $\partial \Sigma \subset
\partial \A (R_1,R_2)$, satisfying the following conditions:
\begin{enumerate}[(B1)]
\item $\Sigma $ makes an angle  greater than or equal to
$\frac{\pi }{2}-\a$ with every sphere $\esf^2(r)$ of radius
$r\in [R_1,R_2]$ centered at the origin.

\item Given $R\in [R_1,R_2/2]$, the image of $\S\cap \A (R,2R)$ through
the Gauss map of $\Sigma$ is contained in the closed spherical
neighborhood of radius $\a$ centered at some point
$v(R)\in \esf^2(1)$.

\item  $\mbox{\rm Length}(\S\cap \esf^2(R_1))<L_0R_1$.
\end{enumerate}
Then there exist $m\in \N$, $m\leq \frac{L_0+1}{2\pi }$, such that for
any $R\in [R_1,R_2/2]$, $\Sigma \cap \A (R,2R)$ consists of
an $m$-valued graph with respect to its projection to the plane
$v(R)^{\perp}$ orthogonal to $v(R)$, of a function $u$ that satisfies
\[
\frac{|u(x)|}{|x|}+|\nabla u|(x)<\frac{\tau}{2}
\]
at every point $x$ in its domain of definition.
Furthermore, for each $R\in [R_1,R_2]$ the following properties hold:
\begin{enumerate}[(C1)]
\item $|\mbox{\rm Length}(\S\cap \esf^2(R))-2\pi mR|<f_1(\a) R$,
where $f_1=f_1(\a)\in (0,\tau]$ is a function that tends to zero
as $\a\to 0$.

\item The intrinsic distance between the two boundary components
of $\S\cap \A (R_1,R)$ is at most $\sqrt{1+\tau^2/4}\ (R-R_1)$.

\item $|\mbox{\rm Area}(\S\cap \A(R_1,R))-\pi m(R^2-R_1^2)|<
f_2(\a) (R-R_1)$, where $f_2=f_2(\a)\in (0,\tau]$
is a function that tends to zero as $\a\to 0$.
\end{enumerate}
\end{lemma}
\begin{proof}
The first step in the proof consists of showing that
for $\tau \in (0,\pi/10]$ and $L_0>0$ given,
items (C1) and (C2) hold in the range $R\in [R_1,2R_1]$
for some choice of $m\in \N$, $m\leq \frac{L_0+1}{2\pi }$
depending on a compact immersed annulus $\Sigma$ satisfying
(B1), (B2), (B3) provided that $0<\a \leq \a_1$ and $\a_1$ is
sufficiently small.

Observe that if $0<\a \leq \a_1<\pi /4$, condition (B2) above for $R=R_1$
implies that $\S\cap \A (R_1,2R_1)$ is a multi-graph with respect
to its projection to the plane $v(R_1)^{\perp}$.  We call
$u$ to the related graphing function, and $m\in \N$ to
its multiplicity.
Taking $\a_1$ sufficiently small, (B2) guarantees that $|\nabla u|$
can be made arbitrarily small. By condition (B1),
the almost orthogonality of $\S$ with spheres
$\esf^2(R)$ with $R\in [R_1,2R_1]$
implies that if $\a_1$ is sufficiently small, we have that
$\frac{|u(x)|}{|x|}$ can also be made arbitrarily small. In particular,
we have that $\frac{|u(x)|}{|x|}+|\nabla u|(x)<\frac{\tau}{2}$ in
$\S \cap \A (R_1,2R_1)$ if $\a_1$ is sufficiently small
in terms of $\tau$.
Similar arguments show that the length of $\S\cap \esf^2(R_1)$ differs
from $2\pi mR_1$ by a function of $\a$ that tends to zero as $\a \to 0$
(in particular, $2\pi m\leq L_0+1$ provided that $\a_1$ is sufficiently
small). Thus (C1) holds in $[R_1,2R_1]$
for some function $f_1(\a)$ that tends to zero as $\a \to 0$.

Regarding the validity of (C2) in the range $[R_1,2R_1]$,
given $R\in [R_1,2R_1]$ and given a point $x\in \S\cap \esf^2(R_1)$,
let $\Pi_x\subset \R^3$ be the plane
passing through the origin that contains
both $v(R_1)$ and $x$; without loss of generality,  assume
$\Pi_x$ is the $(x_1,x_3)$-plane and $v(R_1)=(0,0,1)$.
Let $\G$ be the component of
$\S\cap \A(R_1,R)\cap \Pi_x$ that passes through $x$ and note that
$\G$ is a smooth embedded arc that can be parameterized
using polar coordinates in $\Pi_x$ by $\G(r)=(r,\theta(r))$,
$r\in [R_1,R]$.
Next assume that $\a_1$ is chosen less than or equal to $\arcsin(\tau/2)$
and we will prove that (C2) holds.
Property (B1) implies that the angle between $\G'(r)$
and the radial outward pointing unit vector field $\partial_r$ is at most
$\a_1$, which implies $|\G'(r)|\leq \sqrt{1+\sin^2(\a_1)}\leq
\sqrt{1+\tau^2/4}$.  Therefore, (C2) holds in $[R_1,2R_1]$.

The second step in the proof consists of demonstrating that (C1) and (C2)
hold for every $R\in [R_1,R_2]$. To see this, it suffices to
iterate a finite number of times the above arguments
replacing $R_1$ by $2R_1$, $4R_1,\ldots , 2^kR_1$,
where $k\in \N$ is the first positive integer such
that $2^kR_1>R_2/2$. Then we conclude that
(C1) and (C2) hold in $[R_1,R_2/2]$, and by iterating
once again replacing $R_1$ by $R_2/2$ we get
that (C1) and (C2) hold in $[R_1,R_2]$.

Finally, (C3) holds for every $R\in [R_1,R_2]$ by (C1) and the co-area
formula.
\end{proof}

\begin{remark}
\label{rem5.4}
{\rm
Since the statement of Lemma~\ref{lemma5.3} is invariant under
re-scalings of the ambient metric, we conclude that
last lemma holds if we replace the ambient space $\R^3$ by a
sufficiently small closed geodesic ball
$\ov{B}_{X}(x,R_2)$ centered at any point $x\in X$ of radius
$R_2\in (0,\ve_0/2)$
(using harmonic coordinates, see Definition~\ref{defharm}
and recall that $\ve_0>0$ is a lower bound for Inj$(X)$)
in the Riemannian 3-manifold $X$, with the following changes:
\begin{enumerate}[(D1)]
\item We replace the notion of Gauss map in hypothesis (B2)
of Lemma~\ref{lemma5.3} by parallel translation
of the unit normal vector to $\Sigma$
at a point $q\in \S\cap \ov{B}_{X}(x,R_2)$ along the corresponding
radial geodesic arc joining the point $x$ to $q$.
\item We replace the upper bound in conclusion (C2) of
Lemma~\ref{lemma5.3}
by $\sqrt{1+\tau^2/3}$ times the extrinsic distance in $X$ between
the two boundary components of $ [\ov{B}_{X}(x,R_2)
\setminus B_{X}(x,R_1)]$ (here $0<R_1\leq R_2/2$).
\end{enumerate}
}
\end{remark}

\begin{definition}
\label{def2.5}
{\rm
Fix $\tau \in (0,\pi /10]$. Let $\de_2\in (0,\ve_0]$  be such
that Remark~\ref{rem5.4} holds for any choice of extrinsic
radii $R_1,R_2$ with $0<R_1\leq R_2/2<R_2\leq \de_2$
in  $X$. Fix such $R_1,R_2$, choose  $\tau_1\in (0,\tau]$,
and let $m\in \N$ be an  integer to be fixed later. Given
$x\in X$, we consider the collection
\[
{\mathcal G}(x;R_1,R_2,\tau_1,m)
\]
of multi-graphical (immersed) $H$-annuli
$G\subset \ov{B}_X(x,R_2)\setminus B_X(x,R_1)$  (here $H\in [0,H_0]$)
with multiplicity $m(G)\leq m$, such that $G$ is ``almost flat''
in terms of $\tau_1$, in the sense
that $G$ satisfies the following properties:
\begin{enumerate}[(E1)]
\item $G$ is an immersed $H$-annulus in $X$, whose boundary
$\partial G\subset \partial B_X(x,R_1)\cup \partial B_X(x,R_2)$ consists
of two closed curves, one on each ambient geodesic sphere, and $G$
is the graph over its projection  to a
``planar'' disk $P=\varphi_x(\D_{2R_2})$ (this map $\varphi_x$
gives harmonic coordinates
around $x$), where $\D_{2R_2}\subset T_xX$ is a planar disk of
radius $2R_2$ centered at the origin in $T_xX$,
of a function $u$ defined on a domain $\Omega$ of the
$m(G)$-sheeted cover of the annulus $\D_{2R_2}\setminus\{ 0\} $.
\item Given $y$ in $P$, denote by $|y|$ the distance to the
point $x$ in the ambient metric of $X$.
Then, the graphing function that defines $G$ satisfies
\[
\frac{|u(y)|}{|y|}+|\nabla u|(y)\leq \tau_1 \quad \mbox{in $\Omega $}.
\]
\end{enumerate}
}
\end{definition}

\begin{lemma}
\label{ass4.4'}
In the situation of Definition~\ref{def2.5},
there exist $\de_3\in (0,\de_2]$ and $\tau_1\in (0,\tau]$
such that for every $r\in (0,\de_3]$ and
$G\in \mathcal{G}(x;r/2,r,\tau_1,m)$,
the geodesic curvature function of
$G\cap B_X(x,\frac{3}{4}r)$ along its intersection with
$\partial B_X(x,\frac{3}{4}r)$
is everywhere positive and its integral, the total geodesic curvature
$\kappa (G)$, satisfies
\begin{equation}
\label{eq:aa}
\left| \kappa (G)-2\pi m(G)\right| \leq \frac{\tau}{m}.
\end{equation}
Furthermore, every such graph $G$ is stable.
\end{lemma}
\begin{proof}
Suppose that $G_n\in \mathcal{G}(x;r_n/2,r_n,\tau_n,m)$ has
$(r_n,\tau_n)\to (0,0)$.
Since $\tau_n\to 0$, the image of the ``Gauss map'' of $G_n$
(in the sense of item~(D1) in Remark~\ref{rem5.4})
is arbitrarily small. After re-scaling the ambient metric
on $X$ by $1/r_n$, we find related multi-graphs
$G_n^*=\frac{1}{r_n}G_n$ with constant mean curvature
(which is arbitrarily small if $n$ is taken sufficiently large).
For $n$ sufficiently
large, $G_n^*$ is stable (and $G_n$ as well).
This implies that there exist $\de'_3\in (0,\de_2]$
and $\tau'_1\in (0,\tau]$ such that for every
$r\in (0,\de'_3]$ and $G\in \mathcal{G}(x;r/2,r,\tau'_1,m)$,
$G$ is stable.

From this point on, we will additionally assume that $r_n\in (0,\de'_3]$,
$\tau_n\in (0,\tau'_1]$
while (\ref{eq:aa}) fails to hold for each $n$.
Curvature estimates for stable constant mean curvature surfaces then
imply that there exists $C>0$ (independent
 of $n$)
such that the norm of the second fundamental form of the intersection
of $G_n^*$ with
$A_n^*(\frac{5}{8},\frac{7}{8})$ is less than $C$ for all $n$, where
\[
{\textstyle
A_n^*(\frac{5}{8},\frac{7}{8}):=\frac{1}{r_n}[\ov{B}_X(x,\frac{7}{8}r_n))
\setminus B_X(x,\frac{5}{8}r_n)].}
\]
Since $\tau_n\to 0$, we conclude that the
$G_n^*\cap A_n^*(\frac{5}{8},\frac{7}{8})$
converge as $n\to \infty $ to a flat multi-graph $G^*$ in $\R^3$
over the annulus of inner radius $5/8$ and outer radius $7/8$
(and the convergence $G_n^*\to G^*$ is smooth in the interior of $G^*$),
with some multiplicity $m^*$
at most $m$ (thus the multiplicity $m(G_n)$ of $G_n$ equals $m^*$ for $n$
large enough). Clearly, the total geodesic
curvature of $G^*$ along its intersection with the
sphere $\partial \B (3/4)$ is $2\pi m^*$. Since the
convergence of the $G_n^*$ to $G^*$ is smooth in Int$(G^*)$, we have
that $\kappa (G_n)=\kappa (G_n^*)$ converges as $n\to \infty $ to
$2\pi m^*$,
which equals $2\pi m(G_n)$ for $n$ large enough. Since $\frac{\tau}{m}>0$,
then (\ref{eq:aa}) holds for $n$ large enough, which is contrary to our
hypothesis, and so, the lemma is proved.
\end{proof}

\begin{definition}
\label{def3.6}
{\rm
Fix $L_0>0$ and $m\in \N$, $m\leq \frac{L_0+1}{2\pi }$.
Let $\a_1=\a_1(L_0)\in (0,\tau]$
be the value given by Lemma~\ref{lemma5.3} (recall that
$\tau \in (0,\pi /10]$ is fixed).
Choose $\de_3\in (0,\de_2]$, $\tau_1\in (0,\a_1]$ given by
Lemma~\ref{ass4.4'} such that (\ref{eq:aa}) holds for
every $G\in \mathcal{G}(x;\de_3/2,\de_3,\tau_1,m)$.
}
\end{definition}
Observe that both $\de_3$ and $\tau_1$ depend on the values of $L_0$
and $m$. We will describe later how to choose $L_0,m$ that give rise to 
$\de_3,\tau_1$ by Lemma~\ref{ass4.4'}, in order  to define the values of $\de _1,\de $
that appear in Theorem~\ref{mainStructure}.

\section{The proof of  Structure Theorem~\ref{mainStructure}}
\label{sec3}

Consider numbers $\ve_0>0$, $K_0,H_0,
A_0\in [0,\infty)$, $I\in \N\cup \{0\}$ and
$\tau \in (0,\pi /10]$ and let $\L=\L(I, H_0,\ve_0,A_0,K_0)$ be the
space of CMC immersions given in Definition~\ref{def:L}.

\subsection{The case of uniformly bounded second fundamental form
and the proof of item~(E) of Theorem~\ref{mainStructure}
in the general case}
\label{subsecbdedA}

Suppose that the norms of the second fundamental forms of all
immersions $F\in \Lambda $ are bounded by
a constant $A_1$ independent of $F$ (clearly one can assume
$A_1\geq A_0$).
In this case, Theorem~\ref{mainStructure} holds with the choices  $k=0$
(there are no radii $r_F(i)$ or components $\Delta _i$),
$2\de _1=\de =\de _3$ (this $\de_3$ is given by
Definition~\ref{def3.6} for $(L_0,m)=(2\pi+1,1)$),
and $M=\widetilde{M}$, because:
\begin{enumerate}[(F1)]
\item Items 1, 2, (A), (B) and (D)
of Theorem~\ref{mainStructure} are vacuous.
\item Item 3 of Theorem~\ref{mainStructure} holds by assumption and
item (C) reduces to $g(M)=g(\widetilde{M})$.
\item  We next prove that item~(E) of
Theorem~\ref{mainStructure} holds without the assumption
that the norms of the second fundamental forms of all
immersions $F\in \Lambda $ are bounded by a constant $A_1$
independent of $F$; this will complete the proof of item~(E)
in the general case. In order to
find the constant $C=C(\ve_0,K_0,H_0)>0$
that satisfies item~(E), we distinguish two cases.

\begin{enumerate}[(F3.A)]
\item First suppose that $\partial M\neq \varnothing$.
By item~(A2) in the definition of $\Lambda$,
there exists a point $p_0\in \Int(M)$ such that
$B_M(p_0,\ve_0)$ is contained in the interior of $M$. By
inequality~(\ref{yaulemma1}) in Proposition~\ref{yau},
\begin{equation}
\label{eq:boundary1}
\mathrm{Area}(M)\geq \mathrm{Area}(B_M(p_0,\ve_0))\geq C_A\ve_0,
\end{equation}
where the constants $r_2=r_2(\ve_0,K_0,H_0)>0$,
$C_A=\min \{ \ve_0,\frac{r_2^2}{\ve_0}\} >0$ are given by
Proposition~\ref{yau}.
Given any $y\in M$ such that $d_M(y,\partial M)\geq \ve_0$,
then  (\ref{yaulemma2}) in Proposition~\ref{yau} gives
\begin{equation}
\label{eq:boundary}
\mathrm{Area}(M)\geq \mathrm{Area}(B_M(y,d_M(y,\partial M)))\geq
C_A\, d_M(y,\partial M).
\end{equation}
Define $C_0=\min \{ C_A\ve_0,C_A\}>0$, which
only depends $\ve_0,K_0,H_0$ but not on $I$. We claim that
\begin{equation}
\label{eq:boundary2}
\mbox{\rm Area}(M)\geq C_0\max\{1, \mbox{Radius}(M)\}
\end{equation}
(which proves item~\eqref{itF} of Theorem~\ref{mainStructure}
in this case (F3.A)):
if $\mbox{Radius}(M)\leq 1$,
then our claim follows from~\eqref{eq:boundary1}.
If $\mbox{Radius}(M)>1$, then
our claim follows from~\eqref{eq:boundary} since
$\mbox{Radius}(M)=\sup \{ d_M(y,\partial M) \mid  d_M(y,\partial M)
\geq \ve_0\}$.

\item  Next assume that $\partial M=\varnothing$.  Since the
sectional curvature of $X$ is bounded from above by $K_0$,
then the Ricci curvature of $X$ is bounded from above by
$2K_0$.  It follows that there exists
$\ve_1=\ve_1(K_0,H_0)>0$ such that for any point $x\in X$,
the geodesic spheres of radius at most $\ve_1$ are embedded
with mean curvature greater than  $H_0$. By the
mean curvature comparison principle, for any point $p\in M$,
there is a least one other point $q\in M$
such that the extrinsic distance
$d_X(F(p),F(q))>\ve_1$, and hence,
the intrinsic distance $d_M(p,q)> \ve_1$.
Define
\[
C_A^1=\min \left\{ \ve_1,\frac{r_2^2}{\ve_1}\right\}>0,
\]
where $r_2=r_2(\ve_0,K_0,H_0)>0$ is given by Proposition~\ref{yau},
and let
\[
C_1=\min \{ C_A^1\ve_1,C_A^1\}.
\]
Observe that $C_A^1,C_1$ depend only on $\ve_0,K_0,H_0$ but
not on $I$. We claim that
\begin{equation}
\label{withboundary}
\mbox{\rm Area}(M)\geq C_1\max\{1,
\mbox{\rm Diameter}(M)\}
\end{equation}
(which proves item~\ref{itF} of Theorem~\ref{mainStructure}
in this case (F3.B)).

To prove that \eqref{withboundary} holds, first note that if
$\mbox{\rm Diameter}(M)=\infty$,
then $M$ is non-compact and it has infinite area by
Corollary~\ref{corol11.6}.

Assume now that $ \mbox{\rm Diameter}(M)<\infty $. Since
$M$ is compact, the Hopf-Rinow theorem ensures that there exist points
$p,q\in M$ such that $\mbox{Diameter}(M)=d_M(p,q)$.
Notice that for $n\in \N$ such that $ \mbox{\rm Diameter}(M) >\frac1n $,
the triangle inequality implies
\[
{\textstyle
\mbox{\rm Diameter}(M) -\frac1n = \mbox{\rm Radius}
(M\setminus B_M(q,\frac1n)),}
\]
and so,
\begin{equation}
\mbox{\rm Diameter}(M)=\lim_{n\to\infty}\mbox{\rm Radius}
(M\setminus B_M(q,\textstyle{\frac1n})).
\label{7.5}
\end{equation}
By our choice of $\ve_1$ and for $n$ sufficiently large,
the point $p\in M\setminus B_M(q,\frac1n)$ is at distance at least
$\ve_1$ from $\partial(M\setminus B_M(q,\frac1n))$, and so,
in this case the restriction of the immersion $F\colon M\la X$ to
$M\setminus B_M(q,\frac1n)$ satisfies the hypotheses of
Proposition~\ref{yau}. Therefore, by Proposition~\ref{yau}
and \eqref{eq:boundary2} with $C_0$ replaced by $C_1$,
\begin{equation}
\mbox{\rm Area}(M\setminus B_M(q,\textstyle{\frac1n}))\geq
C_1\max\{ 1, \mbox{\rm Radius}
(M\setminus B_M(q,\textstyle{\frac1n}))\},
\label{7.6}
\end{equation}
hence
\begin{eqnarray}
\mbox{\rm Area}(M)&=&\lim_{n\to \infty}\mbox{\rm Area}
(M\setminus B_M(q,\textstyle{\frac1n}))\nonumber
\\
&\stackrel{\eqref{7.6}}{\geq} &\lim_{n\to\infty}
C_1\max\{1, \mbox{\rm Radius}(M\setminus B_M(q,\textstyle{\frac1n}))\}
\nonumber
\\
&\stackrel{\eqref{7.5}}{=}&C_1\max\{1,\mbox{\rm Diameter}(M)\},\nonumber
\end{eqnarray}
which proves  \eqref{withboundary} holds.
\end{enumerate}
From (F3.A) and (F3.B) we deduce that
item~(E) of Theorem~\ref{mainStructure} holds for the value
$C=\min\{C_0,C_1\}$, regardless of whether or not the 
norms of the second fundamental forms of all
immersions $F\in \Lambda $ are bounded.
\end{enumerate}

In the sequel we will assume that there is no uniform bound
for the norms of the second fundamental forms of surfaces in $\Lambda $.

\subsection{Stable pieces of $H$-surfaces in $\L$ and their
curvature estimate}
\label{sec5.2}
By Theorem~\ref{stableestim1s} and with the notation there,
there exists a universal constant $C_s>0$ such that given
a stable $H$-immersion $F\colon M\la  X$,
\begin{equation}
\label{eqstablecurvestim2}
|A_{M}|(p)\leq \frac{C_s}{\min \{ \ve_0,d_{M}
(p,\partial M),\frac{\pi}{2\sqrt{K_0}}\} },\quad \forall p\in M.
\end{equation}

Define $ \wh{C}_s\colon (0,\ve_0]\to (0,\infty )$ by
\begin{equation}
\label{defC}
\wh{C}_s(\ve)=1+\max\{ A_0,\frac{2C_s}{\min \{ \ve ,
\frac{\pi }{\sqrt{K_0}}\} } \} ,\quad \ve\in (0,\ve_0] .
\end{equation}
It follows that if $F\colon M\la  X$ lies in
$\L=\L(I, H_0,\ve_0,A_0,K_0)$ and $p\in M$
satisfies $|A_M|(p)>\wh{C}_s(\ve)$, then $p\in  U(\partial
M,\ve_0 ,\infty)$ and the intrinsic ball
centered at $p$ of radius $\ve/2$ is unstable.

\begin{lemma}
\label{ass4.4}
Let $F\colon M\la X$ be an element in $\Lambda $
and $\ve \in (0,\ve_0]$ such that $\sup |A_M|
>\wh{C}_s(\ve)$. Then, there exists a finite subset
$\{ q_1,\ldots ,q_k\} \subset U(\partial M,\ve_0 ,\infty) $
with $1\leq k=k(F)\leq I$, such that
\begin{enumerate}
\item $|A_M|$ achieves its maximum in $M$ at $q_1$, and for
$i=2,\ldots ,k$, $|A_M|$ achieves its maximum in
$M\setminus [B_M(q_1,\ve)\cup \ldots \cup
B_M(q_{i-1},\ve)]$ at $q_i$.
\item For each $i=1,\ldots ,k$, $|A_M|(q_i)>\wh{C}_s(\ve)$,
and so, the pairwise disjoint intrinsic
balls $B_M(q_i,\ve/2)$ are unstable.
\item $|A_M|\leq \wh{C}_s(\ve)$ in $M\setminus
[B_M(q_1,\ve)\cup \ldots \cup B_M(q_{k},\ve)]$,
and so, $|A_M|$ is bounded on $M$.
\end{enumerate}
\end{lemma}
\begin{proof}
Since $\sup |A_M|>\wh{C}_s(\ve)$, we can find $q'_1\in M$
such that $|A_M|(q_1')>\wh{C}_s(\ve)$. In particular, the
intrinsic disk $B_M(q_1',\ve/2)$ is unstable. We now
distinguish two possibilities: if $|A_M|\leq \wh{C}_s(\ve)$
on $M\setminus B_M(q_1',\ve)$, then $|A_M|$ is globally
bounded on $M$. Otherwise, there exists
$q'_2\in M\setminus B_M(q_1',\ve)$ such that $|A_M|(q_2')
>\wh{C}_s(\ve)$. In particular, $B_M(q_2',\ve/2)$ is
unstable. Observe that $B_M(q_1',\ve/2)$, $B_M(q'_2,\ve/2)$
are disjoint. Again we discuss two possibilities depending
on whether or not $|A_M|\leq \wh{C}_s(\ve) $ on $M\setminus
[B_M(q_1',\ve)\cup B_M(q_2',\ve)]$. In the first case,
$|A_M|$ is bounded on $M$; in the second case, we repeat
the argument of finding a point $q_3'\in M\setminus
[B_M(q_1',\ve)\cup B_M(q_2',\ve)]$ such that
$|A_M|(q_3')>\wh{C}_s(\ve)$, $B_M(q_3',\ve/2)$ is unstable
and the collection $\{B_M(q_i',\ve/2)\ | \ i=1,2,3\}$ is
pairwise disjoint. Since the index of $F$ is at most $I$,
we cannot repeat this process of finding pairwise disjoint
unstable domains more than $I$ times, say that we can do it
$k'\leq I$ times. Therefore, we conclude that $|A_M|\leq
\wh{C}_s(\ve)$ in $M\setminus [B_M(q'_1,\ve)\cup \ldots
\cup B_M(q'_{k'},\ve)]$; in particular $|A_M|$ is bounded
on $M$. We next replace $q'_1$ by a maximum $q_1$ of
$|A_M|$ in $M$ (which occurs in the compact set
$\ov{B}_M(q'_1,\ve)\cup \ldots \cup
\ov{B}_M(q'_{k'},\ve)$), $q'_2$ by a maximum $q_2$ of
$|A_M|$ in $W_1=[\ov{B}_M(q'_1,\ve)\cup \ldots \cup
\ov{B}_M(q'_{k'},\ve)]\setminus B_M(q_1,\ve)$
(if  $|A_M|$ restricted to $W_1$ is greater than
$\wh{C}_s(\ve)$), and repeat the process to obtain a finite
set of points $\{q_1,\ldots, q_k\}$. Observe that the
number $k$ of these points cannot be greater than $I$. Now
the lemma holds.
\end{proof}

\subsection{Strategy of the proof of Theorem~\ref{mainStructure}}
\label{sec3.4}
Given $t\geq \wh{C}_s(\ve_0)$, let $\Lambda _t$
be the subset of $\Lambda $ consisting of those immersions
$F\colon M\la  X$ such that $\sup \{ |A_M|(p)
\ | \ p\in M\} >t$. Similar arguments as those in
Section~\ref{subsecbdedA} show that
Theorem~\ref{mainStructure} holds for immersions in
$\Lambda \setminus \Lambda _t$, with the choices $k=0$,
$A_1=t$, $2\de _1=\de=\de_3$  given
by Definition~\ref{def3.6} for $(L_0,m)=(2\pi+1,1)$
and $M=\widetilde{M}$. So the theorem will be proven if we
show that it holds for immersions in $\Lambda _t$ for some
large $t\geq \wh{C}_s(\ve_0)$.

Observe that if $I=0$, then $\wh{C}_s(\ve_0)$ is a uniform
bound for the norm of the second fundamental forms of
surfaces in $\L$,  and the theorem holds in this case.

The strategy to prove the theorem consists of proving the
following two steps:
\begin{enumerate}[{\bf Step 1.}]
\item Items~1, 2, 3 of the theorem hold. This will be
proven by induction on $I$, by analyzing local pictures of a
sequence of immersions $\{ F_n\colon M_n\la X_n\}_n\subset \Lambda$ whose
second fundamental forms blow up as $n\to \infty $. We will do this
in Sections~\ref{sec2.6} and~\ref{sec2.7}.

\item If items~1, 2, 3 of the theorem hold, then items
(A)-\ldots -(D) also hold for a possibly larger choice of $A_1$
(recall that we proved item~(E) of Theorem~\ref{mainStructure}
in item~(F3)
of Section~\ref{subsecbdedA}). For this part,
we will verify that the induction argument in Step 1 can be carried
out so that items (A)-\ldots -(D) hold for $F_n$ with $n$ large enough.
This step will be done in Section~\ref{sec2.9}, which in turn needs
some results in Section~\ref{sec9.1}.
\end{enumerate}

Our next goal is to complete Step 1.
Although not strictly needed in the induction process, we
 first explain the arguments needed to prove the case $I=1$ since
they will help clarify why items~1, 2, 3 of the theorem hold for $I+1$
provided that they hold for $I$.

\subsection{Proof of items 1, 2, 3 of
Theorem~\ref{mainStructure} for $I=1$}
\label{sec2.6}

Assume $I=1$. By previous arguments, we can assume that
for each $n>\wh{C}_s(\ve_0)$
there exists an $H_n$-immersion
$F_n\colon M_n\la X_n$ in $\L $ such that $\sup
|A_{M_n}|>n$ with $H_n\in [0,H_0]$.
We will next describe the local picture of any such
sequence $\{ F_n\}_n$ around points of concentrated norm of
their second fundamental forms.
As $I=1$, Lemma~\ref{ass4.4} gives that for each
$n>\wh{C}_s(\ve_0)$, there is a point $p_1(n)\in
U(\partial M_n,\ve_0 ,\infty)$ where $|A_{M_n}|$ achieves
its maximum and $|A_{M_n}|\leq \wh{C}_s(\ve_0)$
in $M_n\setminus B_{M_n}(p_1(n),\ve_0)$.

\subsubsection{Local pictures around points where $|A_M|>t$,
for $t$ sufficiently large}
\label{subseclocpic}

Next we will adapt some arguments in~\cite{mpr20} to this
immersed setting. Given $n>\wh{C}_s(\ve_0)$,
observe that the (unique) maximum of the function
$h_n\colon \ov{B}_{M_n}(p_1(n),\ve_0)\to [0,\infty)$ given
by
\begin{equation}
\label{eq:defhn}
h_n=|A_{M_n}|\, d_{M_n}(\cdot ,\partial B_{M_n}
(p_1(n),\ve_0))
\end{equation}
is attained at $p_1(n)$. Define
$\l_n=|A_{M_n}|(p_1(n))$. Following the arguments at the
beginning of the proof of Theorem~1 in~\cite{mpr20}, we have that:
\begin{enumerate}[(G1)]
\item $\l_n$ tends to infinity as $n\to \infty $.
\item For $r>0$ fixed, the sequence of extrinsic balls
$\{ \l_nB_{X_n}(F_n(p_1(n)),r/\l_n)\}_n$ converges $C^{1,\a}$,
$\a\in (0,1)$, as $n\to \infty $ to the open ball $\B(r)$
of radius $r$ centered at the origin $\vec{0}$ in $\R^3$
with its usual metric, where we have used harmonic
coordinates in $X_n$ centered at $p_1(n)$ and identified
$p_1(n)$ with $\vec{0}$.
\item The intrinsic balls $\l_nB_{M_n}(p_1(n),r/\l_n)$
can be considered to be a sequence of pointed immersions
with constant mean curvature $H_n/\l_n$ (observe that
$H_n/\l_n$ is arbitrarily small for $n$ sufficiently large)
and non-empty topological boundary.
\item For $n$ large, the immersed surface
$\l_nB_{M_n}(p_1(n),r/\l_n)$ passes through $\vec{0}$ with
norm of its second fundamental form equal to $1$ at this
point. Furthermore, the norms of the second fundamental
forms of $\l_nB_{M_n}(p_1(n),r/\l_n)$ are everywhere less
than or equal to 1.
\item After extracting a subsequence, the
$\l_nB_{M_n}(p_1(n),r/\l_n)$ converge $C^{1,\a}$ as mappings
to a relatively compact pointed  minimal immersion
$f_r\colon \Sigma(r)\la \B(r)$  that passes
through $\vec{0}$, with bounded Gaussian curvature and
index at most $1$,  $|A_{\S (r)}|(\vec{0})=1$ and
$|A_{\S (r)}|\leq 1$ on $\Sigma (r)$.
\item Defining $\Sigma=\bigcup _{r\geq 1}\Sigma (r)$ and
$f\colon \Sigma \la \R^3$ by
$f|_{\Sigma(r)}=f_r$, we produce a complete pointed
minimal immersion with index at most $1$, $\vec{0}\in
\Sigma$,  $|A_{\S}|(\vec{0})=1$ and $|A_{\S}|\leq 1$ on
$\Sigma $.
\end{enumerate}

Since $f$ is not flat at the origin, then the index of $f$
is 1. In this setting, L\'opez and Ros~\cite{lor2} proved
that if $\Sigma$ is orientable, then $f$ is either a
catenoid or an Enneper minimal surface. Theorem~1.8 in
Chodosh and Maximo~\cite{ChMa2} gives that $\Sigma$ must be
orientable.

We next show that items 1, 2, 3 of
Theorem~\ref{mainStructure} hold in this case $I=1$ with
the choice $k=1$. Observe that the multiplicity of the end
of the Enneper surface is $m=3$, and the total multiplicity
of the ends of a catenoid is $2$. This motivates the choice
of $L_0$ in the next paragraph. We next explain how to
choose the constants $A_1$ and $\de_1,\de $ that appear in
the main statement of Theorem~\ref{mainStructure}.

Let $\a_1=\a_1(\tau)\in (0,\tau]$ be the constant given by
Lemma~\ref{lemma5.3} for $L_0=6\pi +1$; observe that the
length of the intersection of a catenoid or an Enneper
minimal surface with a sphere $\esf^2(R)$ of sufficiently
large radius $R$ is less than $L_0R$.

We can also pick  a smallest $R>0$ (only depending on
$\tau$) so that the following properties hold:
\begin{enumerate}[(H0)]
\item The index of $f(\Sigma)\cap \B(R/3)$ is $1$.
\end{enumerate}
\begin{enumerate}[(H1)]
\item $f(\Sigma)\setminus \B(R/3)$ consists of one or two
multi-graphs over its projection to a plane
$\Pi\subset \R^3$ that passes though $\vec{0}$; here $\Pi$
is the limit tangent plane at infinity for $f$.

\item The image through the Gauss map of $f$ of
each component $C_j$ of $f(\S)\setminus \B(R/3)$ is contained in
the spherical neighborhood of radius $\a_1/2$ centered at a
point $v\in \esf^2(1)$ perpendicular to $\Pi$ (thus,
$C_j$  satisfies condition (B2) of
Lemma~\ref{lemma5.3} with $R_1=R/3$ and $\a=\a_1/2$).

\item $f(\Sigma)$ makes an angle greater than $\frac{\pi }
{2}-\frac{\a_1}{2}$ with every sphere $\esf^2(r)$ of radius
$r\geq R/3$ centered at the origin (so, $C_j$
satisfies condition (B1) of Lemma~\ref{lemma5.3} with
$R_1=R/3$ and $\a=\a_1/2$).

\item The length of each component of the intersection of
$f(\Sigma)$ with any sphere $\esf^2(r)$ centered at the
origin and radius $r\geq R/3$ is less than
$(L_0-\frac{1}{2})r$ (hence each component of
$f(\S)\setminus \B(R/3)$ satisfies condition~(B3) of
Lemma~\ref{lemma5.3} with $R_1=R/3$).
\end{enumerate}

Applying the estimate
\eqref{eq:lemma5.53} in Proposition~\ref{propos5.5}
with $I=1$ and $B=0$, we deduce:
\begin{enumerate}[(H5)]
\item By item~2 of Proposition~\ref{propos5.5},
the intrinsic distance in the pullback metric by $f$
from $\vec{0}\in \Sigma $ to any point in the boundary of
$f^{-1}(\ov{\B}(R/2))$ is at
most $\wh{C}\frac{R}{2}$, where $\wh{C}$ is defined there.
Observe that~\eqref{eq:lemma5.50}
is not enough to estimate this intrinsic distance, since it only
gives that the intrinsic distance in
the pullback metric by $f$ from $\vec{0}\in \Sigma $ to
the boundary of $f^{-1}(\ov{\B}(R/2))$
is at most $\wh{L}\frac{R}{2}=\frac{\sqrt{3}}{2}R$.
\end{enumerate}

\begin{definition}
\label{defDR}
{\rm
Given $r\in [R/2,4R]$, we denote by ${\Delta}_n(r)\subset M_n$
the connected component of
$(\l _nF_n)^{-1}(\l_n\ov{B}_{X_n}(F_n(p_1(n)),\frac{r}{\l_n}))$
that contains $p_1(n)$.
}
\end{definition}
Properties (H0)-\ldots -(H5) and the convergence in (G5)-(G6)
imply that for $\l_n$ large (in particular, for $n$
sufficiently large), the immersion $\l_nF_n$ satisfies:
\begin{enumerate}[({I}0)]
\item The index of $(\l_nF_n)|_{{\Delta}_n(R/2)}$ equals $1$.
\end{enumerate}
\begin{enumerate}[({I}1)]
\item $(\l_nF_n)({\Delta}_n(4R)\setminus {\Delta}_n(R/2))$
consists of one or two multi-graphs over their projections
to $\Pi$. We let $\wt{G}_n$ denote any of these
multi-graphs inside $(\l_nF_n)({\Delta}_n(4R)\setminus
{\Delta}_n(R/2))$.

\item The image of $\wt{G}_n$ through the ``Gauss map'' of
$\l_nF_n$ (defined through ambient parallel translation,
see Remark~\ref{rem5.4}) is contained in the spherical
neighborhood of radius $\a_1$ centered at $v$ (here we have
identified $\R^3$ with the tangent space to $\l_nX_n$ at
$F_n(p_1(n))$).

\item $\wt{G}_n$ makes an angle greater than $\frac{\pi }
{2}-\a_1$ with every geodesic sphere $\wt{S}(r)$ in $\l_nX_n$
centered at $F_n(p_1(n))$ of radius $r\in [R/2,4R]$.

\item  $\mbox{\rm Length}[ \wt{G}_n\cap \wt{S}(R/2)]<L_0R/2$.

\item The intrinsic distance in the pullback metric by
$\l_nF_n$ on $M_n$, from $p_1(n)$ to any point in the boundary of
$\Delta_n(R/2)$ is at most  $(\wh{C}/2+1)R$.
\end{enumerate}

Back in the original scale, observe that
$\Delta_n(r)\subset F_n^{-1}(\ov{B}_X(F_n(p_1(n)),
\frac{r}{\l_n}))$ for any $r\in [R/2,4R]$,
and the following properties hold for $n$ sufficiently large:
\begin{enumerate}[(J0)]
\item The index of $F_n|_{\Delta_n(R/2)}$ equals $1$.
\end{enumerate}
\begin{enumerate}[(J1)]
\item $F_n(\Delta_n(4R)\setminus \Delta_n(R/2))$ is a union
of one or two multi-graphs over their projections to $\Pi$.
We let $G_n$ denote any of these multi-graphs.

\item The image of $G_n$ through the ``Gauss map'' of $F_n$
is contained in the spherical neighborhood of radius $\a_1$
centered at $v$.

\item $G_n$ makes an angle greater than $\frac{\pi}{2}
-\a_1$ with every geodesic sphere $S(r)$ in $X_n$ centered at
$F_n(p_1(n))$ of radius $r\in
[\frac{R}{2\l_n},\frac{4R}{\l_n}]$.

\item  $\mbox{\rm Length}[ G_n\cap S(\frac{R}{2\l_n})]
<L_0\frac{R}{2\l_n}$.

\item The intrinsic distance in the pullback metric by
	$F_n$ on $M_n$, from $p_1(n)$ to any point in  the boundary of
$\Delta_n(R/2)$ is at most $\frac{1}{\l_n}(\wh{C}/2+1)R$.
\end{enumerate}

Therefore, given $r\in [\frac{R}{2\l _n},\frac{2R}{\l_n}]$,
$G_n\cap [\ov{B}_X(F_n(p_1(n)),2r)\setminus
B_X(F_n(p_1(n)),r)]$
satisfies the hypotheses (B1), (B2), (B3) of
Lemma~\ref{lemma5.3} with the choices $L_0=6\pi +1$, inner
radius $r$, outer radius $2r$ and $\a =\a_1$. Our next step
will be demonstrating that the outer radius for which the
hypotheses of Lemma~\ref{lemma5.3} hold for $F_n$, is
bounded from below by some positive constant, independent
of the sequence.

\subsubsection{Local pictures have a uniform size}\label{sec5.4.2}
\begin{proposition}
\label{ass3.5}
There exists $\de_4\in (0,\de_3]$ (this $\de_3\in (0,\de_2]$
was given in Definition~\ref{def3.6} for the choices
$L_0=6\pi+1$ and $m=3$) such that the hypotheses of Lemma~\ref{lemma5.3}
hold for annular enlargements of the multi-graphs $G_n$
between the geodesic spheres in $X$ centered at $F_n(p_1(n))$
of radii $R_1=\frac{R}{2\l_n}$ and $R_2=\de_4$, and
with the choice $\a =\tau_1$ for hypotheses (B1), (B2)
(this $\tau_1\in (0,\a_1]$ was also introduced in
Definition~\ref{def3.6}).
\end{proposition}
\begin{proof}
Define $r_n$ as the supremum of the extrinsic radii $r\geq
4R/\l_n$ such that annular enlargements of the $G_n$ satisfy
conditions (B1), (B2), (B3) of Lemma~\ref{lemma5.3} for the
choices  $L_0=6\pi +1$, inner radius $R_1=\frac{R}{2\l_n}$,
outer radius $R_2=r$ and $\a=\a_1$. We will prove the
proposition by contradiction, so suppose $r_n\to 0$ as
$n\to \infty $.

Rescale $F_n$ by expanding the ambient metric of $X_n$ by the
factor $1/r_n$ centered at $F_n(p_1(n))$ and denote
the resulting sequence of rescaled immersions by
$\frac{1}{r_n}F_n\colon M_n\la \frac{1}{r_n}X_n$.
Our goal is to understand the limit of (a subsequence of)
$\{ \frac{1}{r_n}F_n\}_n$.

Notice that $4R\leq \l_nr_n$ must go to infinity as $n\to
\infty $: otherwise, $\frac{1}{r_n}F_n$ is rescaled from
$F_n$ on the scale of the second fundamental form,
and in that case we have proved that the subsequential limit of
the $\frac{1}{r_n}F_n$ is a catenoid or an Enneper minimal
surface, each of whose ends satisfies Lemma~\ref{lemma5.3}
for every outer radius--- see properties (H1),(H2),(H3)
above ---, contradicting the definition as a supremum of
$r_n$.

As $\l_nr_n\to \infty $, then property (J0) implies that
$\frac{1}{r_n}F_n$ has index zero away from the origin for
$n$ large; more precisely, the following property holds:

\begin{enumerate}[$(\diamondsuit)$]
\item For any $s>0$ and for every $n\in \N$ sufficiently
large (depending only on $s$), the portion of
$\frac{1}{r_n}F_n(M_n)$ outside of the intrinsic ball
of radius $s$ centered at $F_n(p_1(n))$ is stable.
\end{enumerate}
By curvature estimates for stable $H$-surfaces, we deduce
that the sequence $\{ \frac{1}{r_n}F_n\} _n$ has locally
bounded second fundamental form in $\R^3\setminus\{ \vec{0}\} $.

Applying Lemma~\ref{lemma5.3} (see also Remark~\ref{rem5.4})
to $\frac{1}{r_n}F_n$ with
$\a=\tau_1$, we conclude that for $n$ large,
the image of $\frac{1}{r_n}F_n$ contains an immersed
annulus $\Omega _n(\frac{1}{2},1)$ in the annular region
$\A(\frac{1}{2},1)=\{ x\in  \R^3\ | \  \frac{1}{2}\leq
|x|\leq 1\} $, and $\Omega _n(\frac{1}{2},1)$ is an
$m'$-valued graph with respect to its projection to a plane
$v_1^{\perp}$ passing through the origin (the multiplicity
$m'$ of this graph does not depend on $n$ after passing to a
subsequence; in fact $m'=1$ or $3$; similarly, the plane
$v_1^{\perp}$ is independent of $n$).
Observe that by definition of $r_n$, either
$\Omega _n(\frac{1}{2},1)$ makes an angle of
$\frac{\pi}{2}-\tau_1$ with $\esf^2(1)$
at some point of $\Omega _n(\frac{1}{2},1)\cap \esf^2(1)$,
or the Gauss map image of $\Omega _n(\frac{1}{2},1)$
contains two points at spherical distance $\tau_1$ apart.

After passing to a subsequence, $\Omega _n
(\frac{1}{2},1)$ converges smoothly as $n\to \infty $ to an
immersed minimal annulus $A$ in $\A(\frac{1}{2},1)$ which is
a multi-graph of multiplicity $m'$ with respect to
$v_1^{\perp}$ and either $A$ makes an angle of
$\frac{\pi}{2}-\tau_1$ with $\esf^2(1)$ or the Gauss map
image of $A$ contains two points at spherical distance
$\tau_1$ apart. In particular, $A$ cannot be contained
in a plane passing through the origin.

Repeating the same reasoning in $\A(2^{-k},2^{-k+1})$ for
every $k\in \N$ and using a diagonal argument, we conclude
that (a subsequence of) the $\frac{1}{r_n}F_n$ converges
smoothly in $\A(0,1)=\B(1)\setminus\{ \vec{0}\} $ to an immersed
minimal punctured disk $D^*$ that has $\vec{0}$ in its
closure, such that $A\subset D^*$. As $\{ \frac{1}{r_n}F_n
\} _n$ has locally bounded second fundamental form in
$\R^3\setminus\{ \vec{0}\} $, then (a subsequence of) the
$\frac{1}{r_n}F_n$ converge smoothly to a minimal immersion
$\wt{D}$ in $\R^3\setminus\{ \vec{0}\} $ such that $D^*\subset
\wt{D}$, and $\wt{D}$ is complete away from $\vec{0}$, in
the sense that divergent arcs in $\wt{D}$ either have
infinite length or diverge to $\vec{0}$. Clearly $\wt{D}$
has $\vec{0}$ in its closure. Since $\frac{1}{r_n}F_n$ is
stable away from the origin, then $\wt{D}$ is stable. In
this setting and when $\wt{D}$ is two-sided,
$\wt{D}$ extends smoothly to a plane passing
through $\vec{0}$ (by Lemma~3.3 in
Meeks-P\'erez-Ros~\cite{mpr10}, see also
Colding-Minicozzi~\cite{cm25}). This contradicts the fact that $A$
cannot be  contained in a plane passing through the origin.
In the case that $\wt{D}$ is one-sided, then we can view $\wt{D}$ as a
branched stable minimal immersion with branch locus at the origin (with
finite branching order); in this setting,
item~1 of Lemma~\ref{lema3.6} gives a contradiction.
These contradictions finish the proof of
Proposition~\ref{ass3.5}.
\end{proof}

\begin{definition}\label{def5.4}
Consider the $\de_4\in (0,\de_3]$ given by
Proposition~\ref{ass3.5}. Then we define
\[
\de:=\de_4/2,\quad \de_1=\de/2.
\]
We will show that this is a valid choice for the $\de_1,
\de $ appearing in Theorem~\ref{mainStructure} in the case
$I=1$.
\end{definition}

We finish this section by showing how to deduce items 1, 2, 3
of Theorem~\ref{mainStructure} in this
case of $I=1$ (this is part of Step~1 in our strategy
of proof of Theorem~\ref{mainStructure} explained in
Section~\ref{sec3.4}).
We first explain how to choose the value of $A_1\in
[A_0,\infty)$ that appears in the main statement of the
theorem. In Section~\ref{sec3.4} we saw that it suffices to prove
items 1, 2, 3 of Theorem~\ref{mainStructure} for immersions
in $\Lambda _{t}$ for some large $t> \wh{C}_s(\de_1/2)$.
Choose $t>\wh{C}_s(\de_1/2)$
sufficiently large so that:

\begin{enumerate}[(K1)]
\item  $\frac{R}{t} (\wh{C}/2+1) \leq \frac{\de_1}{10}$)
(recall that $R$ was defined just before items (H0)-\ldots -(H5)
only depending on $\tau$ and $\wh{C}$ was given in item~2
of Proposition~\ref{propos5.5} as a function of $I,B$, which in this
case where $I=1,B=0$ give $\wh{C}=4\sqrt{3}+\frac{11}{2}\pi$,
see also item~(H5)).
\item For every $(F\colon M\la  X)\in
\Lambda _{t}$, Lemma~\ref{ass4.4} applied to $F$ for
$\ve=\ve_0$ implies that there exists a point $p_1\in
U(\partial M,\ve_0 ,\infty)$ such that
\[
|A_{M}|(p_1)=\max \{ |A_{M}|(p)\ | \ p\in M\} ,
\]
and if $t$ is sufficiently large, then
the description in items (J0)-\ldots -(J5) holds for $F$
with $p_1(n),\l_n$ replaced by $p_1$, $|A_{M}|(p_1)$ respectively.
\end{enumerate}

Define $A_1=t$. Next we will prove items 1, 2, 3 of
Theorem~\ref{mainStructure} for immersions in
$\Lambda _{t}$. Given $(F\colon M\la X)\in \Lambda _{t}$, define
$r_F(1)$ to be $\de_1$, and $\Delta_1$ to be the component of
$F^{-1}(\ov{B}_X(F(p_1),r_F(1))$ that contains $p_1$.
Let  $S_F(\frac{R}{2t})$ denote the extrinsic geodesic
sphere in $X$ centered at $F(p_1)$ with radius
$\frac{R}{2t}$. Let $q$ be a point in $\partial \Delta _1$.
Then,
\[
\begin{array}{rclr}
d_M(p_1,q)&\leq &
{\displaystyle \max_{x\in \partial \Delta_1\cap
F^{-1}(S_F(\frac{R}{2t}))}}d_M( p_1,x)+d_M(\Delta_1\cap
F^{-1}(S_F(\frac{R}{2t})),q) &
\\
\rule[-.3\baselineskip]{0pt}{.5cm}
& \leq &
\frac{R}{t} (\wh{C}/2+1)
+d_M(\Delta_1\cap F^{-1}(S_F(\frac{R}{2t})),q). 
& \hspace{-1cm}\mbox{(By (J5),
as $|A_{M}|(p_1)\geq {t}$)}
\end{array}
\]

By properties (J2), (J3) and (J4) and by
Proposition~\ref{ass3.5}, we can apply Lemma~\ref{lemma5.3}
to each of the annular portions of $\Delta _1$ with the
choices $R_1=\frac{R}{2t}$ and $R_2=r_F(1)$ (observe that
$\frac{R}{2t}\leq \frac{R}{t} (\wh{C}/2+1)
\leq \frac{\de_1}{10}<\frac{r_F(1)}{2}$).
Using item (C2) of Lemma~\ref{lemma5.3} (see also item~(D2) of
Remark~\ref{rem5.4})
in the second term of the last RHS, we get
\[
\begin{array}{rclr}
d_M(p_1,q)&\leq &
\frac{R}{t} (\wh{C}/2+1)+\sqrt{1+\frac{\tau^2}{3}}\,
(r_F(1)-\frac{R}{2t}) &
\\
\rule[-.3\baselineskip]{0pt}{.5cm}
&<& \frac{\de_1}{10}+\sqrt{1+\frac{\tau^2}{3}}\, r_F(1) &
\qquad \mbox{(By (K1))}
\\
\rule[-.3\baselineskip]{0pt}{.5cm}
&=&\left( \frac{1}{10}+\sqrt{1+\frac{\tau^2}{3}}\right) r_F(1). &
\end{array}
\]
Since $\tau\leq \pi/10$, then $d_M(p_1,q)<\frac{5}{4}r_F(1)$.
This proves item 1(a) of Theorem~\ref{mainStructure}.

Item 1(b) follows from the definition of $\Delta _1$.
Observe that item 1(c) is vacuous because $k=1$.
Item 1(d) holds because $F\in \Lambda _{t}$ and $A_1=t$.
Item 1(e) follows from (J0) (see also (K2)),
which finishes the proof of item 1 of
Theorem~\ref{mainStructure}. Item 2 follows from
Lemma~\ref{lemma5.3}.

Next we show item 3. Given $q\in \wt{M}=M-\Int(\Delta_1)$,
let $\g \subset M$ be an arc joining $p_1$ with $q$. Let
$\g_1\subset \g$ be the smallest subarc of $\g$ that joins
$p_1$ with some point $q_1\in \partial \Delta_1$. By
definition of $\Delta_1$, $F(q_1)$ is at extrinsic
distance $r_F(1)$ from $F(p_1)$, and thus, length$(\g)\geq
\mbox{length}(\g_1)\geq r_F(1)=\de_1$ for every arc $\g$
joining $p_1$ with $q$. Therefore, $d_M(p_1,q)\geq \de_1$.
As $q$ is any point in $\wt{M}$, we conclude that
$\wt{M}\subset M-B_M(p_1, \de_1)$.
Hence, item~3 of Theorem~\ref{mainStructure}
will be proved if we check that $|A_M|\leq A_1$ in
$M-B_M(p_1, \de_1)$.
Applying item~3 of Lemma~\ref{ass4.4} to
$\ve =\de_1$ (this is possible since
$\de_1\leq \ve_0$ and $\sup |A_M|\geq t>
\wh{C}_s(\de_1/2)\geq \wh{C}_s(\de_1)$),
we conclude that $|A_M|\leq \wh{C}_s(\de_1)$ in
$M-B_M(p_1,\de_1)$. Since $\wh{C}_s$ is non-increasing, we
have $\wh{C}_s(\de_1)\leq \wh{C}_s(\de_1/2)< t=A_1$, so
item~3 of Theorem~\ref{mainStructure}  holds.

Thus, items 1, 2, 3 of Theorem~\ref{mainStructure} hold in
this case $I=1$.

\subsection{Proofs of items 1, 2, 3 of
Theorem~\ref{mainStructure} for $I=I_0+1$}
\label{sec2.7}

Assume that items 1, 2, 3 of Theorem~\ref{mainStructure}
hold for $I=I_0$ and we will prove the same items hold for
$I=I_0+1$.

By the arguments in the first paragraph of
Section~\ref{sec3.4}, we can assume that for each
$n>\wh{C}_s(\ve_0)$
there exists an $H_n$-immersion $F_n\colon M_n\la X_n$ in
$\L (I_0+1,H_0,\ve_0,A_0,K_0)$ such that
$\sup |A_{M_n}|>n$. By Lemma~\ref{ass4.4}, for each
$n>\wh{C}_s(\ve_0)$
there exists a finite set
\[
\{ p_1(n),\ldots ,p_{m(n)}(n)\} \subset U(\partial M_n,\ve_0 ,
\infty),
\]
$m(n)\leq I_0+1$, such that
\begin{enumerate}[(L1)]
\item $|A_{M_n}|$ achieves its maximum in $M_n$ at $p_1(n)$
and for $i=2,\ldots , m(n)$, $|A_{M_n}|$ achieves its
maximum in $M_n\setminus [B_{M_n}(p_1(n),\ve_0)
\cup \ldots \cup B_{M_n}(p_{i-1}(n),\ve_0)]$ at $p_i(n)$.
\item For each $i=1,\ldots ,m(n)$,
$|A_{M_n}|(p_i(n))>\wh{C}_s(\ve_0)$,
and so, the pairwise disjoint intrinsic balls
$B_{M_n}(p_i(n),\ve_0/2)$ are unstable.
\item $|A_{M_n}|\leq \wh{C}_s(\ve_0)$
in $M_n\setminus
[B_{M_n}(p_1(n),\ve_0)\cup \ldots \cup
B_{M_n}(p_{m(n)}(n),\ve_0)]$.
\end{enumerate}

\subsubsection{Local pictures around points where $|A_M|>t$,
for $t$ sufficiently large}
\label{subseclocpic1}

Given $n>\wh{C}_s(\ve_0)$,
consider the function
$h_n\colon \ov{B}_{M_n}(p_1(n),\ve_0)\to [0,\infty)$ given
by (\ref{eq:defhn}). As in the case $I=1$, the
maximum of $h_n$ occurs at $p_1(n)$. Let
$\l_n=|A_{M_n}|(p_1(n))$. Then, properties (G1)-\ldots -(G6)
hold with the only change in (G5) (resp. in (G6)) that $f_r$
(resp. $f$) has index at most $I_0+1$.
In the sequel, will use the same notation as in (G1)-\ldots
-(G6).

Unlike what we had in the case $I=1$, we do not dispose of a
classification result for the possible limit minimal
immersion $f$ in this current setting. Still, we can
estimate some aspects of its geometry; observe that $f$ has
finite total curvature, since it has finite index
(see Fischer-Colbrie~\cite{fi1} for the orientable case, and see
the last paragraph of the proof of Theorem~17 in Ros~\cite{ros9}
for the non-orientable case).
Therefore, $f$ is proper,
the domain $\Sigma$ of $f$ has finite genus and finitely
many ends, each of which is mapped by $f$ to a multi-graph
over the exterior of a disk in a plane of $\R^3$ passing through the origin, 
with finite multiplicity.
We will denote by $e\geq 1$
the number of ends of $f$, and $d_1,\ldots, d_e\geq 1$
the multiplicities of these ends. Hence,
$\sum_{j=1}^{e}d_j$ is the total spinning
of the ends. Also, $g(\Sigma), I(f)$ will stand for the
genus of $\Sigma$ and the index of $f$, respectively.

\begin{claim}[Lower bound for the total spinning plus the
number of the ends of $f$]
\label{claim2.12}
\begin{equation}
\label{eq:lbse}
\sum_{j=1}^{e}(d_j+1)\geq 4.
\end{equation}
\end{claim}
\begin{proof}
If all the ends of $f$ are embedded, then
$e\geq 2$ (as $f$ is not flat) and $d_j=1$
for each $j=1, \ldots ,e$. Thus,
$\sum_{j=1}^{e}(d_j+1)=2e
\geq 4$.

If $f$ has at least one non-embedded end,
then the monotonicity formula for minimal surfaces implies
that the area growth of $f$ at infinity is at least that of
three planes (again because $f$ is not flat). Therefore, in this case,
$\sum_{j=1}^{r}d_j\geq 3$ and the claim follows.
\end{proof}

\begin{claim}[Upper bound for the genus of $\Sigma$]
\mbox{}\newline
If $\Sigma $ is orientable, then $2g(\Sigma)\leq 3I(f)-3$.
If $\Sigma $ is non-orientable, then $g(\widetilde{\Sigma})\leq 3I(f)-4$,
where $g(\wt{\S})$ is the genus of the orientable cover $\wt{\S}$ of $\S$.
\end{claim}
\begin{proof}
This follows directly from equations (\ref{eq:CMindex1})
and (\ref{eq:lbse}),
after observing that the total branching order $B(\Sigma)$ of $f$ is zero.
\end{proof}

\begin{claim}[Upper bound for the total spinning of $f$]
\label{eqts}
\[
2\sum_{j=1}^e d_j\leq \left\{ \begin{array}{ll}
3I(f)+3	& \mbox{if $\S$ is orientable,}
\\
3I(f)+2	& \mbox{if $\S$ is non-orientable.}
\end{array}\right.
\]
\end{claim}
\begin{proof}
This follows directly from 	(\ref{eq:CMindex1}) since $e\geq 1$
and $g(\S)\geq 0$ if $\S$ is orientable (resp. $g(\wt{\S})\geq 0$
if $\S$ is non-orientable).
\end{proof}

Recall that we have fixed $\tau \in (0,\pi/10]$.
Let $\a_1=\a_1(\tau)\in (0,\tau]$ be the constant given by
Lemma~\ref{lemma5.3} for $L_0=3\pi (I_0+2)+1$; observe that
the total length $L^f({r})$ of the intersection of $f(\Sigma)$
with a sphere $\esf^2({r})$ of sufficiently large radius ${r}$
is less than $L_0{r}$; this follows since for ${r}$ large,
\begin{equation}
\frac{L^f({r})}{r}\sim 2\pi \sum_{j=1}^{e}d_j
\stackrel{\mbox{\footnotesize{(Claim \ref{eqts})}}}{\leq}
\pi [3I(f)+3]\leq\pi [3(I_0+1)+3].
\label{spinf}
\end{equation}

We can also pick  a smallest $R>0$ (only depending on $\tau$)
so that the following properties hold
(compare with properties (H0)-\ldots-(H5) above):

\begin{enumerate}[(H0)']
\item The index of $f(\Sigma)\cap \B(R/3)$ is $I(f)$.
\end{enumerate}
\begin{enumerate}[(H1)']
\item $f(\Sigma)\setminus \B(R/3)$ consists of
$e$ multi-graphs over their
projections to planes $\Pi_j\subset \R^3$ passing though
$\vec{0}$, $j=1,\ldots ,e$.
\item The image through the Gauss map of $f$ of
each component $C_j$ of $f(\S)\setminus \B(R/3)$ is contained
in the spherical
neighborhood of radius $\a_1/2$ centered at a point $v_j\in \esf^2(1)$
perpendicular to $\Pi_j$ (thus, $C_j$ satisfies item (B2)  of
Lemma~\ref{lemma5.3} with $R_1=R/3$ and $\a=\a_1/2$).

\item $f(\Sigma)$ makes an angle greater than
$\frac{\pi }{2}-\frac{\a_1}{2}$
with every sphere $\esf^2({r})$ of radius ${r}\geq R/3$ centered
at the origin (so, $C_j$  satisfies item (B1)
of Lemma~\ref{lemma5.3} with $R_1=R/3$ and $\a=\a_1/2$).

\item The total length of the intersection of $f(\Sigma)$ with any sphere
$\esf^2({r})$ centered at the origin and radius ${r}\geq R/3$ is less than
$(L_0-\frac{1}{2}){r}$ (hence $C_j$ satisfies item (B3)
of Lemma~\ref{lemma5.3} with $R_1=R/3$).
\end{enumerate}

Applying the last sentence in item~2 of 
Proposition~\ref{propos5.5} with $I=I_0+1$ and $B=0$, we deduce that:

\begin{enumerate}[(H5)']
\item The intrinsic distance in the pullback metric by $f$
from $\vec{0}\in \Sigma $ to any point in the boundary of
$f^{-1}(\ov{\B}(R/2))$ is at most $a(I_0)R$, where
\begin{equation}
\label{eq:a(I0)}
a(I_0)=\frac{\wh{C}(I_0+1,0)}{2}=\sqrt{6} (3 I_0+1)
\sqrt{I_0+2}+\frac{\pi}{4}(6 I_0+11).
\end{equation}
\end{enumerate}

Given $r\in [\frac{R}{2},4R]$, let ${\Delta}_n(r)$ be the
domain inside $M_n$ given by Definition~\ref{defDR}, related
to the $f,R$ above. Properties (H0)'-\ldots -(H5)' imply
that for $\l_n$ large, the immersion $\l_nF_n$ satisfies
(compare with properties (I0)-\ldots-(I5) above):
\begin{enumerate}[({I}0)']
\item The index of $(\l_nF_n)|_{\Delta_n(R/2)}$ equals
$I(f)$.
\end{enumerate}
\begin{enumerate}[({I}1)']
\item $(\l_nF_n)(\Delta_n(4R)\setminus \Delta_n(R/2))$ can
be considered to be a union of $e$
multi-graphs over their projections to the $\Pi_j$,
$j=1,\ldots ,e$.
We denote these multi-graphs by $\wt{G}_n(1),\ldots, \wt{G}_n(e)$.

\item For $j=1,\ldots ,e$, the image of $\wt{G}_n(j)$ through the
``Gauss map'' of
$\l_nF_n$ (defined through ambient parallel translation, see
Remark~\ref{rem5.4}) is contained in the spherical
neighborhood of radius $\a_1$ centered at $v_j$
(here we have identified $\R^3$ with the tangent space to
$\l_nX$ at $F_n(p_1(n))$).

\item $\wt{G}_n(j)$ makes an angle greater than $\frac{\pi }
{2}-\a_1$ with every geodesic sphere $\wt{S}({r})$ in $\l_nX_n$
centered at $F_n(p_1(n))$ of radius ${r}\in [R/2,4R]$.

\item  $\mbox{\rm Length}[ \wt{G}_n(j)\cap \wt{S}(R/2)]
<L_0R/2$.

\item The intrinsic distance in the pullback metric by
$\l_nF_n$ on $M_n$, from $p_1(n)$ to any point of the  boundary of
$\Delta_n(R/2)$ is at most $[a(I_0)+1]R$.
\end{enumerate}

Back in the original scale, we have that $\Delta_n({r})\subset F_n^{-1}
(\ov{B}_{X_n}(F_n(p_1(n)),{r}/\l_n))$ for all ${r}\in [\frac{R}{2},4R]$
and the following properties hold  for $n$ sufficiently large:

\begin{enumerate}[(J0)']
\item The index of $F_n|_{\Delta_n(R/2)}$ equals $I(f)$.
\end{enumerate}
\begin{enumerate}[(J1)']
\item $F_n(\Delta_n(4R)\setminus \Delta_n(R/2))$ is a union
of $e$ multi-graphs over their projections
to the $\Pi_j$, $j=1,\ldots ,e$.
We call $G_n(1),\ldots, G_n(e)$ to these multi-graphs.

\item For $j=1,\ldots,e$, the image of $G_n(j)$ through the
``Gauss map'' of $F_n$ is contained in the spherical neighborhood of
radius $\a_1$ centered at $v_j$.

\item $G_n(j)$ makes an angle greater than $\frac{\pi}{2}-\a_1$
with every geodesic sphere $S({r})$ in
$X_n$ centered at $F_n(p_1(n))$ of radius
${r}\in  [\frac{R}{2\l_n},\frac{4R}{\l_n}]$.

\item  $\mbox{\rm Length}[ G_n(j)\cap S(\frac{R}{2\l_n})]
<L_0\frac{R}{2\l_n}$.

\item The intrinsic distance in the pullback metric by $F_n$ on $M_n$,
from $p_1(n)$ to any point of the boundary of $\Delta_n(R/2)$ is at most
$\frac{R}{\l_n}[a(I_0)+1]$.
\end{enumerate}

Therefore, given ${r}\in [\frac{R}{2\l _n},\frac{2R}{\l_n}]$, $G_n(j)\cap
[\ov{B}_{X_n}(F_n(p_1(n)),2{r})\setminus B_{X_n}(F_n(p_1(n)),{r})]$
satisfies the hypotheses (B1), (B2), (B3) of Lemma~\ref{lemma5.3}
with the choices $L_0=3\pi (I_0+2)+1$, inner extrinsic radius ${r}$,
outer extrinsic radius $2{r}$ and $\a =\a_1$.

\subsubsection{How to proceed if the (first) local pictures fail to have a
uniform size}

\begin{definition}
\label{defrn}
{\rm
Define ${r}_n$ as the supremum of the extrinsic radii
${r}\geq 4R/\l_n$ such that for all $j=1,\ldots,e$,
annular enlargements $\wh{G}_n(j)$ of the $G_n(j)$
satisfy conditions (B1), (B2), (B3) of Lemma~\ref{lemma5.3}
for the choices $L_0=3\pi (I_0+2)+1$, inner extrinsic radius $R_1=
\frac{R}{2\l _n}$, outer extrinsic radius $R_2={r}_n$ and $\a=\a_1$.
}
\end{definition}
\begin{figure}[h]
\begin{center}
\includegraphics[width=11cm]{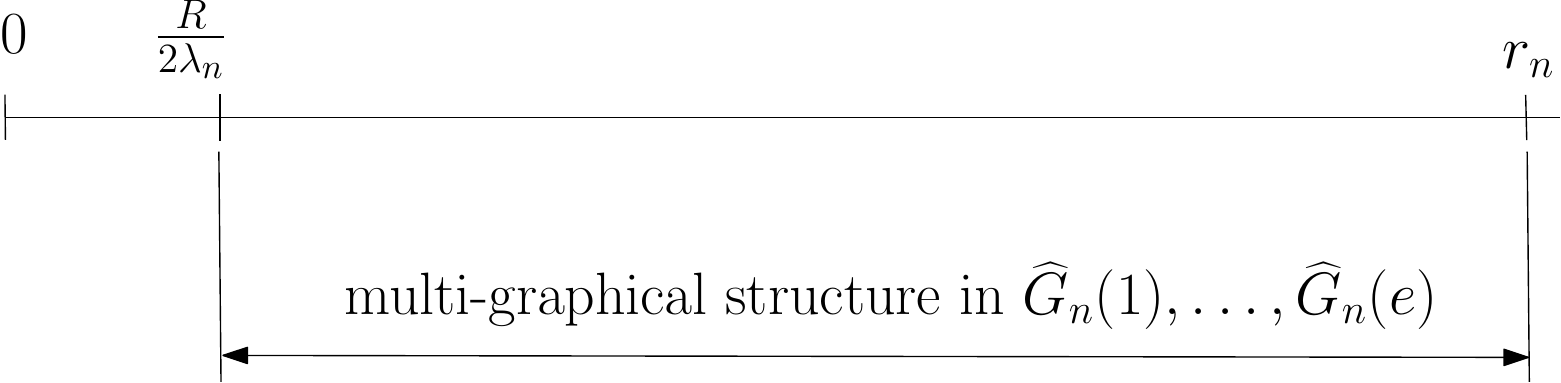}
\caption{Schematic representation of the extrinsic geometry of the
immersion $(F_n\colon M_n\la X_n)\in \Lambda=\L(I_0+1,H_0,\ve_0,A_0,K_0)$
around a point $p_1(n)$ where the maximum of $|A_{M_n}|$ in $M_n$ is
achieved. Here, $\l_n=|A_{M_n}|(p_1(n))$ tends
to infinity and $\l_nF_n$ converges as $n\to \infty$
to the complete minimal immersion $f\colon \S\la \R^3$ with finite total
curvature. Horizontal distances in the figure represent
extrinsic distances in $X_n$ measured from $F_n(p_1(n))$.
For $n$ large enough and
in the range of extrinsic radii between $\frac{R}{2\l_n}$
and $r_n\geq 4R/\l_n$, $F_n$ consists of $e$ multi-graphical pieces
$\wh{G}_n(1),\ldots ,\wh{G}_n(e)$, where $e$ is the
number of ends of $f$.}
\label{fig4a1}
\end{center}
\end{figure}

\begin{remark}
\label{rem2.15}
{\rm
\begin{enumerate}
\item Unlike what happened in the case $I=1$ (Section~\ref{sec2.6}),
we can no longer insure that
the outer extrinsic radius $r_n$ is bounded from below
by some positive constant independent of $n$
(i.e. Proposition~\ref{ass3.5} does not necessarily hold in our setting).
The reason for this difference is that in our current situation,
the estimate $I(f)\leq I_0+1$
is not necessarily an equality (as it was when $I=1$), and thus, with the
notation in the proof of Proposition~\ref{ass3.5}, we cannot
insure that if $r_n\to 0$ as $n\to \infty $, then
$\frac{1}{r_n}F_n$ has index zero away from the origin for $n$ large.

\item If $r_n\to 0$ as $n\to \infty $ and
$\frac{1}{r_n}F_n$ has index zero away from the origin for
$n$ large in the sense that property $(\diamondsuit)$ above
holds, then the arguments in the proof of Proposition~\ref{ass3.5} lead to a contradiction.
Hence we conclude that one of the two following excluding
possibilities holds:
\begin{enumerate}[(2.A)]
\item $\{r_n\}_n$ is bounded away from zero, with this lower bound
being independent of the sequence $\{ F_n\}_n\subset \Lambda$.
In this case, Proposition~\ref{ass3.5} holds (now
$\de_3\in (0,\de_2]$ is given by Definition~\ref{def3.6} for the choices
$L_0=3\pi(I_0+2)+1$
and $m$ being 1 plus the integer part of $\frac{1}{2}[3(I_0+1)+3]$, see
equation~\eqref{spinf} which estimates the total spinning of $f$ by
above; also see Proposition~\ref{ass3.5'} below).
In this case, we can apply Proposition~\ref{propos7.17} below to
conclude the proof of items 1, 2, 3 of Theorem~\ref{mainStructure}.

\item There exists some sequence $\{ F_n\}_n\subset \Lambda$ 
(with associated base points $p_1(n)$) such that
$r_n\to 0$ and $\frac{1}{r_n}F_n$ fails to have index
zero away from the origin for $n$ large, in the sense that
the property $(\diamondsuit)$ above fails.
\end{enumerate}
\end{enumerate}
}
\end{remark}

Assume that we are in case (2.B) above.
Roughly speaking, we will show that the immersions $\frac{1}{r_n}F_n$
converge as $n\to \infty$ to a possibly finitely branched,
complete minimal immersion $f_2\colon \Sigma_2\la \R^3$
away from finitely many points where curvature blows-up.
Furthermore, $\Sigma_2$ is finitely connected and its Morse index is
at most $(I_0+1)-I(f_1)\leq I_0$.
This compactness result is delicate and we will divide its proof into the
following two steps.
\begin{enumerate}[(M1)]
\item Describe the behavior of the immersions $\frac{1}{r_n}F_n$ 
near the origin as $n\to \infty $. We will
do this in Lemmas~\ref{disks} and~\ref{lemma7.11}.

\item Analyze the {\it global} convergence of the $\frac{1}{r_n}F_n$
(after passing to a subsequence) to a complete, finitely branched minimal
immersion $f_2\colon \S_2\la \R^3$ with finite total curvature. We will
do this in Proposition~\ref{prop7.12}.
\end{enumerate}
The proof of the next lemma follows easily from the behavior of the blow-down limit
of any of the $e$ ends of the complete minimal immersion
$f=f_1\colon \Sigma \la  \R^3$ defined just after items (L1)-(L3).

\begin{lemma} \label{disks}
Relabel as $e_1=e$ the number of ends of $f_1$. Suppose
$r_n\to 0$ as $n\to \infty $.  Then, after choosing a subsequence,
each of the finite number of  extended and scaled multi-graphs
$(\frac{1}{r_n}F_n)|_{\wh{G}_n(j)}$, considered to be a
mapping on an open annulus, converges as $n\to \infty $ to a
conformal minimal immersion of a punctured disk
\[
f_{2,j}\colon \D^*=\{z\in \C\mid 0<|z|<1\}\la\R^3
\]
(here $j\in
\{ 1,\ldots ,e_1\}$ refers to the j-th end of $f_1$) with
$f_{2,j}(\D^*)\subset \B(1)\setminus
\{ \vec{0}\}$. Furthermore, for each such $j$,
\begin{enumerate}
\item $f_{2,j}$ extends analytically to a possibly branched minimal disk
$\ov{f}_{2,j}\colon \D=\D^*\cup \{ 0\}\la \rth$
with $\ov{f}_{2,j}(0)=\vec{0}$;

\item The branching order of $\ov{f}_{2,j}$ at $0$ is one less than the
multiplicity of the associated sequence of multi-graphs
$(\frac{1}{r_n}F_n)|_{\wh{G}_n(j)}$. Such
multiplicity (which is independent of $n$ large) coincides with the
spinning of associated j-th end of $f_1\colon \Sigma \la \R^3$.
\end{enumerate}
\end{lemma}

Let $\cD=\{D_1,\ldots,D_{e_1}\}$ be the set of
parameter domains of the associated branched minimal
disks $\{\ov{f}_{2,1},\ldots,\ov{f}_{2,e_1}\}$ given by item~1
of Lemma~\ref{disks},
and consider the map $F_\infty\colon \cup\cD\la \B(1)$ be defined by
\[
F_{\infty }|_{D_i}=\ov{f}_{2,i},\quad i=1,\ldots ,e_1.
\]
Observe that $\cup\cD$ (disjoint union) can be considered to be a
smooth surface.
Let $\cS(0)\subset \cup\cD$ be the finite set of centers of the disks
$D_i$, $i=1,\ldots ,e_1$. Consider the quotient space $\wh{\cD}$ of
$\cup \cD$ where each of the elements in $\cS(0)$
identifies to one point that we denote by $\wh{0}\in \wh{\cD}$,
and every other point of $\cup \cD$ only identifies with itself. Let
\[
\pi \colon \cup \cD \to \wh{\cD}
\]
be the related quotient map, that is, $\pi |_{\cS(0)}$ is the
constant map equal to $\wh{0}$,
and the restriction of $\pi $ to $(\cup \cD)\setminus \cS(0)$
is injective. After endowing $\wh{\cD}$ with the quotient topology,
$\wh{\cD}$ is a path-connected topological space and
\begin{equation}
\wh{\cS}(0):=\pi(\cS(0))=\{ \wh{0}\}.
\label{S(0)}
\end{equation}
Furthermore, $\wh{\cD}\setminus \wh{\cS}(0)$ is a smooth immersed
surface. In what follows, we will at times consider
the induced well-defined continuous map
$F_\infty\colon \wh{\cD}\la  \B(1) $, which we denote the same way.

The next statement can be viewed as an direct consequence of
Lemma~\ref{disks}.
\begin{lemma}
\label{lemma7.11}
In the above situation, the following properties hold:
\begin{enumerate}[1.]
\item $F_\infty$ restricted to $F_\infty^{-1}(\B(1)\setminus
\B(\frac12))$ consists of $e_1$ multi-graphs.
\item The sequence of immersions $\frac{1}{r_n}F_n$ restricted to the
component $\Delta_{2,n}\subset M_n$ of
\[
{\textstyle (\frac{1}{r_n}F_n)^{-1}(B_{\frac{1}{r_n}X_n}(
\vec{0},1/2))}
\]
that contains $p_1(n)$, converges as
$n\to \infty$ to $F_{\infty}$,  where we consider
$F_\infty\colon \wh{\cD}\la  \B(1) $ to be defined on the quotient space
$\wh{\cD}$.
\item The convergence in item~2 is smooth away from $\cS(0)$, or from
$\wh{\cS}(0)$ when we consider $F_\infty$ to be defined on $\wh{\cD}$.
\end{enumerate}
\end{lemma}

Lemma~\ref{lemma7.11} describes the convergence of (a subsequence of) the
$\frac{1}{r_n}F_n$ in a neighborhood of $\wh{\cS}(0)$, to a family
$F_{\infty}\colon \wh{\cD}\la \B(1)$ of minimal disks branched
at the origin, and finishes step (M1) above.

Step (M2) needs two ingredients, which are Lemma~\ref{lemma7.10}
and Proposition~\ref{prop7.12} below. The first one relies on the
validity of Theorem~\ref{mainStructure} for $I=I_0$
(by the induction hypothesis), while
in Proposition~\ref{prop7.12} we will construct the complete,
finitely branched minimal
immersion $f_2\colon \Sigma \la \R^3$ of finite total Gaussian curvature,
which  is the limit of a subsequence
of  the $\frac{1}{r_n}F_n$ as a consequence of Lemma~\ref{lemma7.10}.

We remark that the surfaces $M_n$ and the associated
points $p_1(n)$ in the next theorem are not the same surfaces
and points that we have been using previously in this section
with this notation; so the reader should keep in mind
this abuse of notation when reading the next result.

\begin{lemma}
\label{lemma7.10}
Consider a sequence $(\wt{F}_n\colon M_n\la \wt{X}_n)\in
\Lambda(I_0,H_0,\ve_0,A_0,K_0)$
such that the following properties hold.
\begin{enumerate}[(N1)]
\item $\{ \max_{M_n}|A_{\wt{F}_n}|\}_n$ is not bounded from
above. In particular, after passing to a subsequence we can assume there
exists $p_1(n)\in M_n$ such that
\[
\max_{M_n}|A_{\wt{F}_n}|=|A_{\wt{F}_n}|(p_1(n))>\max\{ n,A_1\}
\quad \mbox{for every $n\in \N$,}
\]
where $A_1\in [A_0,\infty)$ is given in the statement of
Theorem~\ref{mainStructure} for $I=I_0$ (which can be applied by the
induction hypothesis); observe that the existence of $p_1(n)$ is
guaranteed by Lemma~\ref{ass4.4}.

\item In harmonic coordinates centered at $\widetilde{F}_n(p_1(n))$, hence
$\widetilde{F}_n(p_1(n))=\vec{0}$ for all $n\in \N$, the metrics on $\wt{X}_n$
converge uniformly in the $C^0$ norm to the flat metric
on $\R^3$, and the (constant) mean curvatures of the $\wt{F}_n$ converge
to zero as $n\to \infty$. 
\end{enumerate}

Let $\Delta_1(n)$ be the component of
$\wt{F}_n^{-1}(\ov{B}_{\wt{X}_n}(\wt{F}_n(p_1(n),r_{\wt{F}_n}(1)))$
described in item~\ref{it1} of Theorem~\ref{mainStructure} and let
$\Delta_1(n,\frac23)=\Delta_1(n)\cap
\wt{F}_n^{-1}(\ov{B}_{\wt{X}_n}(\wt{F}_n(p_1(n)),
\frac{2}{3}r_{\wt{F}_n}(1)))$.
Then, after replacing by a subsequence, the following hold:
\begin{enumerate}
\item $\{ r_{\wt{F}_n}(1)\}_{n\in \N}$ converges to a positive number
$r\in [\de_1,\frac{\de}{2}]$, where
$\de_1,\de\in (0,\ve_0/2]$ are given by Theorem~\ref{mainStructure}.

\item Let $b$ be the number of boundary components of $\Delta_1(n)$, which
is independent of $n$. Then, the   $b$ multi-graphs
\[
\wt{F}_n(\Delta_1(n))\cap[\ov{B}_{\wt{X}_n}(\wh{F}_n(p_1(n)),
r_{\wt{F}_n}(1))\setminus  B_{\wt{X}_n}(\wh{F}_n(p_1(n)),\textstyle{\frac{1}{2}}r_{\wt{F}_n}(1))]
\]
described in item~\ref{it3} of Theorem~\ref{mainStructure}
converge as $n\to \infty$ to $b$ minimal multi-graphs in
$\ov{\B}(\vec{0},r)\setminus  \B(\vec{0},r/2)$, each of which satisfies
the same estimate~\eqref{estimu} as the multi-graphs in the sequence that
converge to it.
		
\item There exists $J\in \N$, $J\leq I_0$, $\ve_1\in (0,r)$
and a finite set for each $n\in \N$
\[
\textstyle{Q(n)=\{ q_1(n)=p_1(n),q_2(n),\ldots,q_J(n)\}\subset B_{M_n}
\left( p_1(n),\frac{2}{3}r\right),}
\]
such that
\begin{enumerate}
\item $|A_{\wt{F}_n}|(q_i(n))>\max\{ n,A_1\}$, for all $i=1,\ldots, J$
and for each $n\in \N$ (compare to item 1.d of
Theorem~\ref{mainStructure}).

\item Given $i,j\in 1,\ldots ,J$ with
$i\neq j$, the intrinsic distance in $M_n$ between $q_i(n)$ and $q_j(n)$
is at least $\ve_1$ (compare to item 1.c of Theorem~\ref{mainStructure}).

\item Given $s\in \N$, $\{ |A_{\wt{F}_n}|\}_n$ is uniformly bounded in
$B_{M_n}(p_1(n),\frac{2}{3}r)\setminus
\cup_{i=1}^JB_{M_n}(q_i(n),\frac{\ve_1}{3s})$
(compare to item~\ref{it7} of Theorem~\ref{mainStructure}).

\item There exist (not necessarily distinct) points $x_1=\vec{0},x_2,
\ldots ,x_J\in \B(\vec{0},\frac{2}{3}r)$ (this is the ball in $\R^3$ with
its flat metric) such that when viewed in harmonic coordinates
in $\wt{X}_n$ centered at $p_1(n)$,
the points $\wt{F}_n(q_i(n))$ converge as $n\to \infty$
to $x_i$, for each $i=1,\ldots ,J$.
\end{enumerate}

\item For $s\in \N$ large and fixed, and for each $i\in \{1,\ldots, J\}$,
there exist $\de_i(s),\de_i(1,s),r_i(n,s)$
with $0<\de_i(1,s)\leq r_i(n,s)\leq
\de_i(s)/2<\de_i(s)<\frac{2\ve_1}{3s}$,
such that the following hold.
Let $A_i(n,s)$  be the component of
$\wt{F}_n^{-1}(B_{\wt{X}_n}(\wt{F}_n(q_i(n)),r_i(n,s)))$
that contains $q_i(n)$. Then, there exists $s_0\in \N$  such that for each
integer $s\geq s_0$, there exists $N(s)\in \N$ so that for $n\geq N(s)$:
\begin{enumerate}
\item The positive numbers $r_i(n,s)$ converge as $n\to \infty$ to
some $r_i(s)\in [\de_i(1,s),\de_i(s)/2]$.

\item $A_i(n,s)$ is  compact with smooth non-empty boundary and
$$\wt{F}_n(\partial A_i(n,s))\subset \partial
B_{\wt{X}_n}(\wt{F}_n(q_i(n)),r_i(n,s))$$ (compare to item~1.b of
Theorem~\ref{mainStructure}).

\item The number $\wt{e}_i\in \N$ of boundary components of $A_i(n,s)$
is independent of $n,s$ and the restriction of $\wt{F}_n$ to an annular
neighborhood of each boundary component of $A_i(n,s)$ is a multi-graph
of positive integer multiplicity $m_{h,i}$ independent of $n,s$
(here $h\in \{ 1,\ldots ,\wt{e}_i\}$), whose related graphing
function $u=u_{n,s}$ satisfies inequality~\eqref{estimu}
for $n,s$ sufficiently large, where $x$ expresses harmonic coordinates in
$B_{\wt{X}_n}(\wt{F}_n(q_i(n)),\frac{\ve_1}{2})$
(compare to item~\ref{it3} of Theorem~\ref{mainStructure});
the union of these annular neighborhoods
of $\partial A_i(n,s)$ can be taken to
be $ A_i(n,s)\setminus  \wt{F}_n^{-1}(B_{\wt{X}_n}(\wt{F}_n(q_i(n)),
\textstyle{\frac{r_i(n,s)}{2}}))$.

\item The $\wt{F}_n$ restricted to $\Delta_1(n,\frac23)\setminus
\cup_{i=1}^JA_i(n,s)$ converge smoothly as $n\to \infty $ to a minimal
immersion
\[
F_{\infty,s}\colon M_s \la \textstyle{\B(\vec{0},\frac{2}{3}r)}
\]
of a compact surface $M_s$ with boundary, and
$F_{\infty,s}(M_s)\cap [\ov{\B}(\vec{0},\frac{2}{3}r)\setminus
\B(\vec{0},\frac{1}{2}r)]$ consists of the intersection of the limiting
multi-graphs appearing in item~2 with
$\ov{\B}(\vec{0},\frac{2}{3}r)\setminus \B(\vec{0},\frac{1}{2}r)$.
			
\item The boundary $\partial M_s$ decomposes into $J+1$ collections of curves
\[
\{ \a_1,\ldots ,\a_b\},\quad \{ \be_{1,i}(s),\ldots ,\be_{\wt{e}_i,i}(s)
\} _{i=1,\ldots,J}
\]
(recall that $b$ is the number of boundary components of $\Delta_1(n)$)
where $F_{\infty,s}(\a_h)\subset \partial\B(\vec{0},\frac23r)$ for each
$h=1,\ldots ,b$, and $F_{\infty,s}(\be_{l,i}(s))\subset
\partial\B(x_i,r_i(s))$ for some $i=1,\ldots ,J$ and
for every $l=1,\ldots ,\wt{e}_i$.
\end{enumerate}

\item There exists an infinite
strictly increasing sequence
\[
\mathfrak{S}=\{s_1,s_2,\ldots,s_j,\ldots\}\subset \N
\]
such that for each $j\in \N$, for $n$ sufficiently large depending on $j$,
\[
A_i(n,s_{j+1})\subset \Int(A_i(n,s_{j}))\; \mbox{\rm and } \;
\Ind(A_i(n,s_{j}))=\Ind(A_i(n,s_{1})).
\]
In particular, for each $j\in \N$ and $n$ sufficiently large depending
on $j$, $A_i(n,s_{1})\setminus A_i(n,s_{j})$ is stable.

\item For each $s_j\in \mathfrak{S}$ defined in item~5,
$M_{s_{j+1}}\subset M_{s_j}$ and
$F_{\infty,s_{j+1}}|_{M_s}=F_{\infty,s_j}$.
Then $M_{\infty}=\cup_{s_j\in \mathfrak{S}}M_{s_j}$ is a compact
Riemann surface with boundary,
punctured in $e:=\sum_{i=1}^J\wt{e}_i$  points
$\{ P_{1,i},\ldots ,P_{\wt{e}_i,i}\} _{i=1,\ldots,J}$,
and the immersion $F_{\infty}\colon M_{\infty}\la \R^3$ given by
$F_{\infty}|_{M_s}=F_{\infty,s}$ extends to a finitely branched minimal
immersion
\[
\ov{F}_{\infty}\colon M_{\infty}\cup \{ P_{1,i},\ldots ,P_{\wt{e}_i,i}
\} _{i=1,\ldots,J}\la \textstyle{\B(\vec{0},\frac{2}{3}r)},
\]
such that $\ov{F}_{\infty}(\{ P_{1,i},\ldots ,P_{\wt{e}_i,i}\})=\{x_i\}$,
$i=1,\ldots ,J$, and the branch points of $\ov{F}_{\infty}$ are contained
in the set $\{ P_{1,i},\ldots ,P_{\wt{e}_i,i}\ | \ i=1,\ldots,J\}$.

\item For $i\in \{ 1,\ldots ,J\}$ fixed and $\ve >0$ sufficiently small
and fixed, the branching contribution $B_i\in \N\cup \{ 0\}$ to
$\ov{F}_{\infty}$ from $\{ P_{1,i},\ldots ,P_{\wt{e}_i,i}\}$ is
$B_i=S_i-\wt{e}_i$, where
\begin{equation}
S_i=\sum_{h=1}^{\wt{e}_i}m_{h,i}
\label{7.13b}
\end{equation}
is the total spinning of the  boundary curves  of $\wt{F}_n$
restricted to the component
$\Delta(i,n,\ve )$ of $\wt{F}_n^{-1}(B_{\wt{X}_n}(\wt{F}_n(q_i(n)),\ve ))$
containing $q_i(n)$ (for $n$ sufficiently large, $S_i$ is independent
of $n$). Furthermore,
\begin{equation}
S_i\leq 3I(\Delta(i,n,\ve ));
\label{7.13a}
\end{equation}
(here $I(\Delta(i,n,\ve ))$ is the index
of $\Delta(i,n,\ve )$) and so, the total branching of
$\ov{F}_\infty$ is at most $\sum_{i=1}^J(S_i-\wt{e}_i)\leq 3I_0-J$.
\end{enumerate}
\end{lemma}
\begin{proof}
		
Since $[\de_1,\frac{\de}{2}]$ is compact, after replacing by a
subsequence, the sequence $\{ r_{\wt{F}_n}(1)\}_n\subset
[\de_1,\frac{\de}{2}]$
given by Theorem~\ref{mainStructure} converges to a positive number $r\in
[\de_1,\frac{\de}{2}]$. The convergence stated in item~2 of the
multi-graphs
\[
\wt{F}_n(\Delta_1(n))\cap[\ov{B}_{\wt{X}_n}(\wh{F}_n(p_1(n)),
r_{\wt{F}_n}(1))\setminus  B_{\wt{X}_n}(\wh{F}_n(p_1(n)),
\textstyle{\frac{1}{2}}r_{\wt{F}_n}(1))]
\]
to minimal multi-graphs in $\ov{\B}(\vec{0},r)\setminus
\B(\vec{0},\frac{1}{2}r)$
is standard by curvature estimates for CMC graphs.
This gives items~1 and 2 of the lemma.
	
We next prove that item~3 holds. To find the finite set $Q(n)$, we will
proceed as follows. Suppose for the moment that, after
replacing by a subsequence, for each $s\in \N$,  $|A_{\wt{F}_n}|$ is
uniformly bounded in $B_{M_n}(p_1(n),\frac{2}{3}r)	\setminus
B_{M_n}(p_1(n),\frac{2}{3s}r)$.
In this case, the set $Q(n):=\{q_1(n)=p_1(n)\}$ is easily
seen to satisfy item~3 of the lemma with the choice $\ve_1=\frac{2}{3}r$.
Otherwise, after replacing by a subsequence, there exists an $s_1\in \N$
and a point $q_2(n)\in B_{M_n}(p_1(n),\frac{2}{3}r)\setminus
B_{M_n}(p_1(n),\frac{2}{3s_1}r)$ such that $|A_{\wt{F}_n}|(q_2(n))>
\max\{ n,10A_1\}$. If after replacing by a subsequence,
$\{ A_{\wt{F}_n}\}_n$ is uniformly bounded
in $B_{M_n}(p_1(n),\frac{2}{3}r)\setminus
\cup_{i=1}^2B_{M_n}(q_i(n),\frac{2}{3s}r)]$
for each $s\in \N$, then the set $Q(n):=\{q_1(n),q_2(n)\}$
satisfies item~3 of the lemma with
$\ve_1=\frac{1}{2}d_{M_n}(q_1(n),q_2(n))$,
since after replacing by another subsequence, $\wt{F}_n(q_2(n))$ converges
as $n\to \infty $ to some $x_2\in \B(\vec{0},\frac{2}{3}r)$
(note that $x_2$ might be $\vec{0}$).
Continuing inductively, we arrive at two sets of points
\[
Q(n)=\{q_1(n)=p_1(n),q_2(n),\ldots,q_J(n)\},
\quad \{x_1=\vec{0},x_2,\ldots,x_J\}\subset
\B (\vec{0},{\textstyle \frac{2}{3}r}),
\]
satisfying item~3 of the lemma with respect to
\[
\ve_1=\frac{1}{2}\min \{ d_{M_n}(q_i(n),q_j(n))\ | \ i,j=1,\ldots ,J,\
i\neq j\} .
\]
Here, $J\leq I_0$ because the index of $B_{M_n}(p_1(n),\frac{2}{3}r)$
is at most $I_0$. This finishes the	proof of item~3 of the lemma.

Regarding item 4, we make the following two observations:
\begin{enumerate}[(O1)]
\item For each $s\in \N$, there is a uniform upper 
bound $A_2(s)\geq A_0$ on the norm of the
second fundamental forms of the immersions $\wt{F}_n$ restricted to
\[
\bigcup_{i=1}^J\left[ \ov{B}_{M_n}(q_i(n),\textstyle{\frac{\ve_1}{3s}})
\setminus B_{M_n}(q_i(n),\textstyle{\frac{\ve_1}{4s}})\right].
\]
(this follows from the already proven item~3(c) of this lemma).

\item For each $n\in \N$, let $\wh{F}_{n,s}$ be the restriction
of $\wt{F}_n$ to
$\bigcup_{i=1}^J\ov{B}_{M_n}(q_i(n),\textstyle{\frac{\ve_1}{3s}})$.
Then, observation (O1) implies that for each $n\in \N$, $\wh{F}_{n,s}$
lies in the space $\L(I_0, H_0,\ve_1/12s,A_2(s),K_0)$.

\end{enumerate}
We next apply Theorem~\ref{mainStructure} to $\wh{F}_{n,s}\in
\L(I_0, H_0,\ve_1/12s,A_2(s),K_0)$
(which can be applied by the induction hypothesis), from where one has a
corresponding constant $\wh{A}_1(s)\geq A_1$ that replaces the previous
constant $A_1$ and where the choice of $\tau$ is the same as previously
considered. Assume that $n$ is chosen sufficiently large,
so that given $i=1,\ldots ,J$, the point $q_i(n)$ satisfies that the
maximum of the norm of the second fundamental form of
$\wh{F}_{n,s}$ in $\ov{B}_{M_n}(q_i(n),\frac{\ve_1}{3s})$ is achieved at
$q_i(n)$, with value greater than $10\wh{A}_1(s)$
(by item~\ref{it1}(d) of Theorem~\ref{mainStructure}). Another consequence
of Theorem~\ref{mainStructure} applied to $\wh{F}_{n,s}$ is that
for $n$ large and for each $i=1,\ldots ,J$, we have
associated positive numbers
\[
\de_i(s),\de_i(1,s),r_{\wh{F}_{n,s}}(i,s)
\]
with $\de_i(1,s),\de_i(s),r_{\wh{F}_{n,s}}(i,s)$) playing the respective
roles of the related numbers $\de_1,\de,r_F(i)$ in
Theorem~\ref{mainStructure}, where
\begin{equation}\label{5.16}
0<\de_i(1,s)\leq r_{\wh{F}_{n,s}}(i,s)
\leq \de_i(s)/2<\de_i(s)<\frac{2\ve_1}{3s} \quad \mbox{for all $n$.}
\end{equation}
We also have a $\Delta$-type domain
$\Delta_i(q_i(n), r_{\wh{F}_{n,s}}(i))$ defined by
item~\ref{it1} of Theorem~\ref{mainStructure}, that is, $\Delta_i(q_i(n),
r_{\wh{F}_{n,s}}(i))$ is the component of
$\wh{F}_{n,s}^{-1}(\ov{B}_{\wt{X}_n}
(\wh{F}_{n,s}(q_i(n)), r_{\wh{F}_{n,s}}(i)))$ containing $q_i(n)$, so that
the conclusions of Theorem~\ref{mainStructure} hold for these
$\de_i(s),\de_i(1,s),r_{\wh{F}_{n,s}}(i,s),
\Delta_i(q_i(n), r_{\wh{F}_{n,s}}(i))$.
In particular,
\begin{eqnarray}
\Delta_i(q_i(n), r_{\wh{F}_{n,s}}(i))&\subset &
B_{M_n}(q_i(n),\textstyle{\frac{\ve_1}{3s}}),\nonumber
\\
\wh{F}_{n,s}\left[ \partial \Delta_i(q_i(n), r_{\wh{F}_{n,s}}(i))\right]
&\subset &
\partial B_{\wt{X}_n}(\wh{F}_{n,s}(q_i(n)),
r_{\wh{F}_{n,s}}(i)).\label{5.17}
\end{eqnarray}

We next check that the domains and numbers
\[
A_i(n,s):=\Delta_i(q_i(n), r_{\wh{F}_{n,s}}(i)),\qquad
r_i(n,s):=r_{\wh{F}_{n,s}}(i)
\]
satisfy items 4(a)-\ldots -4(e) stated in the lemma.
Item 4(a) follows directly from~\eqref{5.16}, and item 4(b) follows
from~\eqref{5.17}.
Regarding item~4(c), since $A_i(n,s)$ is compact,
its boundary has a finite number
$\wt{e}_i(n,s)\in \N$ of components. By item~\ref{it3} of
Theorem~\ref{mainStructure},
each boundary component of $A_i(n,s)$ admits an annular neighborhood
which is a multi-graph of positive integer multiplicity
$m_{h,i}(n,s)\in \N$ (the index $h$
parameterizes the set of boundary components of $A_i(n,s)$).
The fact that both the
number of such boundary components and the multiplicities
$m_{h,i}(n,s)$ can be considered to be independent of $n,s$
(after passing to a subsequence in $n$)
follow from the fact the $m_{h,i}(n,s)$ are bounded independently of $n$,
which in turn can be deduced from the following inequality
(see item~\ref{it4}(a) of Theorem~\ref{mainStructure}):
\[
m_{h,i}(n,s)\leq 3\, \mbox{Index}(A_i(n,s))\leq 3\, I_0,
\quad \mbox{for all $n$.}
\]
Now, the rest of properties stated in item~4(c) of the lemma are direct consequences of
Theorem~\ref{mainStructure} applied to $\wh{F}_{n,s}$.

The convergence statement in item~4(d) follows from standard curvature
estimates for CMC immersions. The last sentence in item~4(d) follows from
uniqueness of the limit as $n\to \infty$ of the  $\wt{F}_n$ restricted to
$\Delta_1(n,\frac23)\setminus \cup_{i=1}^J
\Delta_i(q_i(n), r_{\wh{F}_n}(i))$
and from the already proven item~2 of this lemma. Item~4(e) holds by
construction, which finishes the proof of item 4 of the lemma.

Regarding item~5, choose $s_1'=1$ and $s_2'\in \N$, $s_2'>1$,
such that $\frac{\ve_1}{3s_2'}<\de_i(1,1)$. Assuming
$s_1',\ldots ,s_j'$ defined, choose  $s_{j+1}'\in \N$ such that
$s_{j+1}'>s_j'$ and $\frac{\ve_1}{3s_{j+1}'}<\de_i(s_j',1)$.
This inequality implies that
$A_i(n,s'_{j+1})\subset \Int(A_i(n,s'_{j}))$ and so,
$\Ind(A_i(n,s'_{j+1}))\leq \Ind(A_i(n,s'_{j}))$.
Since $\Ind(A_i(n,s'_{1}))$ is finite,
there exists $j_0\in \N$ such that
$\Ind(A_i(n,s'_{j}))=\Ind(A_i(n,s'_{j_0}))$ for all $j\geq j_0$.
Now label $s_j=s'_{j+j_0}$ for each $j\in \N$, and item~5 of the lemma
is proved.

By item 4(d), for each $j\in \N$ the restrictions of $\wt{F}_n$ to
$\Delta_1(n,\frac{2}{3})\setminus \cup_{i=1}^JA_i(n,s_j)$
converge smoothly as $n\to \infty $ to a minimal
immersion $F_{\infty,s_j}\colon M_{s_j} \la
\textstyle{\B(\vec{0},\frac{2}{3}r)}$
of a compact surface $M_{s_j}$ with boundary.
Since $A_i(n,s_{j+1})\subset \Int(A_i(n,s_{j}))$
then $M_{s_{j+1}}\subset M_{s_j}$ for each $j$,
and by uniqueness of the limit we have
$F_{\infty,s_{j+1}}|_{M_s}=F_{\infty,s_j}$ for each $j$.

By a standard diagonal argument in $n$ and $s_j$,
the map
\[
F_{\infty}\colon M_{\infty}=\bigcup_{s_j\in \mathfrak{S}}M_{s_j}\la
\textstyle{\B(\vec{0},\textstyle{\frac{2}{3}r})},
\]
given by $F_{\infty}|_{M_{s_j}}=F_{\infty,s_j}$ for each $j\in \N$,
is a minimal immersion with finite area,
defined on a surface $M_{\infty}$ of finite genus: the bound on the
genus of $M_{\infty}$ is the same bound as on the genus of the surface
$\Delta_1(n)$, which, by item~\ref{it4} of Theorem~\ref{mainStructure},
is at most $6I(\Delta_1(n))-8\leq 6I_0-8$
if the index $I(\Delta_1(n))\geq 2$; if $I(\Delta_1(n))=1$,
then the genus of $\Delta_1(n)$ is zero.
Observe that $M_{\infty}$ has at least $J$
annular ends, and the number $e$ of these ends of $M_{\infty}$
is finite (at most $3I_0-1$ by item B of Theorem~\ref{mainStructure}).
Furthermore, the image by $\wh{F}_{\infty}$ of these ends of $M_{\infty}$
is $\{x_1,\ldots,x_J\}\subset \B(\vec{0},\frac{2}{3}r)$.
By regularity results in \cite{gr1}, $M_{\infty}$ is a compact
Riemann surface with $b$ boundary components, punctured in
$e:=\sum_{i=1}^J\wt{e}_i$ points,
and we can denote the set of ends of $M_{\infty}$ by
\[
\{ P_{1,i},\ldots ,P_{\wt{e}_i,i}\} _{i=1,\ldots,J},
\]
in such a way that the immersion $F_{\infty}$ extends to a finitely
branched minimal immersion
\begin{equation}
\label{eq:Finfty}
\ov{F}_{\infty}\colon M\cup \{ P_{1,i},\ldots ,P_{\wt{e}_i,i}
\} _{i=1,\ldots,J}\la \textstyle{\B(\vec{0},\frac{2}{3}r)},
\end{equation}
such that $\ov{F}_{\infty}(\{ P_{1,i},\ldots ,P_{\wt{e}_i,i}\})=\{x_i\}$,
$i=1,\ldots ,J$, and by construction, the set of branch points
of $\ov{F}_{\infty}$ is contained in the set
$\{ P_{1,i},\ldots ,P_{\wt{e}_i,i}\ | \ i=1,\ldots,J\}$.
This proves item~6 of the lemma.

Finally, we prove item 7. Observe that the branching order
$B(P_{h,i})\in \N\cup \{0\}$ of
$\ov{F}_{\infty}$ at $P_{h,i}$ equals
\begin{equation}
B(P_{h,i})=m_{h,i}-1,
\label{7.14}
\end{equation}
where $m_{h,i}\in \N$ is the multiplicity defined in item~4(c) above.
By adding this in the
set $\{ P_{1,i},\ldots ,P_{\wt{e}_i,i}\}$,
we deduce that the branching contribution
$B_i\in \N\cup \{ 0\}$ to $\ov{F}_{\infty}$ from this set is
$B_i=S_i-\wt{e}_i$, where $S_i=\sum_{h=1}^{\wt{e}_i}m_{h,i}$,
and thus, \eqref{7.13b} is proved. Finally,
the estimate~\eqref{7.13a} for the total spinning of $\Delta(i,n,\ve)$
(for a sufficiently small $\ve>0$) follows from
item~B of Theorem~\ref{mainStructure}.
This finishes the proof of the lemma.
\end{proof}

We now come back to item (M2) above.
Using the notation in Lemma~\ref{lemma7.11},
suppose, after choosing a subsequence,
that the $\frac{1}{r_n}F_n$ restricted to
$\Delta_{2,n}$ converge to a family
\begin{equation}\label{5.21}
F_{\infty}\colon \wh{\cD}\la \B(1)
\end{equation}
of minimal disks branched at the origin as described in
Lemmas~\ref{disks} and~\ref{lemma7.11}.
Thus, the desired (global) limit $f_2\colon \S_2\la \R^3$ of the
$\frac{1}{r_n}F_n$ is already
constructed in a neighborhood of $\wh{\cS}(0)$ in $\wh{\mathcal{D}}$
(see equation~\eqref{S(0)}),
where a non-trivial part of the index of $\frac{1}{r_n}F_n$ is collapsing
(namely, this collapsing index is $I(f_1)>0$);
since the remaining index of
$\frac{1}{r_n}F_n$ is at most $(I_0+1)-I(f_1)\leq I_0$, we are
allowed to apply Lemma~\ref{lemma7.10} to $\frac{1}{r_n}F_n$. We next
make this paragraph and the previously alluded to global convergence
in item~(M2) rigorous.

\begin{proposition}
\label{prop7.12}
In the situation above, let $\wt{F}_n\colon M_n\la \wt{X}_n$ be
$\frac{1}{r_n}F_n\colon M_n\la
\frac{1}{r_n}X_n$. Then, after replacing by a subsequence, there exist
$R_0 \geq 10,\ve_2\in (0,\de_1]$ and a
collection of points $Q_2(n)=\{q_1(n)=p_1(n),q_2(n),\ldots ,q_J(n)\}
\subset  B_{M_n}(p_1(n),R_0)$, $J\leq I_0$, such that:

\begin{enumerate}
\item For any $R>R_0$, $\{ |A_{\wt{F}_n}|\} _n$ is 
uniformly bounded in
$B_{M_n}(p_1(n),R)\setminus B_{M_n}(p_1(n),R_0)$.

\item $d_{M_n}(q_i(n),q_j(n))\geq \ve_2$ \, for each $n\in \N$ and
$i\neq j\in \{1, 2,\ldots ,J\}$.

\item For each $i\in \{ 1, 2,\ldots ,J\}$ and
$m\in \N$ with $\frac1m<\ve_2$,  
\[
|A_{\wt{F}_n}|(q_i(n))>n=\max \{|A_{\wt{F}_n}|(x) : x\in B_{M_n}(q_i(n),1/m)\}
\]
and there exists $A_2(m)>1$  such that $|A_{\wt{F}_n}|<A_2(m)$ in
$B_{M_n}(p_1(n),R_0)\setminus \cup_{i=1}^JB_{M_n}(q_i(n),1/m)$.

\item There exist (not necessarily distinct) points $x_1=\vec{0},x_2,
\ldots ,x_J\in \B(\vec{0},R_0)$ (here $\B(\vec{0},R)$ denotes
the ball  centered the origin with radius $R>0$ in $\R^3$ with its flat
metric) such that when viewed in harmonic  coordinates in $\wt{X}_n$
centered at $\wt{F}_n(p_1(n))$, the points $\wt{F}_n(q_i(n))$ converge as
$n\to \infty $ to $x_i$, for each $i=1,2,\ldots ,J$.

\item For almost all $R>R_0$ and for $m$ sufficiently large, the
$\wt{F}_n$ restricted to
\[
{\textstyle \overline{B}_{M_n}(p_1(n),R)\setminus
\cup_{i=1}^JB_{M_n}(q_i(n),\frac{1}{m})} 
\]
converge smoothly as $n\to \infty $ to a minimal immersion
$F_{\infty,m,R}\colon M_{m,R} \la \ov{\B}(\vec{0},R)$
of a compact surface with boundary $M_{m,R}$. Furthermore,
$M_{m,R}\subset M_{m+1,R'}$ and
$F_{\infty,m+1,R'}|_{M_{m,R}}=F_{\infty,m,R}$
whenever $R'>R>R_0$.
\item Define
\[
\S_2^*:=\bigcup_{\stackrel{m\in \N}{R>R_0}}M_{m,R},\qquad
f_2^*\colon \S_2^*\la \R^3,\
f_2^*|_{M_{m,R}}=F_{\infty,m,R}.
\]
Then, $\S_2^*$ is a (possibly disconnected) open Riemann surface and
$f_2^*$ is a minimal immersion. Furthermore, the conformal completion
$\ov{\S}_2$ of $\S_2^*$ has the structure of a compact Riemann surface,
$\ov{\S}_2\setminus \S_2^*=\cS(f_2)\cup \mathcal{E}_2$ is a finite set,
and $f_2^*\colon \S_2^*\la \R^3$ extends through $\cS(f_2)$
to a finitely branched, complete minimal immersion
$f_2\colon \S_2=\S_2^*\cup \cS(f_2)\la \R^3$
with finite total curvature, where
\begin{enumerate}[(a)]
\item $\cS(f_2)$ is the disjoint union of the finite set
$\cS(\vec{0})=\{ P_{1,1},\ldots P_{e_1,1}\}
\subset f_2^{-1}(\{ x_1=\vec{0}\})$
that appears in Lemma~\ref{lemma7.11}, together with the closely related finite sets
\[
\cS(x_i)=\{ P_{1,i},\ldots ,P_{b_i,i}\}\subset f_2^{-1}(\{x_i\}),
\quad i=2,\ldots J.
\]
Furthermore, the set of branch points
of $f_2$ is contained in $\cS(f_2)$
and its branch locus (image) is contained 
in $\{x_1=\vec{0},x_2,\ldots ,x_J\}\subset \B(R_0)$.
\item The set of ends of $f_2$ is $\cE_2=\{ E_1,\ldots ,E_{e_2}\}$.
\item The map $F_{\infty}$ given in~\eqref{5.21} coincides with $f_2$ in a
neighborhood of $\cS(x_1=\vec{0})$ in $\S_2$.
\end{enumerate}

\item The total
branching order $B(f_2)$ of $f_2$ can be estimated from above as follows:
\begin{equation}
B(f_2)\leq 3[I_0+1-\Ind(f_1)]-J\leq 3I_0-1.
\label{7.21}
\end{equation}
\item The following properties hold for some $R> 3R_0$:
\begin{enumerate}
\item The index of $f_2^{-1}( \B(R/3))$ is $I(f_2)$ (compare to property
(H0)' above).
\item $f_2(\Sigma_2)\setminus \B(R/3)$ consists of $e_2$ multi-graphs
over their projections to planes $\Pi_j\subset \R^3$ passing though
$\vec{0}$, $j=1,\ldots ,e_2$ (compare to property (H1)').
Furthermore,  each of these end representatives contains no non-trivial
geodesic arcs with boundary points in the boundary
of  $\S_2\setminus f_2^{-1}(\B(R/3))$.

\item The image through the Gauss map of $f_2$ of each component $C_j$ of
$f_2(\S_2)\setminus \B(R/3)$ is contained in the spherical
neighborhood of radius $\a_1/2$ centered at a point $v_j\in \esf^2(1)$
perpendicular to $\Pi_j$, where $\a_1=\a_1(\tau)\in (0,\tau]$ is the
constant given by Lemma~\ref{lemma5.3} for $L_0=3\pi (I_0+2)+1$
(thus, $C_j$ satisfies item (B2)  of Lemma~\ref{lemma5.3}
with $R_1=R/3$ and $\a=\a_1/2$, compare to item (H2)').

\item $f_2(\Sigma_2)$ makes an angle greater than
$\frac{\pi }{2}-\frac{\a_1}{2}$
with every sphere $\esf^2({r})$ of radius $r\geq R/3$ centered at the
origin (so, $C_j$  satisfies item (B1) of Lemma~\ref{lemma5.3} with
$R_1=R/3$ and $\a=\a_1/2$, compare
to (H3)').

\item The total length of the intersection of $f_2(\Sigma_2)$ with
any sphere $\esf^2({r})$ centered at the origin and radius $r\geq R/3$
is less than $(L_0-\frac{1}{2})r$ (hence $C_j$ satisfies item (B3) of
Lemma~\ref{lemma5.3} with $R_1=R/3$,
compare to (H4)').

\item  For all $n\in \N$, the component
$\Delta_{2,n}(R/3)$ of $\wt{F}_n^{-1}(B_{\wt{X}_n}(\wt{F}_n(p_1(n)),R/3))$ that
contains $p_1(n)$ has index at least $I(f_1)+I(f_2)+(J-1)$ and
if $J=1$, then $I(f_2)>0$. In particular,
$I(\Delta_{2,n}(R/3))>I(f_1)$.
\end{enumerate}
\end{enumerate}
\end{proposition}
\begin{proof}
Recall the notation and statement of Lemma~\ref{lemma7.11}.
By assumption, the $\wt{F}_n$ restricted  to $\Delta_{2,n}$ converge to
$F_{\infty}$ given by equation~\eqref{5.21}. Since the restriction of
$F_{\infty}$ to $F_{\infty}^{-1}(\B(1)\setminus \B(\frac{1}{2}))$ consists
of $e_1$ multi-graphs (here $e_1$ is the number of ends of $f_1$),
then $\wt{F}_n(\Delta_{2,n})$ is graphical in the region
$B_{\wt{X}_n}(\wt{F}(p_1(n)),1)\setminus B_{\wt{X}_n}(
\wt{F}(p_1(n)),\frac12)$, and thus, the surfaces
\[
{\textstyle M'_n=M_n\setminus [\Delta_{2,n}\cap
\wt{F}_n^{-1}(B_{\wt{X}_n}(\wt{F}(p_1(n)),\frac12)]}
\]
have uniform curvature estimates  in a fixed sized $\ve'_0$-neighborhood
of its boundary (for some $\ve_0'\in (0,\ve_0]$).
Let $F'_n\colon M_n'\la \wt{X}_n$ be the restriction of $\wt{F}_n$ to
$M_n'$ and for all $n\in \N$, we can consider $F_n'$ to be an element
in a fixed related space $\Lambda '$ except that
the index of the immersions in
\[
\Lambda '=\L(I_0,H_0,\ve_0',A_0,K_0)
\]
is at most $I_0$.
By induction, we can suppose that Theorem~\ref{mainStructure} holds for
the subspace $\Lambda'$.

The construction of the finite set $\{q_2(n),\ldots ,q_J(n)\}\subset
B_{M_n}(p_1(n),R_0)$, $J\leq I_0$, appearing in the statement of
the proposition, follows exactly the same arguments used to prove the
existence of the related set $Q(n)$ given in item~3 of
Lemma~\ref{lemma7.10}.
Similarly, items 2, 3 and 4 of the proposition can be deduced
from the same reasoning as items~3(b), 3(c) and 3(d)  of
Lemma~\ref{lemma7.10} respectively; in particular, we use the
number $\de_1\in (0,\ve_0/2]$ defined in item~1 of Lemma~\ref{lemma7.10}
in order to find $\ve_2\in (0,\de_1]$ satisfying item~2 of the
proposition. We leave the details to the reader.

The existence of the number $R_0\geq 10$ and item 1 of
the proposition follows from the fact that the number $J$ of sequences
\[
\{ q_1(n)=p_1(n)\}_n,\ldots ,\{ q_J(n)\}_n
\]
around which the second fundamental form of $\wt{F}_n$ fails to be
bounded, is finite (at most $I_0+1$, by Lemma~\ref{ass4.4}).

Items 5 and 6 of the proposition also follow with small modifications from
the proof of items~4(d) and 6 of Lemma~\ref{lemma7.10},
where one also uses the fact that a complete minimal surface in $\rth$
with compact boundary and finite index has finite total curvature
(see \cite{fi1} for this result when the surface is orientable
and see the last paragraph of the proof of
Theorem~17 in Ros~\cite{ros9} for the non-orientable case).
The proof of item~7 of the proposition
follows from the same arguments that proved item~7 of
Lemma~\ref{lemma7.10},
observe that the index of $f_2$ is at most $(I_0+1)-\mbox{Index}(f_1)$.

The proof of items~8(a) and of the first statement of 8(b) are clear
after taking $R>0$ sufficiently large, since $f_2$ has finite total
curvature. The second statement of 8(b) follows from the fact that
for $R>0$ sufficiently large, the collection of ends
$f_2^{-1}(\R^3\setminus \B(R/3)) $ of $f_2$ is
foliated by the simple closed curves
in $\{f_2^{-1}(\partial \B(R')) \mid R'\geq R/3)\}$, each of which has
positive geodesic curvature.
The proof of parts (c) -- (e) of item~8 also follow from previous
considerations (compare to items (H2)'--(H4)').

To finish the proof of the proposition, we check that item~8(f) holds.
First suppose that $J=1$.  In this case, the sequence
$\{ \frac{1}{r_n}F_n\}_n$
converges smoothly (up to a subsequence) to $f_2$ in a neighborhood
of $\partial\B(1)$. This implies, by construction of $r_n$
(see Definition~\ref{defrn}), that $f_2$ is not flat in any neighborhood
of $\partial\B(1)$. In particular, $f_2$ is not flat and the
image of its branch locus is the origin. Then, by item~1 of
Lemma~\ref{lema3.6}, $f_2$ has positive index.

Regardless of the value of $J$, and
by the already proven item 8(a) of this proposition,
the index of $f_2^{-1}( \B(R/3))$ is $I(f_2)$.
Since the index of a compact minimal surface with boundary
remains the same after removing a sufficiently small neighborhood of
a finite subset of its interior, we deduce that for
$m$ sufficiently large, the index of
\begin{equation}
f_2^{-1}( \B(R/3))\setminus [\B(\vec{0},\textstyle{\frac{1}{m}})
\cup (\cup_{i=2}^J\B(x_i,\textstyle{\frac{1}{m}}))]
\label{f2Rm}
\end{equation}
is also equal to  $I(f_2)$. Let $\Delta_{2,n}(R/3)$ be the component
of $\wt{F}_n^{-1}(B_{\wt{X}_n}(\vec{0},R/3))$ that contains $p_1(n)$.
By the convergence in item~5 of the proposition,
for $m\in \N$ sufficiently large,
the index of
\[
\Delta_{2,n}^*(R/3):=\Delta_{2,n}(R/3)\setminus
[B_{M_n}(p_1(n),\textstyle{\frac{1}{m}})
\cup (\cup_{i=2}^JB_{M_n}(q_i(n),\textstyle{\frac{1}{m}}))]
\]
is equal to the index of the surface in~\eqref{f2Rm}.
Observe that for $n$ sufficiently large and $m$ large and fixed,
that index of $B_{M_n}(p_1(n),\frac{1}{m})$ is equal to the
index $I(f_1)$ of $f_1$, and each of the balls in the pairwise disjoint
collection
\[
\{ B_{M_n}(p_1(n),\textstyle{\frac{1}{m}}),
B_{M_n}(q_2(n),\textstyle{\frac{1}{m}})),
\ldots,B_{M_n}(q_J(n),\textstyle{\frac{1}{m}}))  \}
\]
is unstable. Then, if we denote by $I(S)$ the Morse index of
a surface $S$, we get (after replacing by a subsequence)
\begin{eqnarray}
I(\Delta_{2,n}(R/3))
&\geq &I(\Delta_{2,n}^*(R/3))+
I(B_{M_n}(p_1(n),\textstyle{\frac{1}{m}}))+
\displaystyle{\sum_{i=2}^J}I(B_{M_n}(q_i(n),\textstyle{\frac{1}{m}}))
\nonumber
\\
&=&I(f_2)+I(f_1)+
\displaystyle{\sum_{i=2}^J}I(B_{M_n}(q_i(n),\textstyle{\frac{1}{m}})).
\label{7.23}
\end{eqnarray}
If $J=1$, then the last sum is empty and~\eqref{7.23} gives
$I(\Delta_{2,n}(R/3))\geq I(f_2)+I(f_1)>I(f_1)$ as desired. Finally,
if $J\geq 2$, then we estimate each
$I(B_{M_n}(q_i(n),\textstyle{\frac{1}{m}}))\geq 1$ and so,
\eqref{7.23} gives $I(\Delta_{2,n}(R/3))\geq I(f_2)+I(f_1)+(J-1)$.
This completes the proof.
\end{proof}

\begin{lemma} \label{lemma7.13}
With the notation of Proposition~\ref{prop7.12} consider the partition of
$\cS(f_2)\subset \S_2$ by the subsets
\[
\cS(f_2,i)=\cS(x_i), \ i=1,\ldots ,J
\]
introduced in item~6(a) of that proposition.
Define the quotient space $\wh{\S}_2$ of $\S_2$ where each of the
elements in $\cS(f_2,i)$ identifies to one point that we denote by
$\wh{\cS}(f_2,i)\in \wh{\S}_2$, $i=1,\ldots ,J$
and every other point of $\S_2$ only identifies with itself. Let
\[
\pi \colon \S_2 \to \wh{\S}_2
\]
be the related quotient map, that is, $\pi |_{\cS(f_2,i)}$ is the
constant map equal to $\wh{\cS}(f_2,i)$,
and the restriction of $\pi $ to $\S_2\setminus \cS(f_2)$
is injective. After endowing $\wh{\S}_2$ with the quotient topology,
the following hold:
\ben
\item $\wh{\S}_2$ is a path-connected topological space and
\[
\wh{\cS}(f_2):=\pi(\cS(f_2))
\]
consists of $J$ elements in $\wh{\S}_2$.

\item  $\wh{\S}_2\setminus \wh{\cS}(f_2)$ is a smooth Riemannian surface
that induces a  metric space structure $d_{\wh{\S}_2}$ on $\wh{\S}_2$.

\item The restriction of $f_2$ to $\S_2\setminus \cS(f_2)$,
considered to be a subset of $\wh{\S}_2$,
extends to a continuous mapping $\wh{f}_2\colon \wh{\S}_2 \to \rth$.

\item Let $p=\wh{\cS}(f_2,1)$ (so $\wh{f}_2(p)=\vec{0}$). Given $q\in
\wh{f}_2^{-1}(\B(R))$ a point different from $p$, where $R>0$ was defined
in item~8 of Proposition~\ref{prop7.12},
there is an injective  continuous path
$\a_{p,q}\colon [0,1] \to \wh{\S}_2$ of least length joining $p$ to $q$
satisfying:
\ben
\item $\wh{f}_2\circ \a_{p,q}$ is a piecewise smooth curve in $\rth$
with image in the ball $\ov{\B}(R)$.

\item $\a_{p,q}([0,1])\setminus \wh{\cS}(f_2)$ consists of $j_1(q)\leq J$
smooth geodesic arcs in $\wh{\S}_2\setminus \wh{\cS}(f_2)$,
each of which has length less than $\wh{C}R$,
where $\wh{C}=\wh{C}(I_0,B)>0$ is defined in item~2 of
Proposition~\ref{propos5.5} and $B$ is the total branching order
of $f_2$ (recall that $B\leq 3I_0-1$ by~\eqref{7.21}).

\item In particular, as $j_1(q)\leq J\leq I_0$, then (compare to
item~(H5)' above)
\begin{equation}
d_{\wh{\S}_2}(p,q)< I_0\wh{C} R.
\end{equation}
\een
\een
\end{lemma}
\begin{proof}
The path-connectedness of $\wh{\S}_2$ follows immediately
from the fact that for all $R>R_0$ (this $R_0$ is defined
in Proposition~\ref{prop7.12}),
$B_{M_n}(p_1(n)),R)$  is path-connected with $\cS(f_2)\subset
B_{M_n}(p_1(n)),R_0)$ and because the projection of a continuous path
in $\S_2$ to $\wh{\S}_2$ is a continuous path.
This proves that item~1 holds;
the proof of items~2 and 3 follow from the definition of the quotient
space $\wh{\S}_2$ and the fact that the
composition of continuous mappings is continuous.

The existence of the embedded minimizing geodesic $\a_{p,q}$ joining $p$
to $q$ is standard, where $\a_{p,q}\setminus \wh{\cS}(f_2)$ 
consists of a finite number $j_1(q)\leq J$ of open
geodesic arcs that have least-length joining their endpoints;
the reason that there are at most $J$ such arcs in $\a_{p,q}$
follows from the fact that if there is more
than one such geodesic arc in $\a_{p,q}$, then each such arc contains a
point of $ \wh{\cS}(f_2)\setminus \{p\}$.
Clearly $f_2\circ\a_{p,q}$ is a piecewise smooth curve in $\rth$
and its image is contained in $\ov{\B}(R)$ by the second statement in
item 8(b) of Proposition~\ref{prop7.12}, which completes
the proof of item~4(a).

Since $\a_{p,q}$ is injective, length-minimizing
and only fails to be smooth at points in $\wh{S}(f_2)$, then
$\a_{p,q}([0,1])\setminus \wh{\cS}(f_2)$ consists of $j_1(q)\leq J$
smooth geodesic arcs in $\wh{\S}_2\setminus \wh{\cS}(f_2)$.
Item~4(b) follows directly  from item~2 of
Proposition~\ref{propos5.5} (note that $I(f_2)\leq I_0$ by item~8(f) of
Proposition~\ref{prop7.12} since $I(f_1)>0$)).
As $j_1(q)\leq J$ and $J\leq I_0$ by Proposition~\ref{prop7.12},
then item~4(c) is proved.
\end{proof}

\subsubsection{Finding an $s_0$-th local picture with a uniform size}
\label{sec7.5.3}
Recall that in Definition~\ref{defrn} we introduced $r_n$ in terms of
$\l_{1,n}:=\l_n$ and a certain $R>0$
given in terms of the limit immersion $f_1$ so that hypotheses (B1),
(B2), (B3) of Lemma~\ref{lemma5.3}
hold for annular portions of the $F_n$ with the choices
$L_0=3\pi(I_0+2)+ 1$. We now proceed
in a similar manner replacing $f_1$ by $f_2$ and $F_n$ by
$\wt{F}_n=\frac{1}{r_n}F_n$. Properties~8(b)-8(e)
of Proposition~\ref{prop7.12} for $f_2$ are similar to properties
(H1)'-(H4)' for $f_1$. Recall that these
properties (H1)'-(H4)' produce related properties (I1)'-(I4)' for
$\l_nF_n$ and $n\in \N$ large, in particular,
we found $e_1$ multi-graphical annuli $\wt{G}_n(1),\ldots,\wt{G}_n(e_1)$
in $(\l_nF_n)(\Delta_n(4R)\setminus \Delta_n(R/2))$,
see property (I1)'. We now call $\l_{2,n}=\frac{1}{r_n}$
for each $n\in \N$, which tends
to $\infty$ as $n\to \infty$ by (2.B) in Remark~\ref{rem2.15}.
Reasoning analogously as we did with the first limit $f_1$,
properties~8(b)-8(e) of Proposition~\ref{prop7.12}
produce corresponding properties (I1)'-(I4)' for $\l_{2,n}F_n$ and
$n\in \N$ large, in particular, we find $e_2$ multi-graphical annuli
$\wt{G}_{2,n}(1),\ldots,\wt{G}_{2,n}(e_2)$ in
$(\l_{2,n}F_n)(\Delta_n(4R)\setminus \Delta_n(R/2))$ (this $R>0$
is now introduced in item~8 of Proposition~\ref{prop7.12}).
\begin{definition}
{\rm
Define $r_{2,n}$ as the supremum of the extrinsic radii $r\geq
4R/\l_{2,n}$ such that annular enlargements $\wh{G}_{2,n}(j)$ of the
$\wt{G}_{2,n}(j)$ satisfying conditions (B1),(B2),(B3)
of Lemma~\ref{lemma5.3} for the choices $L_0=3\pi (I_0+2)+1$,
inner extrinsic radius $R_1=\frac{R}{2\l_{2,n}}$,
outer extrinsic radius $R_2=r_{2,n}$ and angle $\a=\a_1$.
}
\end{definition}

As we did in Remark~\ref{rem2.15}, we next discuss  whether or not
$r_{2,n}$ tends to zero as $n\to \infty$. If $\{ r_{2,n}\}_n$ is bounded
away from zero with this bound independent of the sequence
$\{ F_n\}_n\subset \Lambda$, then Proposition~\ref{ass3.5'}
below holds with $s_0=2$. Otherwise, we repeat the process in steps (M1)
and (M2) above for the sequence $\frac{1}{r_{2,n}}F_n$ and find a
complete, finitely branched minimal immersion $f_3\colon \S_3\la \R^3$
with finite total curvature which is a limit of (a subsequence of) the
$\l_{3,n}F_n$, where $\l_{3,n}=\frac{1}{r_{2,n}}$ for each $n\in \N$.
This process of finding scales $\{ \l_{s,n}\}_n$ and limits $f_s$
($s=1,2,\ldots $) must stop after a finite
number $s_0$ of times ($s_0\leq I_0+1$), because each time we apply the
process we find $\Delta$-type components
in (a subsequence of) $\{ F_n\}_n$ with strictly larger index
by item~8(f) of Proposition~\ref{prop7.12}, but the index
of each $F_n$ is at most $I_0+1$. This implies that $r_{s_0,n}$ is
bounded away from zero, with the lower bound being independent of the
sequence $\{F_n\}_n\subset \Lambda$. In this setting, the discussion in
item~2(A) of Remark~\ref{rem2.15} implies that Proposition~\ref{ass3.5}
holds for the scale of $f_{s_0}\colon \Sigma_{s_0}\la \R^3$.
More precisely:

\begin{proposition}
\label{ass3.5'}
There exists $\de_4\in (0,\de_3]$ (this $\de_3\in (0,\de_2]$ was given in
Definition~\ref{def3.6} for the choices $m=3(I_0+1)+3$
and $L_0=3\pi (I_0+2)+1$) such that the hypotheses of
Lemma~\ref{lemma5.3} hold for annular
enlargements $\wh{G}_{s_0,n}(j)$ of the multi-graphs $\wt{G}_{s_0,n}(j)$
(here $j=1,\ldots ,e_{s_0}$ with $e_{s_0}$ being the number of ends of
$f_{s_0}$) between the geodesic spheres in $X$
centered at $F_n(p_1(n))$ of extrinsic inner radius
$\frac{R_{s_0}}{2\l_{s_0}(n)}$
and extrinsic outer radius $\de_4$, and with the choice $\a =\tau_1$
for hypotheses (B1), (B2)
(this $\tau_1\in (0,\a_1]$ was also introduced in
Definition~\ref{def3.6}).
\end{proposition}

With Proposition~\ref{ass3.5'} at hand, we define
\begin{equation}
\label{de,de1}
\de:=\de_4/2,\quad \de_1=\de/2,
\end{equation}
(here $\de_4\in (0,\de_3]$ is given by Proposition~\ref{ass3.5'}).
We are now ready to achieve the main goal of Section~\ref{sec2.7}.
\begin{proposition}
	\label{propos7.17}
Items 1, 2, 3 of Theorem~\ref{mainStructure} hold in the case $I=I_0+1$
for immersions in $\Lambda _{t}$, for some
$t >\wh{C}_s(\de_1/2)$ sufficiently large.
\end{proposition}
\begin{proof}
The idea is to adapt appropriately the arguments at the end of
Section~\ref{sec5.4.2} (after Definition~\ref{def5.4}).
Pick a smallest $R_{s_0}>0$ so that items (H0)'-\ldots -(H4)' hold
with $f$ replaced by $f_{s_0}$ and with the same value
$L_0=3\pi (I_0+2)+1$
(also see items~8(b)-\ldots -8(e) of Proposition~\ref{prop7.12}
for the particular case $s_0=2$).
In particular, item (H5)' can be also adapted to $f_{s_0}$ after applying
the estimate \eqref{eq:lemma5.53} in Proposition~\ref{propos5.5} with
$I=I_0+1$ and $B=B(f_{s_0})$ (this is the total branching order of
$f_{s_0}$, which satisfies $B(f_{s_0})\leq 3I_0-1$ by~\eqref{7.21});
equivalently, we can adapt item~4(c) of Lemma~\ref{lemma7.13}
to $f_{s_0}$ and conclude the following estimate:
\begin{enumerate}[(H5)'']
\item Given $R\geq R_{s_0}$, the intrinsic distance in the pullback
metric by $f_{s_0}$ from $\vec{0}\in \Sigma_{s_0}$ to any point in the
boundary of $f_{s_0}^{-1}(\ov{\B}(R))$ is at most $a(I_0)R$,
where $a(I_0)>0$ can be bounded from above depending only on $I_0$;
in fact,
\[
a(I_0)\leq I_0\, \wh{C}(I_0,B(f_{s_0})),
\]
where $\wh{C}$ is defined in item~2 of Proposition~\ref{propos5.5}.
\end{enumerate}	

Define ${\Delta}_{s_0,n}(R_{s_0})\subset M_n$ as the component of
\[
(\l _{s_0,n}F_n)^{-1}\left( \l_{s_0,n}\ov{B}_X( F_n(p_1(n)),
{\textstyle \frac{R_{s_0}}{\l_{s_0,n}}})\right)
\]
that contains $p_1(n)$. Reasoning as when we deduced (I5)'
from (H5)' and (J5)' from (H5)', we have the
following adaptation of (J5)' to this setting:
\begin{enumerate}[(J5)'']
\item The intrinsic distance in the pullback metric by $F_n$ on $M_n$,
from $p_1(n)$ to the boundary of $\Delta_{s_0,n}(R_{s_0}/2)$ is at most
$\frac{R}{\l_{s_0,n}}[a(I_0)+1]$ (here, $a(I_0)$ is introduced in (H5)''
above).
\end{enumerate}

Take $t$ large enough such that
\begin{enumerate}[(K1)']
\item $\frac{R_{s_0}}{t} [a(I_0)+1] \leq \frac{\de_1}{10}$.
\item The description in items (J1)'-\ldots -(J5)' holds for $F_n$ with
$e=e_{s_0}$ being the number of ends of $f_{s_0}$ and
$L_0=3\pi (I_0+2)+1$.
\end{enumerate}

Define $A_1:=t$ and $r_F(1):=\de_1$.

Given $(F\colon M\la X)\in \Lambda _{t}$,
take a point $p_1\in U(\partial M,\ve_0,\infty )$
where the maximum of $|A_M|$ in $M$ is achieved. Define
$\Delta_1$ to be the component of
$F^{-1}(\ov{B}_X(F(p_1),r_F(1))$ that contains $p_1$, see
Figure~\ref{fig4}.
\begin{figure}[h]
\begin{center}
\includegraphics[width=11cm]{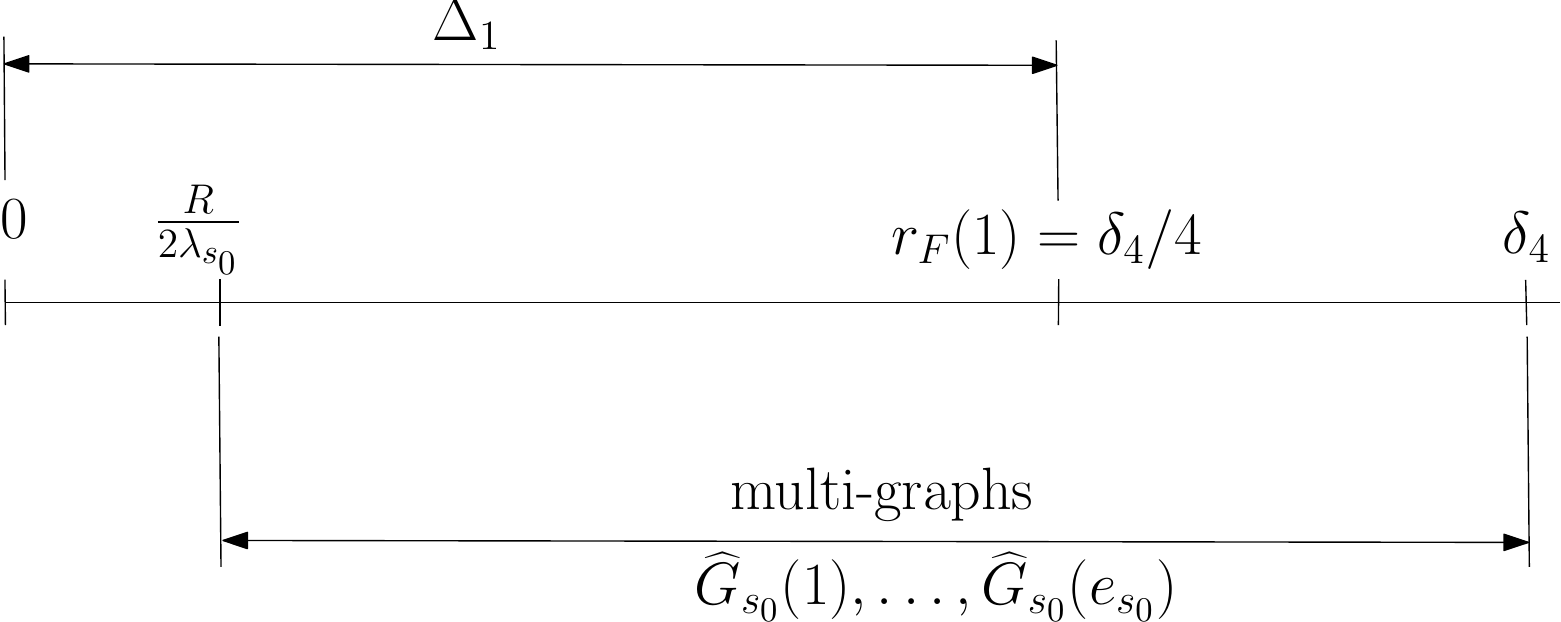}
\caption{Schematic (non-proportional) representation of the extrinsic
geometry of an immersion $(F\colon M\la X)\in \Lambda_{t}$
around a point $p_1$ where the maximum of $|A_M|$ in $M$ is achieved.
Here, $\l_{s_0}>0$ is a large number ($\l_{s_0}\leq \max |A_M|$)
that is the scale of the local picture $f_{s_0}$ of $F$ around $p_1$
that appears in Proposition~\ref{ass3.5'}. Horizontal distances
in the figure represent extrinsic distances in $X$ measured from $F(p_1)$;
for example, $\Delta _1$ has its boundary at extrinsic distance $r_F(1)$
from $F(p_1)$. In the range of extrinsic radii between
$\frac{R}{2\l_{s_0}}$ and $\de_4$
($\de_4$ is fixed and given by Proposition~\ref{ass3.5'}),
$F$ consists of $e_{s_0}$ multi-graphical annuli $\wh{G}_{s_0}(1),
\ldots ,\wh{G}_{s_0}(e_{s_0})$, where $e_{s_0}$ is the number of ends of
$f_{s_0}$. A similar representation holds around relative maxima
$p_{j+1}$ of $|A_M|$ in	$M\setminus (\Delta_1\cup \ldots \cup
\Delta _j)$.}
\label{fig4}
\end{center}
\end{figure}

Next we prove item~\ref{it1}(a) of Theorem~\ref{mainStructure} in the
case $I=I_0+1$ for $\Delta _1$.
Let $q$ be any point in $\partial \Delta _1$. Then,
arguing similarly as in the case $I=1$, we have,
using  $S_F(\frac{R_{s_0}}{2t})$ to denote the extrinsic geodesic
sphere in $X$ centered at $F(p_1)$ with radius $\frac{R_{s_0}}{2t}$, that

\[
\begin{array}{rclr}
d_M(p_1,q)&\leq &
{\displaystyle \max_{x\in \partial \Delta_1\cap
F^{-1}(S_F(\frac{R_{s_0}}{2t}))}}d_M( p_1,x)+d_M(\Delta_1\cap
F^{-1}(S_F(\frac{R_{s_0}}{2t})),q) &
\\
\rule[-.3\baselineskip]{0pt}{.5cm}
& \leq &
\frac{1}{t}[a(I_0)+1]R_{s_0}+d_M(\Delta_1\cap
F^{-1}(S_F(\frac{R_{s_0}}{2t})),q)  & \mbox{(by (J5)'')}
\\
& \leq &
\frac{1}{t}[a(I_0)+1]R_{s_0}+\sqrt{1+\frac{\tau^2}{3}}(r_F(1)
-\frac{R_{s_0}}{2t})  & \mbox{(by Lemma~\ref{lemma5.3})}
\\
\rule[-.3\baselineskip]{0pt}{.5cm}
&\leq & \frac{\de_1}{10}+\sqrt{1+\frac{\tau^2}{3}}r_F(1)
& \mbox{(by (K1)')}
\\
\rule[-.3\baselineskip]{0pt}{.5cm}
&=&\left( \frac{1}{10}+\sqrt{1+\frac{\tau^2}{3}}\right) r_F(1) &
\\
& < & \frac{5}{4}r_F(1). & \mbox{(because  $\tau\leq \pi/10$)}
\end{array}
\]

This proves that item 1(a) of Theorem~\ref{mainStructure} holds in the
case $I=I_0+1$ for $\Delta _1$.
To find the remaining $\Delta _2,\ldots ,\Delta _k$ and the related
$r_F(2),\ldots ,r_F(k)$ that appear in the main statement of
Theorem~\ref{mainStructure},
we will apply the induction hypothesis
to the restriction of $F$ to $M\setminus \Delta _1$, as an element in a
collection
\begin{equation}
\Lambda' =\Lambda (X,I_0,H_0,
\ve'_0,A'_0),
\label{L'}
\end{equation}
specified as in Definition~\ref{def:L}, for some choices of
$\ve'_0,A'_0$ that we will explain later.

First observe that the restriction
of $F$ to $M\setminus \Delta _1$ is an $H$-immersion with
smooth boundary and index at most
\[
(I_0+1)-\sum_{j=1}^{s_0}I(f_j)\leq (I_0+1)-s_0\leq I_0,
\]
that is, condition (A2) in Definition~\ref{def:L} for $\Lambda'$
holds for the upper index bound $I_0$.

Next we will
explain how to choose the remaining
 parameters $\ve'_0,A'_0$ that determine $\Lambda'$
in order to apply the induction hypothesis to $F|_{M\setminus \Delta _1}$
as an element in $\Lambda'$.

By Proposition~\ref{ass3.5'}, we have:
\begin{enumerate}[(P1)]
	\item \label{P1} Let $\wt{\Delta}_1$ be the component of
	$F^{-1}(\ov{B}_X(F(p_1),\de_4))$ that contains $p_1$. Then,
	the intersection of $F(\wt{\Delta}_1)$ with the
	region of $X$ between the extrinsic spheres $\partial B_X(F(p_1),
	\frac{R_{s_0}}{2t})$ and $\partial B_X(F(p_1),\de_4)$, consists
	of $e_{s_0}$ multi-graphical annuli $\wh{G}_{s_0}(1),
	\ldots ,\wh{G}_{s_0}(e_{s_0})$.
\end{enumerate}
In particular, the intrinsic distance between the two boundary curves of
each $\wh{G}_{s_0}(h)$, $h\in \{ 1, \ldots, e_{s_0}\}$,
is greater than or equal to
the following positive number independent of $F$:
\begin{equation}
\label{ve1}
\ve_1:=\de_4-\frac{R_{s_0}}{2t}.
\end{equation}
Observe that taking $\de_4$ smaller if necessary (this does not
affect the validity of Proposition~\ref{ass3.5'}), we can
assume $\de _4\in (0,\ve_0]$. Now, define $\ve_0'=\ve_1$.

Property (P1) implies that the following 
property hold:
\begin{enumerate}[(P2)]
\item The second fundamental form of $F$ is uniformly bounded
(independently of $(F\colon M\la X)\in \Lambda_{t}$)
in
\[
\Delta_1\cap F^{-1}\left( \ov{B_X}(F(p_1),\de)\setminus
B_X(F(p_1),\textstyle{\frac{R_{s_0}}{t}})\right)
\]
by a constant $A_1>0$ independent of $F$. Define
$A'_0=\max \{ A_0,A_1\} $.
\end{enumerate}


With the above choices, it follows that the restriction of $F$ to
$M\setminus \Delta _1$ lies in the collection $\Lambda'$ introduced
in~\eqref{L'}. By induction hypothesis (with the same choice
of $\tau$, recall that we are proving items 1,2,3 of
Theorem~\ref{mainStructure} by induction on~$I$), we can find
$A_1'\in [A'_0,\infty )$,
$\de'_1,\de'\in (0,\ve_0]$ (independent of $F$) with $\de'_1\leq \de'/2$,
and a possibly empty finite collection of points
\begin{equation}
\label{eq:PFrest}
\cP_{F|_{M\setminus \Delta _1}}=\{p'_1,\ldots,p'_k\}\subset
U(\partial (M\setminus \Delta _1),\ve'_0 ,\infty )
\end{equation}
$k\leq I_0$, and related numbers
\begin{equation}
\label{eq:r_Forder}
r'_F(1)>4r'_F(2)>\ldots >4^{k-1}r'_F(k),
\end{equation}
with $\{ r'_F(1),\ldots ,r'_F(k)
\}\subset [\de'_1,\frac{\de'}{2}]$ and  satisfying
properties~1,   2, 3 of Theorem~\ref{mainStructure}.

Finally, define
\begin{eqnarray}
A_1&=&\max \{ t,A'_1\},\qquad \de =\min \{ \de_4/2, \de'\},\qquad \de_1
=\min \{ \de_4/4,\de'_1\}  \label{eq:2.14}
\\
\cP_F&=&\{p_1,p_2=p'_1,\ldots,p_{k+1}=p'_k\}\subset
U(\partial M,\ve'_0 ,\infty ), \label{eq:PF}
\\
r_F(1)&=&\de_4/4, \quad r_F(2)=r'_F(1),\ldots ,\ r_F(k+1)=r'_F(k).
\label{eq:defrF}
\end{eqnarray}
(here $\de_4$ is the number defined in  Proposition~\ref{ass3.5'},
and $t$ was defined just after item (J5)'';
observe that we do not lose generality by assuming that
$r_F(1)>4r'_F(1)$)).
Also notice that the points $p_1,\ldots ,p_{k+1}$ belong to
$U(\partial M,\ve_0,\infty)$ (compare to \eqref{eq:PF} and to the
statement of Theorem~\ref{mainStructure}): the reason for this is that
$|A_{F}|(p_j)|>A_0'\geq A_1$ for each $j=1,\ldots ,k+1$.

Now, its clear that
items 1, 2, 3 of Theorem~\ref{mainStructure} hold for $I=I_0+1$ with the
exception of the first statement of item~\ref{it1}(c) for $i=1$
and $j\in \{ 2,\ldots ,k+1\} $, that we prove next.
To conclude that $B_M(p_1,\frac{7}{5}r_F(1))\cap
B_M(p_j,\frac{7}{5}r_F(j))=\varnothing$,
first note that $\frac{7}{5}r_F(1)=\frac{7}{20}\de_4<\frac{1}{2}\de_4$,
hence it suffices to show that $B_M(p_j,\frac{7}{5}r_F(j))$
does not intersect $F^{-1}\left[ \ov{B}_X(F(p_1),
\de_4)\setminus B_X(F(p_1),\de_4/2)\right] $. Arguing by contradiction,
suppose that there exists a point $q\in B_M(p_j,\frac{7}{5}r_F(j))$
such that $F(q)\in \ov{B}_X(F(p_1),
\de_4)\setminus B_X(F(p_1),\de_4/2)$. Then,
\[
\begin{array}{rclr}
\ve'_0&\leq & d_M(p_j,\partial \Delta _1)& \mbox{(by
(\ref{eq:PFrest}) and (\ref{eq:PF}))}
\\
\rule[-.3\baselineskip]{0pt}{.5cm}
& \leq &d_M(p_j,q)+d_M(q,\partial \Delta _1) &
\\
\rule[-.3\baselineskip]{0pt}{.5cm}
& < &\frac{7}{5}r_F(j)+d_M(q,\partial \Delta _1) &
\mbox{(because  $q\in B_M(p_j,\frac{7}{5}r_F(j))$)}
\\
\rule[-.3\baselineskip]{0pt}{.5cm}
& \leq  &\frac{7}{5}r_F(j)+\sqrt{1+\tau^2/3}\, \de_4/4 &
\mbox{(by item (C2) of Lemma~\ref{lemma5.3}, see $(\star)$ below)}
\\
\rule[-.3\baselineskip]{0pt}{.5cm}
& <  &\frac{7}{20}r_F(1)+\sqrt{1+\tau^2/3}\, \de_4/4 &
\mbox{(by (\ref{eq:r_Forder}))}
\\
\rule[-.3\baselineskip]{0pt}{.5cm}
& =  &\frac{1}{4}\left( \frac{7}{20}+\sqrt{1+\tau^2/3}\right) \de_4, &
\mbox{(by (\ref{eq:defrF}))}
\end{array}
\]
where in $(\star)$ we have used that $F(\partial \Delta _1)
\subset \partial B_X(F(p_1),\de_4/4)$ and $F(q)\notin
B_X(F(p_1),\de_4/2)$. Hence it suffices to show that
the inequality
\[
\de_4-\frac{R_{s_0}}{2t}\stackrel{(\ref{ve1})}
{=}\ve_1=\ve'_0<\frac{1}{4}\left( \frac{7}{20}+\sqrt{1+\tau^2/3}\right)
\de_4
\]
leads to contradiction. Manipulating the last inequality, it is clearly
equivalent to
\[
\left[ 1-\frac{1}{4}\left( \frac{7}{20}+\sqrt{1+\tau^2/3}\right)
\right] \de_4<\frac{R_{s_0}}{2t}
\stackrel{\mbox{\footnotesize (K1)'}}{\leq }
\frac{\de_1}{10}\frac{5}{a(I_0)+1}
\stackrel{(\ref{de,de1})}{=}\frac{\de_4}{8}\frac{1}{a(I_0)+1},
\]
Therefore,
$B_M(p_1,\frac{7}{5}r_F(1))\cap B_M(p_j,\frac{7}{5}r_F(j))=\varnothing$
for $i=1$ and $j\in \{ 2,\ldots ,k+1\} $.
This completes the proof of Proposition~\ref{propos7.17}.
\end{proof}

Recall that the domains $\Delta_1=\Delta_1(n)\subset M_n$ 
are defined in the proof of Proposition~\ref{propos7.17} and  
each such domain is geometrically the component of $p_1=p_1(n)$ in the
preimage by $F=F_n$ of an extrinsic ball in $X=X_n$ 
centered at $F_n(p_1(n))$ of a small radius $r_F(1)=\de_1$ independent of $n$. 
For future referencing in the definition of ``the hierarchy structure of $\Delta_1$'' 
appearing in the next section, we make the following definition. 

\begin{definition}  \label{def:singSets}
Suppose that the number of
ascending levels $s_0\in \N$ in the construction of $\Delta_1(n)$  
satisfies $s_0>1$. In this case, for each $i\in \{2,\ldots,s_0\}$,
we define the related sets:
\ben
\item $Q_2(n)\subset M_n$ (defined in  Proposition~\ref{prop7.12}), which satisfy:
\ben
\item $Q_2(n)$ contains   $p_1(n)$ and its finite  cardinality is
independent of $n$ 
and at most $I$.
\item The norms of the second fundamental 
forms of the immersions $\frac{1}{r_n}F_n\colon M_n\la\frac{1}{r_n}X_n$ 
have local maxima at points in $Q_2(n)$ that are blowing-up 
as $n\to \infty$.
\item The  points in $Q_2(n)$ stay at a uniform distance at most $R_{0,2}$ 
(this is the constant $R_0$ appearing in the main statement of Proposition~\ref{prop7.12})
from the points $p_1(n)$ in the metric of $M_n$ induced by 
$\frac{1}{r_n}F_n\colon M_n\la\frac{1}{r_n}X_n$, and these points stay at 
a uniform distance greater than some $\ve_{2,2}>0$ 
(called $\ve_2>0$ in Proposition~\ref{prop7.12})
from each other. 
\een
For $i\in \{3,\ldots,s_0\}$, 
$Q_i(n)\subset M_n$ are the similarly defined finite 
sets in $M_n$ with related positive numbers  $R_{0,i}$, 
$\ve_{2,i}$, with respect to rescalings of the immersions $F_n\colon M_n\la X_n$. 
Furthermore, for $i\neq i'\in \{2,\ldots,s_0\}$,   $Q_i(n) \cap Q_{i'}(n)=\{p_1(n)\}$,
and so, each of the sets $Q_i(n)$ contains the point $p_1(n)$.
\item The set $\cS_2\subset \S_2$ is defined in item~6(a) of Proposition~\ref{prop7.12}
(it was called $\cS(f_2)$ there).  
For $i\in \{3,\ldots,s_0\}$, the sets $\cS_i\subset \S_i$ are defined in the similar manner.
\een
\end{definition}

\subsection{Counting index, genus and total spinning for local
hierarchies}
\label{sec9.1}

In Section~\ref{sec2.7} we have explained a process of going ``up'' in
finding scales and limits with center $p_1(n)$, so that after
$s_0\leq I_0+1$ steps, we finish the ``ascending'' process and define the final
$\Delta$-piece containing $p_1(n)$
(called $\Delta_1$ in the proof of Proposition~\ref{propos7.17}).
Throughout this ascending process, we have found other points
occurring inside $\Delta_1$ where the
second fundamental form can blow up; we will refer to
these blow-up points as $q$-points in $\Delta_1$ (these $q$-points that lie in
the sets $Q_i(n)\subset M_n$ described in item~1 of Definition~\ref{def:singSets} 
and produce corresponding sets $\cS_i\subset \S_i$, $i=2\ldots,s_0$, described in
item~2  of Definition~\ref{def:singSets}). It is crucial to remark that
the compact piece $\Delta_1=\Delta_1(n)$ occurs in a sequence of
immersions $F_n\colon M_n\la X_n$, while its topological and geometric
structure also depends on the complete,
possibly branched minimal surfaces
which are limits obtained after blowing up $\Delta_1(n)$
around its $q$-points.

In order to understand the structure of the piece $\Delta_1$ (i.e.,
to prove the estimates in items B, C, D of Theorem~\ref{mainStructure}),
we must analyze how the related $\Delta$-pieces around these $q$-points
affect the geometry of $\Delta_1$. This analysis will be done
by going ``down  levels'' in $\Delta_1$: we will first analyze 
the $q$-points in $Q_{s_0}(n)$, i.e., those $q$-points 
occurring at the level of the limit
$f_{s_0}$ (this is the {\it top level} of the piece $\Delta_1$
in the language introduced in Section~\ref{sec5.6.1} below),
and subsequently go to lower levels which occur at every
$q$-point not being a minimal element in the sense of
Definition~\ref{def7.18} below.
The notion of {\it hierarchy} of $\Delta_1$ (Definition~\ref{def7.20})
will encompass all $q$-points at different levels and the related
$\Delta$-type pieces around them.
The way that this hierarchy affects some quantities
appearing in items  B, C, D of Theorem~\ref{mainStructure} (like index,
genus, number of boundary components, total spinning along the
boundary, etc) is encoded in Theorem~\ref{thm9.6} below,
which is an inequality that generalizes the Chodosh-Maximo
estimate~\eqref{eq:CM1} to the new framework of hierarchies.
Although it is premature at this point for the reader to fully understand
what is meant by a hierarchy, we suggest that the reader frequently
check his/her developing understanding of this concept
by referring to the schematic Figure~\ref{figh1} below,
which  represents a particular example
of a hierarchy; also see Example~\ref{ex5.18} and item~3 in
Example~\ref{example7.19} for further explanations of this example.

In the remainder of this section, $|X|$ will denote the number of elements of a
finite set $X$, and if $X$ is a topological space with finitely many
connected components, $\#_c(X)$ will denote the number of these
components.

\subsubsection{The hierarchy associated to a $\Delta$-type
piece}\label{sec5.6.1}
Let $\{ F_n\}_n$ be a sequence in the space
$\L=\L(I_0,H_0,\ve_0,A_0,K_0)$
with second fundamental form not uniformly bounded. Let
$\Delta=\Delta(n)$ be the
connected, compact surface that arises around an initial blow-up point
$p(n) \in M_n$ for $n$ large (this is a $\Delta$-piece, in the language
of the first two paragraphs of Section~\ref{sec9.1}).
Recall that the construction given
in Section~\ref{sec2.7} performs finitely many blow-ups centered at the
$p(n)$, giving rise to $s_0$ {\it stages}
$(f_i,{\mathcal S}_i,\{ \l_{i,n}\}_n)$, $i=1,\ldots ,s_0$
described in items (S1)-(S2) below.
\begin{enumerate}[(S1)]
\item $f_i\colon \Sigma_i\la \R^3$ is a (possibly finitely connected)
complete minimal surface in $\R^3$ with finite total curvature
that passes through the origin, possibly with a finite number of
branch points and possibly with non-orientable components.
$\Sigma_1$ is connected and $f_1\colon \Sigma_1\la \R^3$ is
unbranched and non-flat, but for $i=2,\ldots ,s_0$,
$f_i$ could have flat components with or without
branch points, in the sense that the image set of the related branched
immersion lies in a flat plane
(that could fail to pass through the origin).

\item ${\mathcal S}_i\subset \Sigma _i$ is a finite subset
(${\mathcal S}_1=\varnothing$) and $\{ \l_{i,n}\}_n\subset \R^+$ is
a sequence diverging to $\infty $ such that (see Section~\ref{sec7.5.3})
\begin{enumerate}[(a)]
\item $\{ \l_{i,n}F_n\}_n$ converges to $f_i$ 
in $\S_i\setminus {\mathcal S}_i$ as $n\to \infty$.

\item $\{ \l_{i,n}F_n\}_n$ fails to have bounded
second fundamental form around each point of $Q_i(n)$ (this is the set 
introduced in Definition~\ref{def:singSets}, which gives rise to $\cS_i$). 

\item $\l_{i,n}/\l_{i+1,n}\to \infty $ as $n\to
\infty$ for each $i=1,\ldots ,{s_0}-1$.
\end{enumerate}
Because of properties (S2.a), (S2.b), we will refer to
$\cS_i$ as the {\it singular set of convergence} of $\l_{i,n}F_n$
to $f_i$.
\end{enumerate}

\begin{example}\label{ex5.18}
{\rm
We will illustrate the above description with an example based on
Figure~\ref{figh1}. The blue circle around $\Delta_{q_{1,1}}$ represents a
compact $\Delta$-piece of $\l_{1,n}F_n$ based at the blow-up points $q_{1,1}(n)\in M_n$ 
which resembles arbitrarily well
(for $n$ large) the intersection of the first stage limit $f_1\colon \S_1
\la \R^3$ introduced in item (S1) with a ball of large radius centered at
the origin;  the ascending blue straight line segment connecting the
blue circle around $\Delta_{q_{1,1}}$ with the red circle around
$\Delta_{q_1}$ represents a component $W_1'$ of the second stage limit
$f_2\colon\S_2\la \R^3$
which contains at least one  point in $\cS_2$ obtained as a blow-down
limit (by scale $\frac{\l_{2,n}}{\l_{1,n}}\to 0$) of the $\Delta$-piece
$\Delta_{q_{1,1}}$ in $\l_{2,n}F_n$. In fact, each end of $f_1$ is a
multi-graph outside of a ball of some finite multiplicity $m_1\in \N$,
such an end produces a branch point for $f_2$ of multiplicity $m_1-1$, and
the number of leaves of $f_2$ passing through the image of such a branch point is
at least equal to the number of ends of $f_1$.
The red circle around  $\Delta_{q_1}$ represents a compact $\Delta$-piece
of $\l_{2,n}F_n$ which resembles arbitrarily well (for $n$ large) the
intersection of $f_2(\S_2)$ with a ball of large radius centered at the
origin; the ascending red straight line segment connecting the red circle
around $\Delta_{q_1}$ to the black circle around $\Delta$ represents a
component $W_1$ of the third stage limit $f_3\colon \S_3\la \R^3$
which contains a point in $\cS_3$ obtained as a blow-down limit
(by scale $\frac{\l_{3,n}}{\l_{2,n}}\to 0$) of the $\Delta$-piece
$\Delta_{q_1}$ inside $\l_{3,n}F_n$. Similarly as before,
each end of $f_2$ is a multi-graph outside of a ball
of some finite multiplicity $m_2\in \N$, this end produces a branch point
for $f_3$ of multiplicity $m_2-1$, and the number of leaves of $f_3$
passing through such a branch point is at least equal to the number
of ends of $f_2$. The black circle around  $\Delta$
represents the final compact $\Delta$-piece of $\l_{3,n}F_n$, i.e.,
Proposition~\ref{ass3.5'} holds with $s_0=3$ for this ''ascending''
linear subgraph starting at $\Delta_{q_{1,1}}$ and finishing at $\Delta$.
If we start ascending from $\Delta_{q_{1,2}}$ instead of from
$\Delta_{q_{1,1}}$, we will find again $s_0=3$ (although the stage
limits are different than before, since the rescaling
is centered at a different blow-up sequence in $M_n$)
but if we start ascending from $\Delta_{q_2}$
(resp. from $\Delta_{q_{3,1,1}}$) we will find $s_0=2$ (resp. $s_0=4$).
Both $W_1'$ and the T-shaped polygon $W_2'$ connecting the blue circles around
$\Delta_{q_{1,1}}, \Delta_{q_{1,2}}$ with the red circle around
$\Delta_{q_1}$ represent that $\S_2$ has two components,
each one with its own number of ends, and that each of these
ends possibly produce branch points in $\cS_3$ as explained above.
We will continue with explaining aspects of this Figure~\ref{figh1}
in Example~\ref{ex5.19}.
}
\end{example}

We now come back to the general description with the notation in (S1)-(S2)
and in Definition~\ref{def:singSets}.
The {\it hierarchy} $\cH(\Delta)$ of $\Delta $ decomposes
into finitely many {\it levels,} which are defined
recursively as follows, starting from what we will call the
{\it top level} of $\cH(\Delta)$. There exists a possibly disconnected
complete, branched minimal
immersion $f_{\rm T}\colon \S_{\rm T} \la \R^3$ (the subindex ${\rm T}$
stands for top; in the notation in (S1)-(S2), $f_{\rm T}=f_{s_0}$), such
that the convergence of portions of suitable expansions
$\l_{\rm T}(n)F_n=\l_{s_0,n}F_n$ of $F_n$ to $f_{\rm T}$
is smooth away from a finite singular set of convergence
$\mathcal{S}_{\rm T}\subset \S_{\rm T}$
($\mathcal{S}_{\rm T}$ could be empty), and the second
fundamental forms of $\l_{\rm T}(n)F_n$ 
fail to be bounded around (extrinsically) each point $q\in \mathcal{S}_{\rm T}$; 
suppose that such a point $q$ corresponds to
a sequence $\{ q(n)\}_n$ with $q(n)\in Q_{\rm T}(n)\subset M_n$ for $n\in \N$ sufficiently large. 
This means that $\{ \l_{\rm T}(n)F_n(q(n))\}_n$
converges to $f_{\rm T}(q)$ (in harmonic coordinates 
of radius $R_{0,{\rm T}}$ centered at $F_n(p_1(n))$, where $R_{0,{\rm T}}$ is defined in item~1(c) 
of Definition~\ref{def:singSets}) and 
\[
\lim_{n\to \infty} \sup \left\{
|A_{\l_{\rm T}(n)F_n}|(x) \ : \
x\in B_{\l_{\rm T}(n)X_n}(F_n(q(n)),1/m)\right\}=\lim_{n\to \infty}|A_{\l_{\rm T}(n)F_n}|(q(n)) =\infty ,
\]
for each $m\in \N$ sufficiently large.
Moreover:
\begin{enumerate}[(T1)]
\item $f_{\rm T}$ is unbranched away from $\mathcal{S}_{\rm T}$.
\item The number of ends $e(\S_{\rm T})$ of $\S_{\rm T}$
(resp. the total spinning at infinity $S(f_{\rm T})$ of
$f_{\rm T}$)  equals the number of boundary components of $\Delta$
(resp. total spinning $S(\Delta)$ of $\Delta $ along $\partial \Delta$):
\begin{equation}\label{9.3}
e(\S_{\rm T})=\#_c(\partial \Delta):=e(\Delta),
\qquad S(f_{\rm T})=S(\Delta).
\end{equation}
\end{enumerate}
Let $\mathcal{W}_{\rm T}$ be the set of components of $\S_{\rm T}$.

We next make a similar quotient space of the abstract surface $\S_{\rm T}$
of this branched immersion $f_{\rm T}$ as the one in
Lemma~\ref{lemma7.13},
thereby defining a quotient space $\wh{\S}_{\rm T}$ of $\S_{\rm T}$,
a related quotient map $\pi\colon \S_{\rm T}\to \wh{\S}_{\rm T}$, and
a singular set 
\[
\wh{\cS}_{\rm T}=\pi (\cS_{\rm T})
\]
defined as in Lemma~\ref{lemma7.13}. Observe that $|\wh{\cS}_{\rm T}|=|Q_{\rm T}(n)|$
(which is independent of $n$). Given $q\in \wh{\cS}_{\rm T}$, let
$\cS_{\rm T}(q)=\pi^{-1}(q)\subset \cS_{\rm T})$. Thus, every point in 
$\cS_{\rm T}(q)$ identifies to the point $q$ in $\wh{\S}_{\rm T}$,
and every other point of $\S_{\rm T}$ only identifies with itself.
After endowing $\wh{\S}_{\rm T}$ with the quotient topology,
$\wh{\S}_{\rm T}$ becomes a path-connected metric space,
and $f_{\rm T}\colon \S_{\rm T}\to \R^3$ induces a well-defined
continuous map, denoted also by $f_{\rm T}\colon \wh{\S}_{\rm T}\to \R^3$
with a slight abuse of notation.
Observe that $\wh{\S}_{\rm T}\setminus \wh{\cS}_{\rm T}$ has the induced 
structure of a (smooth) Riemannian surface, and that $\wh{\S}_{\rm T} $ is a topological 
surface in a small neighborhood
of a given point $q\in\wh{\cS}_{\rm T}$ if and only if
$\cS_{\rm T}(q)$ consists of a single point. Also, the restriction of
$f_{\rm T}$ to $\wh{\S}_{\rm T}\setminus \wh{\cS}_{\rm T}$
is a minimal immersion with finite total curvature in $\R^3$, which is
complete away from its  limit point set $\wh{\cS}_{\rm T}$ in $\wh{\S}_T$.

\begin{example}\label{ex5.19}
{\rm
As announced in Example~\ref{ex5.18}, we continue to explain some aspects
in Figure~\ref{figh1}. The red component $W_1$ of $\S_3$ connects to the
red circles around $\Delta_{q_1},\Delta_{q_2}$,
meaning that $W_1$ contains at least two
distinct points in $\cS_3$ which lead to two distinct points $q_1,q_2
\in \wh{\cS}_3$. The blue component $W_2'$ of $\S_2$ connects to the red
circle around $\Delta_{q_1}$ and to the blue circles  around
$\Delta_{q_{1,1}},\Delta_{q_{1,2}}$, meaning that $W_2'$ contains at least
two distinct points in $\cS_2$ which produce distinct points $q_{1,1},
q_{1,2}\in \wh{\cS}_2$, in contrast to the blue component $W_1'$
of $\S_2$, whose points in $\cS_2$ only give rise to one point in
$\wh{\cS}_2$, namely $q_{1,1}$.
}
\end{example}

We now return to the general situation. Given $q\in \wh{\cS}_{\rm T}$,
for all $n$ sufficiently large we can find a related compact,
connected piece $\Delta_{q}=\Delta_{q}(n)\subset M_n$
satisfying item~8(f) of Proposition~\ref{prop7.12}
for $\wt{F}_n=\l_{\rm T}(n)F_n$.

The index of $\Delta_{q}$ is
strictly less than the index of $\Delta$  (this is clear in the case that
$\wh{\cS}_{\rm T}\setminus \{q\}\neq \varnothing$; in the case that
$ \wh{\cS}_{\rm T}= \{q\}$ we have that
$f_{\rm T}$ cannot be flat, since this corresponds to the case $J=1$
in item~8(f) of Proposition~\ref{prop7.12}), and so,
we can apply Lemma~\ref{lema3.6} to conclude that $f_{\rm T}$ is
not stable, which gives
Index$(\Delta)\geq \mbox{ Index}(\S_{\rm T})+\mbox{ Index}(\Delta_{q})
>\mbox{ Index}(\Delta_{q})$. 
For different points $q,q'\in \wh{\cS}_{\rm T}$,
the corresponding compact domains $\Delta_{q(n)},\Delta_{q'(n)}\subset M_n$
are disjoint.

Let
\[
\mathcal{V}_{\rm T}=
\mathcal{V}_{\rm T}(n)=\{ \Delta_{q}=\Delta_{q(n)}\subset M_n\ |
\ q\in \wh{\cS}_{\rm T}\},
\]
Given $q\in \wh{\cS}_{\rm T}$, let $\cW_{\rm T}(q)$ be the (finite) set of
components of $\S_{\rm T}$ such that each $W\in \cW_{\rm T}(q)$ contains
at least one point of $\cS_{\rm T}(q)=\pi^{-1}(q)$.
We can choose a finite collection
$\cD_{q}$ of sufficiently small (possibly branched) {\it stable} minimal
disks in $\S_{\rm T}$ centered at the points in $\cS_{\rm T}(q)$,
such that
\begin{enumerate}[(U4)]
\item For each component $W$ of $\S_{\rm T}$, it holds
$I(W)=I(W\setminus \cup_{q\in \wh{\cS}_T}\cD_{q})$.
\end{enumerate}
\begin{enumerate}[(U5)]
\item The set
\begin{equation}\label{9.4a}
\cV_{\rm T}^c:=\bigcup_{q\in \wh{\cS}_{\rm T}}\cD_{q}\subset \S_{\rm T}
\end{equation} is contained in the limit as $n\to \infty$ of
$\l_{\rm T}(n)\cV_{\rm T}(n)$.
\end{enumerate}

Let
\begin{equation}\label{9.4}
	\S_{\rm T}^c=\S_{\rm T}\setminus \cV_{\rm T}^c.
\end{equation}
Property (U4) implies that the index $I(\S_{\rm T})=I(\S_{\rm T}^c)$.
Note that the number of components $\#_c(\S_{\rm T})=\#_c(\S_{\rm T}^c)$,
since removing an interior disk from a connected surface does not
disconnect it.

\begin{definition}
\label{def7.18}
{\rm
If $\cS_{\rm T}=\varnothing$ in the situation above, then
$\S_{\rm T}$ consists of a single non-flat, connected, unbranched minimal
surface with finite total curvature. In this case, we say that the
hierarchy $\mathcal{H}=\mathcal{H}(\Delta)$ of $\Delta$ is {\it trivial}
(with no levels) and that $\Delta $ is a {\it minimal element}.

If $\cS_{\rm T}\neq \varnothing$, then we
define the {\it top level} of $\Delta=\Delta(n)$ (for $n$ large)
as the triple $(\bf \wh{\cS}_{\rm T},\mathcal{V}_{\rm T},
\mathcal{W}_{\rm T})$.
In this case, we can apply for each $q\in \wh{\cS}_{\rm T}$
the above description to the corresponding compact domain $\Delta_{q}$
(exchange $\Delta $ by $\Delta _q$), which produces the 
triple $(\wh{\cS}_{{\rm T}(q)},\mathcal{V}_{{\rm T}(q)},
\mathcal{W}_{{\rm T}(q)})$
associated to $\Delta_{q}$ with top level ${\rm T}(q)$.
As before, we have two cases:
\begin{itemize}
\item If $\wh{\cS}_{{\rm T}(q)}=\varnothing $ for a point
$q\in \wh{\cS}_{\rm T}$, then the hierarchy of $\Delta_{q}$ is trivial and
$\Delta_{q}$ is called a {\it minimal element}
(for instance, in Figure~\ref{figh1}, the minimal elements are
$\Delta_{q_2},\Delta_{q_{1,1}},\Delta_{q_{1,2}},\Delta_{q_{3,1,1}}$, which
have associated number of stages $s_0(q_2)=2$,
$s_0(q_{1,1})=s_0(q_{1,2})=3$, $s_0(q_{3,1,1})=4$;
observe that the number of stages is not defined for the $\Delta_q$-pieces
which are not minimal elements).

\item If $\wh{\cS}_{{\rm T}(q)}\neq \varnothing $,
we say that the corresponding top level $(\wh{\cS}_{{\rm T}(q)},
\mathcal{V}_{{\rm T}(q)},\mathcal{W}_{{\rm T}(q)})$ of $\Delta_{q}$ is a
{\it level of the hierarchy $\mathcal{H}(\Delta)$}
different from its top level, and proceed
recursively. Let us denote by $L\in \N\cup \{ 0\}$ the number of these
levels of $\mathcal{H}(\Delta)$;
see Figure~\ref{figh1} for the schematic representation of a hierarchy
$\mathcal{H}(\Delta)$ with four levels.
\end{itemize}
}
\end{definition}

\begin{remark}
\label{rem9.4}
{\rm
\begin{enumerate}
\item Observe that the notion of level only makes sense provided that
$\wh{\cS}_{\rm T}\neq \varnothing $.
\item This recursive process of assigning levels to $\Delta $
(not being a minimal element) is finite, since each $\Delta_{q}$ has
non-zero index, which can be realized by a compact unstable domain
in $M_n$ for $n$ large, and the related compact unstable
domains for different $q$-points  in the same level of $\Delta$
can be taken pairwise disjoint (recall that the index of $F_n$ was
assumed to be less than or equal to some bound $I_0$ independent of $n$).
\item This recursive process of assigning levels to $\Delta $
(not being a minimal element) is finite. In fact, it follows
from the arguments used to prove item~8(f) of Proposition~\ref{prop7.12} that the index increases
each time we add a level, and so $L+1\leq I(\Delta)$.  
\end{enumerate}
}	
\end{remark}

\begin{definition}
\label{def7.20}
{\rm We define the {\it singular set} $\wh{\cS}$ as the union
of all the singular sets $\wh{\cS}_{{\rm T}(q)}$ for singular points of
previously defined levels (including $\wh{\cS}_{\rm T}$).
Similarly, we let  $\cS$ be the union of all $\cS_{{\rm T}(q)}$
for singular points of previously defined levels. Let
$\mathcal{V}\subset M_n$ be the union of
$\{ \Delta\}$ together with all the
compact pieces $\Delta_{q}$ for singular points of levels of
$\cH(\Delta)$, and let $\mathcal{W}$ be the union of all the
components of related limit surfaces $\S_{{\rm T}(q)}$
for singular points of previously defined levels
(including $\S_{\rm T}$). We define the {\it hierarchy}
$\mathcal{H}(\Delta)$ of $\Delta=\Delta(n)$ (for $n$ large)
as the triple  $(\wh{\cS},\mathcal{V},\mathcal{W})$; and the number
$L\in \N\cup \{0\}$ associated to $\Delta$ (see
Definition~\ref{def7.18}) is called the {\it number of levels of
$\mathcal{H}(\Delta)$.}
If $\mathcal{H}(\Delta)$ is non-trivial, a compact domain
$\Delta_{q}\in \mathcal{V}$
(here $q\in \wh{\cS}$)	is called a {\it minimal element} of
$\mathcal{H}(\Delta)$ if the hierarchy associated to
$\Delta_{q}$ is trivial (recall that if $\mathcal{H}(\Delta)$ is trivial,
we called $\Delta $ itself a minimal element).
}
\end{definition}

\begin{example}
\label{example7.19}
{\rm
\begin{enumerate}
\item $\wh{\cS}=\varnothing$ if and only if
$\wh{\cS}_{\rm T}=\varnothing$ if and
only if the hierarchy of $\Delta$ is trivial. In this case,
$\cW=\cW_{\rm T}=\{ \S_{\rm T}\}$,
$\cV_{\rm T}=\varnothing$, $\cV=\{ \Delta\}$, $L=0$,
and $\Delta$ is a minimal element.

\item The simplest case of a non-trivial hierarchy $\cH(\Delta)$ is that
of having just one single singular point in its top level (i.e.,
$\wh{\cS}=\wh{\cS}_{\rm T}=\{ q\}$) and where $\Delta_{q}$
has one boundary curve. In this example,
$\cV_{\rm T}=\{ \Delta_{q}\}$, $\cV=\{\Delta,\Delta_{q}\}$,
$\cW_{\rm T}$ consists of a
single, non-flat  (non-flatness of this single component of $\cW_{\rm T}$
follows from the proof of item~8(f) of Proposition~\ref{prop7.12}),
connected, complete minimal surface $\S_{\rm T}$ with finite total
curvature and a unique branch point at $q$ with branching order at least
two, $\cW=\{ \S_1,\S_{\rm T}\}$ where
$\S_1$ is a non-flat, connected, complete minimal immersion
(no branch points) with
finite total curvature, the number of levels is $L=1$,
and $\Delta_{q}$ is a minimal element.

\item  See Figure~\ref{figh1} for an example of a hierarchy with four
levels. In this example, $\wh{\cS}_{\rm T}=\{ q_1,q_2,q_3\}$,
$\wh{\cS}_{{\rm T}(q_1)}=\{ q_{1,1},q_{1,2}\}$, $\wh{\cS}_{{\rm T}(q_3)}=\{ q_{3,1}\}$,
$\wh{\cS}_{{\rm T}(q_{3,1})}=\{ q_{3,1,1}\}$.
The minimal elements of this hierarchy are $\Delta_{q_2},\Delta_{q_{1,1}},
\Delta_{q_{1,2}},\Delta_{q_{3,1,1}}$. The surface $\S_{\rm T}$
has two (possibly) branched components $W_1,W_2$, and the
$\pi$-projection of the set of branch points of $W_1$  is contained in $\{q_1,q_2\}$, while
$\pi$-projection of the set of branch points of $W_2$ 
 is contained in $\{q_2,q_3\}$.
Observe that in this example $\Delta_{q_2}$ has at least two boundary components
(for $n$ large), one component  which corresponds to  the boundary
of a possibly branched minimal disk
in the limit branched minimal surface $W_1$ and another component which
corresponds to the boundary of a possibly branched minimal disk in the limit $W_2$.
\end{enumerate}
}
\end{example}
\begin{figure}[h]
\begin{center}
\includegraphics[width=8cm]{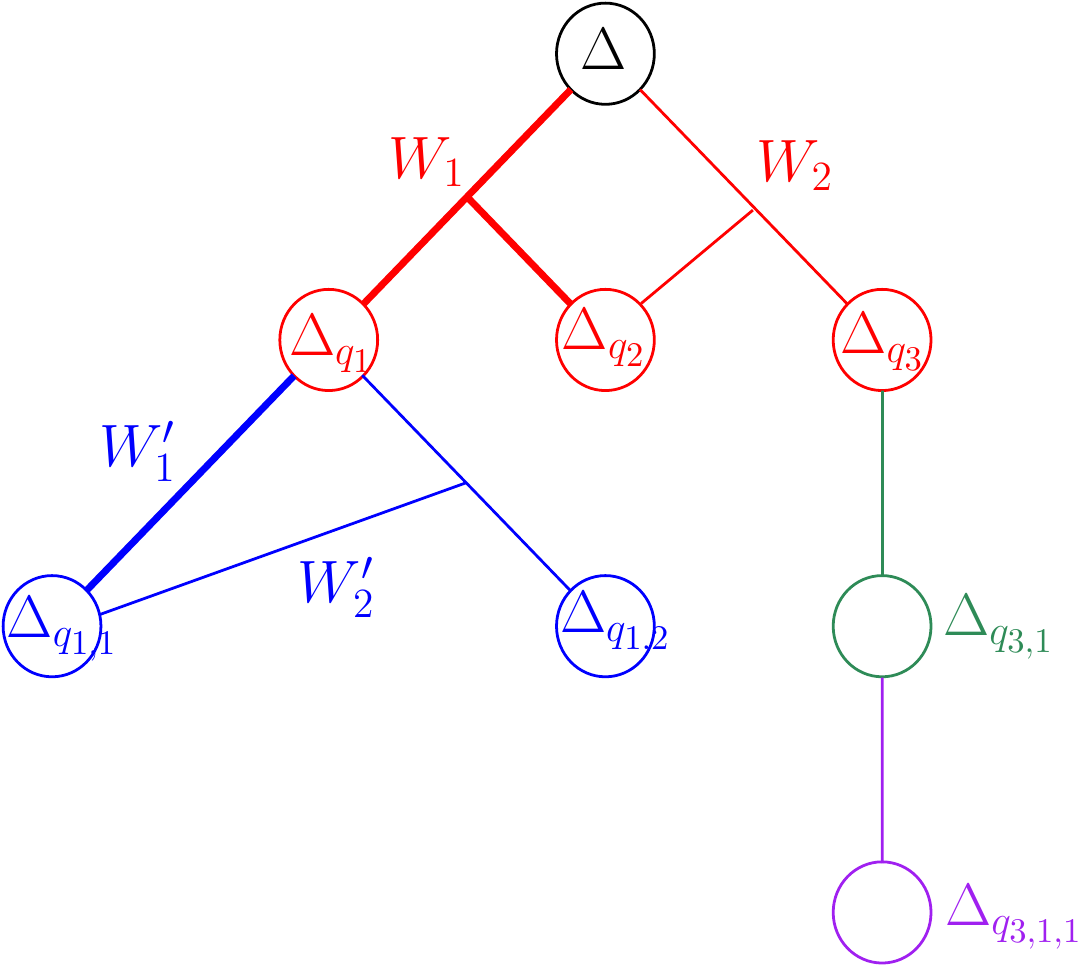}
\caption{Schematic representation of a hierarchy
$\mathcal{H}(\Delta)$ with four levels (top level in red, other levels in
blue, green and purple.}
\label{figh1}
\end{center}
\end{figure}

$\mathcal{V}$ can be equipped with the following partial order:
given $\Delta',\Delta''\in \mathcal{V}$, we denote
$\Delta'\preceq\Delta''$ if $\Delta'\subseteq\Delta''$.
Thus, $\Delta_{q}\preceq \Delta$ for every $q\in \wh{\cS}$, and
$\Delta_{q}\in \mathcal{V}$ is a minimal element of $\mathcal{H}(\Delta)$
precisely when $\Delta_{q}$ is minimal with respect to the partial order
$\preceq$.

The set $\mathcal{V}$ decomposes into
\begin{equation}
\mathcal{V}=\mathcal{V}^m\cup \mathcal{V}^{nm},
\label{mnm1}
\end{equation}
where $\mathcal{V}^m=\{ \Delta'\in \mathcal{V}\ | \
\Delta' \mbox{ is a minimal element}\} $
and $\mathcal{V}^{nm}=\mathcal{V}\setminus \mathcal{V}^{m}$. Note that
each non-minimal element $\Delta_{q}\in \mathcal{V}^{nm}$ with
$q\in \wh{\cS}$ creates a level of $\mathcal{H}(\Delta)$ below it with respect to $\preceq$
(namely, its top level $(\wh{\cS}_{{\rm T}(q)},
\mathcal{V}_{{\rm T}(q)},\mathcal{W}_{{\rm T}(q)})$).
Assuming that $\cH(\Delta)$ is non-trivial,
all the levels of $\mathcal{H}(\Delta)$ except for the
top one are created this way; hence
\[
L=|\mathcal{V}^{nm}| \quad \mbox{ if $\mathcal{H}(\Delta)$
is non-trivial.}
\]

Also, observe that $|\wh{\cS}|+1=|\mathcal{V}^{m}|+|\mathcal{V}^{nm}|$
regardless of whether or not $\Delta$ is a minimal element.
In particular, $|\wh{\cS}|\geq L$.

\begin{definition}
\label{def7.22}
{\rm
We define the {\it excess index} associated to the subset of
minimal elements of $\Delta$ as
\begin{equation}
I^*(\mathcal{H})=\sum_{\Delta'\in \mathcal{V}^m}(I(\Delta')-1)\in
\N\cup \{ 0\}.
\label{I*}
\end{equation}
}
\end{definition}

This abstract model of the hierarchy $\mathcal{H}(\Delta)$ produces a
``decomposition'' of the compact domain $\Delta=\Delta(n)\subset M_n$
for $n\in \N$ large into compact pieces (in the sense that each piece is a
compact surface with boundary inside $\Delta$, the union of the pieces is
$\Delta$ and the pieces only intersect along their boundaries): these
pieces correspond to the $\Delta_{q}$ with $q\in \wh{\cS}_{\rm T}$
(observe that $\Delta_{q}=\Delta_{q}(n)$ is contained in $M_n$),
together with a (finitely connected) compact surface
$W(n)\subset \Delta(n)$
which is  the closure of $\Delta \setminus
\left( \bigcup _{q\in \wh{\cS}_{\rm T}}\Delta_{q}\right) $.
Observe that after suitable rescaling by some $\l_{\rm T}(n)\in \R^+$
diverging to $\infty$, the $\l_{\rm T}(n)W(n)$ converge as  $n\to \infty $
to the components  of the surface $\S_{\rm T}^c$ defined in~\eqref{9.4}.

\begin{definition}
\label{def9.4}
{\rm
We define
\begin{equation}
\left.
\begin{array}{ll}
\cS=\bigcup_{q\in \wh{\cS}}\pi^{-1}(q)&
\\
\cW(\partial =1)&\mbox{set of components $W$ in $\cW$ such that
$|W\cap \cS|=1$,}
\\
\mathcal{W}^{f} &\mbox{set of flat components in $\cW$,}
\\
\mathcal{W}^{t}= \cW(\partial =1)\cap \cW^f
&\mbox{set of {\it trivial} components in $\cW$,}
\\
\mathcal{W}^{nt}=\mathcal{W}\setminus \mathcal{W}^{t} &
\mbox{set of non-trivial components in $\mathcal{W}$,}
\\
\mathcal{W}^{nt,f} &\mbox{set of non-trivial flat components in
$\mathcal{W}$,}
\\
\mathcal{W}^{nt,nf}=\mathcal{W}^{nt}\setminus \mathcal{W}^{nt,f}&
\mbox{set of non-trivial, non-flat components in  $\mathcal{W}$, and}
\\
\cW(\partial >1)=\cW\setminus \cW(\partial=1)&
\mbox{set of components $W$ in $\cW$ such that $|W\cap \cS|>1$}.
\end{array}\right\}\label{9.7c}
\end{equation}
}
\end{definition}

We will also need the following decomposition of $\mathcal{W}^{nt,nf}$:
\begin{equation}\label{9.8a}
\mathcal{W}^{nt,nf}=\mathcal{W}^{nt,nf}(\partial =1)
\cup \mathcal{W}^{nt,nf}(\partial >1),
\end{equation}
where $\mathcal{W}^{nt,nf}(\partial =1)=\cW^{nt,nf}\cap \cW(\partial=1)$
(resp. $\mathcal{W}^{nt,nf}(\partial >1)
=\cW^{nt,nf}\cap \cW(\partial>1)$). In turn, the following decomposition
of $\mathcal{W}^{nt,nf}(\partial>1)$ will be useful:
\begin{equation}\label{9.8b}
\mathcal{W}^{nt,nf}(\partial>1)=\mathcal{W}^{nt,nf,or}(\partial >1)
\cup \mathcal{W}^{nt,nf,no}(\partial >1),
\end{equation}
where the super-index {\it or} (orientable), {\it no} (non-orientable)
refers to the orientability character of each component.

In this paragraph we indicate how the notion of hierarchy arises
in the proof of Structure Theorem~\ref{mainStructure}.
Previously we used the  notion of ``ascension with
$s_0$ stages'' associated to a sequence of points $p_1(n)\in M_n$
with sufficiently large norm of its second fundamental form,
which created a compact piece $\Delta=\Delta_1$,
defined just after item (K2)'.
This is the first step in constructing the hierarchy $\cH(\Delta)$), and
in the previous sections we have proven the following partial result
related to Theorem~\ref{mainStructure}:
For  any $(F\colon M\la  X)\in \L=\L(I, H_0,\ve_0,A_0,K_0)$,
there exists 
a (possibly empty) finite collection $\cP_F=\{p_1,\ldots,p_k\}\subset
U(\partial M,\ve_0 ,\infty )$ of points, $k\leq I$,
numbers $r_F(1),\ldots ,{r_F}(k)\in [\de_1,\frac{\de}{2}]$
with $r_F(1)>4r_F(2)>\ldots >4^{k-1}r_F(k)$ and
and related domains $\{\Delta_1,\ldots,\Delta_k\}$ satisfying
items 1, 2, 3, (A) and (E) of Theorem~\ref{mainStructure},
with respect to some
constant $A_1=A_1(\L)\in [A_0,\infty)$.
It remains to prove that   $A_1=A_1(\L)\in [A_0,\infty)$ can also be
chosen sufficiently large so that items (B), (C) and (D) of
Theorem~\ref{mainStructure} also hold for each $\Delta=\Delta_i$,
$i=1,\ldots, k$. Otherwise,
for some $i=1,\ldots, k$, at least one of the items (B), (C) and (D)
of the theorem fails to hold for $\Delta=\Delta_i$;
without loss of generality, assume that $\Delta=\Delta_1$.
In this case we may consider $F|_{\Delta}\colon \Delta\la X$ to be an
element in $\L'=\L(I, H_0,\de_1/3,A_1,K_0)$ (regarding the bound $A_1$
of the second fundamental form of $F|_{\Delta}$ in the
$\frac{\de_1}{3}$-neighborhood of its boundary,
see the two paragraphs just
after Definition~\ref{def5.4}). The failure of the Structure Theorem
to hold for $\Delta$, no matter how large one chooses $A_1$,
leads to a sequence
$(F_n\colon \Delta(p_1(n))\la  X_n)\in\L(I, H_0,\de_1/3,A_1,K_0)$,
where the norm of the second fundamental form of $F_n$
has a maximum value greater than $n$
at $p_1(n)\in \Delta$. By our previous arguments,
after replacing by a subsequence,
$(F_n\colon \Delta(p_1(n))\la  X_n)$ leads to the creation of a hierarchy
$\cH(\Delta)$ for $\Delta=\Delta(n)$.
It is this hierarchy that we are referring to in the statement of
Theorem~\ref{thm9.6} below.

The notion of hierarchy $\mathcal{H}(\Delta)$ has a good behavior with
respect to proving properties by induction on the number  $L$ of levels,
which will be the method of proof of Theorem~\ref{thm9.6} below.
Observe that the truncation of a hierarchy $\mathcal{H}(\Delta)$ with
$L\geq 1$ levels by simply deleting its top level is again a hierarchy,
with the only difference that the role of $\Delta$ is played
by the disjoint
union of the compact pieces $\Delta_{q}$ with $q\in \wh{\cS}_{\rm T}$.
To simplify the notation in the next statement, we will denote again by
$\Delta$ this disjoint union, and so we will  no longer assume that
$\Delta$ is connected;   by hierarchy of such a disconnected $\Delta$,
we mean the union of the hierarchies of the components of $\Delta$.

\begin{theorem}
\label{thm9.6}
Let $\Delta$ be as  described previously and be finitely connected.
Then, the index $I(\Delta)$ of $\Delta$ can be estimated from below by
\begin{equation} \label{GenCMf}
6\,I(\Delta)\geq -\chi(\Delta)+2S(\Delta)+e(\Delta) + C(\mathcal{H}),
\end{equation}
where $\chi(\Delta)$ is the Euler characteristic of $\Delta$,
$e(\Delta)=\#_c(\partial \Delta)$ the number of boundary components,
$S(\Delta)$ is the total spinning of $\Delta$ along its boundary, and
the ``correction term''$C(\mathcal{H})$ is the following non-negative
integer, which depends on the complexity of the hierarchy $\cH$ of
$\Delta$:
\begin{equation}
C(\mathcal{H})=3I^*(\mathcal{H})+|\wh{\cS}|-L+|\mathcal{W}^{nt,f}|
+2|\mathcal{W}^{nt,nf}(\partial=1)|+3|\mathcal{W}^{nt,nf,or}(\partial>1)|,
\label{correction}
\end{equation}
where $\wh{\cS}$ is the singular set of the hierarchy $\mathcal{H}$ and
$L\geq 0$ is the number of its levels.
Furthermore, if $\Delta$ is connected and has a trivial hierarchy,
then $I^*(\mathcal{H})=I(\Delta)-1$, $C(\mathcal{H})= 3I(\Delta)-3$,
and so, \eqref{GenCMf} reduces to the Chodosh-Maximo
estimate~\eqref{eq:CM1}.
\end{theorem}
\begin{remark}
{\rm If $\Delta $ is orientable, the relation
$\chi(\Delta)=2\#_c(\Delta)-2g(\Delta)-e(\Delta)$
allows us to write \eqref{GenCMf} as
\begin{equation} \label{CM:5.43}
6\,I(\Delta)\geq 2g(\Delta)+2S(\Delta)+2e(\Delta)-2\#_c(\Delta)
+ C(\mathcal{H}).
\end{equation}
}
\end{remark}
\begin{proof}[Proof of Theorem~\ref{thm9.6}]
First observe that the functions
$I(\Delta),\chi(\Delta),S(\Delta),e(\Delta)$
are additive on components of $\Delta$. The same holds for
$C(\mathcal{H})$, with the understanding that adding components
of $\Delta$ also adds the number of levels as well as the other terms
appearing in~\eqref{correction}. Therefore,
\eqref{GenCMf} holds if it holds for
connected $\Delta$. The proof of~\eqref{GenCMf} will be carried out by
induction on the number $L\geq 0$ 	of levels of $\mathcal{H}(\Delta)$.

Suppose first that $\Delta$ is connected and its hierarchy $\mathcal{H}$
is trivial. In this case, $L=0$, and $|\wh{\cS}|=|\mathcal{W}^{nt,f}|
=|\mathcal{W}^{nt,nf}(\partial=1)|=
|\mathcal{W}^{nt,nf,or}(\partial>1)|=0$,	
hence $C(\mathcal{H})=3I^*(\mathcal{H})=3I(\Delta)-3$,
which reduces~\eqref{GenCMf} to~\eqref{eq:CM1}.  This argument also
proves the last statement in the theorem holds.

By the principle of mathematical induction, assume that $L>0$ is the
number of levels of $\Delta $ and that~\eqref{GenCMf} holds for (possibly
disconnected) $\Delta'$ if its hierarchy $\cH'$ has less than $L$ levels.
Without loss of generality, we will \underline{assume that $\Delta $
is connected}. Since $L>0$, then $\mathcal{H}(\Delta)$
is non-trivial, $\wh{\cS}_{\rm T}\neq
\varnothing$ and $\mathcal{V}_{\rm T}\neq \varnothing$.

By~\eqref{mnm1}, the set $\mathcal{V}_{\rm T}$ can be written as the
disjoint union
\begin{equation}
\mathcal{V}_{\rm T}=\mathcal{V}_{\rm T}^m\cup \mathcal{V}_{\rm T}^{nm},
\label{mnm2}
\end{equation}
where $\mathcal{V}_{\rm T}^m=\mathcal{V}_{\rm T}\cap \mathcal{V}^m$ and
$\mathcal{V}_{\rm T}^{nm}=\mathcal{V}_{\rm T}\cap \mathcal{V}^{nm}$.
	
In the first paragraph after Definition~\ref{def7.22} we explained that
for $n$ large, $\Delta=\Delta(n)$ can be decomposed into the compact
pieces $\Delta_{q}$
with $q\in \cS_{\rm T}$ and finitely many compact connected domains
$W(n)$ whose indices are independent of $n$ and satisfy
\[
I(W(n))=I(W),
\]
for some component $W\in \mathcal{W}\cap \S_{\rm T}$. This equality,
together with~\eqref{mnm2}, lead us to the inequality
\begin{equation}
I(\Delta)\geq I( \mathcal{V}_{\rm T}^{m})
+I( \mathcal{V}_{\rm T}^{nm})+I(\S_{\rm T}).
\label{9.7b}
\end{equation}
To estimate the first term in the RHS of~\eqref{9.7b}, we will
apply~\eqref{eq:CM1} to each of the components
$\Delta_{q}\in \mathcal{V}_{\rm T}^m$ (observe that
the total branching number $B$ in~\eqref{eq:CM1} vanishes
in our setting), so we get
\begin{eqnarray}
6I(\mathcal{V}_{\rm T}^m)&=&3I(\mathcal{V}_{\rm T}^m)
+3I(\mathcal{V}_{\rm T}^m)\nonumber
\\
&\stackrel{\eqref{eq:CM1}}{\geq}&-\chi(\mathcal{V}_{\rm T}^m)
+2S(\mathcal{V}_{\rm T}^m)+e(\mathcal{V}_{\rm T}^m)
-3\#_c(\mathcal{V}_{\rm T}^m)+3I(\mathcal{V}_{\rm T}^m)\nonumber
\\
&\stackrel{\eqref{I*}}{=}&-\chi(\mathcal{V}_{\rm T}^m)
+2S(\mathcal{V}_{\rm T}^m)+e(\mathcal{V}_{\rm T}^m)
+3I^*(\mathcal{V}_{\rm T}^m).
\label{X1a}
\end{eqnarray}
	
Since the number of levels of the hierarchy for each compact piece
$\Delta_{q}$ with $q\in \mathcal{V}_{\rm T}^{nm}$ is less than $L$,
we can estimate the second term in the RHS of~\eqref{9.7b} by the
induction hypothesis, hence
\begin{equation}
\label{X2a}
6I(\mathcal{V}_{\rm T}^{nm})\geq  -\chi(\mathcal{V}_{\rm T}^{nm})
+2S(\mathcal{V}_{\rm T}^{nm})+e(\mathcal{V}_{\rm T}^{nm})
+ C(\mathcal{V}_{\rm T}^{nm}),
\end{equation}
where $C(\mathcal{V}_{\rm T}^{nm})$ is the sum of the correction terms
$C(\mathcal{H}')$ with $\mathcal{H}'$ varying in the hierarchies
of all compact pieces $\Delta_{q}$ with
$q\in \mathcal{V}_{\rm T}^{nm}$.
	
To estimate the third term in the RHS of~\eqref{9.7b},
we will apply~\eqref{eq:CM1} to each of the components
of $\S_{\rm T}$, so we get
\begin{equation}  \label{STa}
3I(\S_{\rm T})\geq-\chi(\S_{\rm T})+2S(f_{\rm T})
+e(\S_{\rm T})-2B(\S_{\rm T})-3\#_c(\S_{\rm T})+\#_c(\S^{f}_{\rm T}),
\end{equation}
where $\#_c(\S^{f}_{\rm T})$ is the number of flat components of
$\S_{\rm T}$ (see item 1 of Remark~\ref{rem3.3}).
	
Thus,
\begin{eqnarray}
6I(\Delta)&\stackrel{\eqref{9.7b}}{\geq}& 6I( \mathcal{V}_{\rm T}^{m})
+6I( \mathcal{V}_{\rm T}^{nm})+3I(\S_{\rm T})+3I(\S_{\rm T})\nonumber
\\
&\stackrel{\eqref{X1a},\eqref{X2a},\eqref{STa}}{\geq}
&-\chi(\mathcal{V}_{\rm T}^m)
+2S(\mathcal{V}_{\rm T}^m)+e(\mathcal{V}_{\rm T}^m)
+3I^*(\mathcal{V}_{\rm T}^m)
\nonumber
\\
& & -\chi(\mathcal{V}_{\rm T}^{nm})
+2S(\mathcal{V}_{\rm T}^{nm})+e(\mathcal{V}_{\rm T}^{nm})
+ C(\mathcal{V}_{\rm T}^{nm})\nonumber
\\
& & -\chi(\S_{\rm T})+2S(f_{\rm T})+e(\S_{\rm T})
-2B(\S_{\rm T})-3\#_c(\S_{\rm T})
+\#_c(\S^{f}_{\rm T})+3I(\S_{\rm T}).\nonumber
\end{eqnarray}
Since $B(\S_{\rm T})=S(\mathcal{V}_{\rm T})-e(\mathcal{V}_{\rm T})
=[S(\mathcal{V}_{\rm T}^m)-e(\mathcal{V}_{\rm T}^m)]
+[S(\mathcal{V}_{\rm T}^{nm})-e(\mathcal{V}_{\rm T}^{nm})]$,
the RHS of the last expression can be written as
\begin{eqnarray}
& & -\chi(\mathcal{V}_{\rm T}^m)
+3e(\mathcal{V}_{\rm T}^m)+3I^*(\mathcal{V}_{\rm T}^m)\nonumber
\\
& & -\chi(\mathcal{V}_{\rm T}^{nm})+3e(\mathcal{V}_{\rm T}^{nm})
+ C(\mathcal{V}_{\rm T}^{nm})\nonumber
\\
& & -\chi(\S_{\rm T})+2S(f_{\rm T})
+e(\S_{\rm T})-3\#_c(\S_{\rm T})+\#_c(\S^{f}_{\rm T})
+3I(\S_{\rm T}).\nonumber
\end{eqnarray}
By using $\chi(\Delta)=\chi(\mathcal{V}_{\rm T}^m)
+\chi(\mathcal{V}_{\rm T}^{nm})
+\chi(\S_{\rm T})-e(\mathcal{V}_{\rm T}^m)-e(\mathcal{V}_{\rm T}^{nm})$,
$e(\mathcal{V}_{\rm T})=e(\mathcal{V}_{\rm T}^{m})
+e(\mathcal{V}_{\rm T}^{nm})$ and
\eqref{9.3}, we can rewrite the last displayed expression as
\begin{eqnarray}
& &-\chi(\Delta)+2S(\Delta)+e(\Delta) \label{eq:9.17}
\\
& &+2e(\mathcal{V}_{\rm T})-3\#_c(\S_{\rm T})
+3I(\S_{\rm T})+\#_c(\S^{f}_{\rm T}) \label{eq:9.18}
\\
& &+3I^*(\mathcal{V}_{\rm T}^m)+ C(\mathcal{V}_{\rm T}^{nm}).
\label{eq:9.19}
\end{eqnarray}
We next analyze the terms in~\eqref{eq:9.18}.

First note that $e(\mathcal{V}_{\rm T})=\#_c(\partial \S_{\rm T}^c)$,
where $\Sigma_{\rm T}^c$ is the surface defined in~\eqref{9.4}.
With this in mind, let $\mathcal{W}_{\rm T}^c$ denote
the set of components of
$\S_{\rm T}^c$  and we obtain the equality:
\begin{equation}\label{9.16a}
2\#_c(\partial\S_{\rm T}^c)-3\#_c(\S_{\rm T}^c)+3I(\S_{\rm T})
=\sum_{W^c\in \mathcal{W}_{\rm T}^c} (2\#_c(\partial W^c)-3 +3I(W^c)).
\end{equation}
We will analyze the sum in the RHS of~\eqref{9.16a} attending to the
following partition  of $\mathcal{W}_{\rm T}^c$ (compare to~\eqref{9.7c}
and~\eqref{9.8a}):
\begin{enumerate}[(Q1)]
\item $\mathcal{W}_{\rm T}^{c,t}$ is the subset of trivial components
in $\mathcal{W}_{\rm T}^c$.

\item $\mathcal{W}_{\rm T}^{c,nt}(\partial =1)$ is the subset of
components in $\mathcal{W}_{\rm T}^c$ that have one boundary curve
and are non-trivial. Equivalently, it is the
subset of components in $\mathcal{W}_{\rm T}^c$ that have one boundary
curve and are not flat.

\item $\mathcal{W}_{\rm T}^{c,nt,f}$ is the subset of components in
$\mathcal{W}_{\rm T}^c$ that have more than one boundary curve
and are flat.

\item $\mathcal{W}_{\rm T}^{c,nt,nf}(\partial >1)$ is the subset of
components  in $\mathcal{W}_{\rm T}^c$ having more than one boundary
curve and which are not flat.
\end{enumerate}
	
For case (Q1) we have the equality
\begin{eqnarray}
\sum_{W^c\in \mathcal{W}_{\rm T}^{c,t}} (2\#_c(\partial W^c)-3 +3I(W^c))
+\#_c(\S^{f}_{\rm T})
&=&\sum_{W^c\in \mathcal{W}_{\rm T}^{c,t}} (2-3 +0)
+\left| \mathcal{W}_{\rm T}^{c,t}\right|
+| \mathcal{W}_{\rm T}^{c,nt,f}|
\nonumber
\\
&\hspace{-2cm}=&
\hspace{-1.2cm}
-\left| \mathcal{W}_{\rm T}^{c,t}\right|
+\left| \mathcal{W}_{\rm T}^{c,t}\right|
+| \mathcal{W}_{\rm T}^{c,nt,f}|
=| \mathcal{W}_{\rm T}^{c,nt,f}|.\label{9.18a}
\end{eqnarray}
	
Regarding case (Q2), for elements $W^c\in
\mathcal{W}_{\rm T}^{c,nt}(\partial =1)$ we will estimate
$I(W^c)\geq 1$ (observe that this inequality holds even if $W^c$ is
non-orientable, by item~2 of Lemma~\ref{lema3.6}). Therefore,
\begin{equation}
\sum_{W^c\in \mathcal{W}_{\rm T}^{c,nt}(\partial =1)} \hspace{-.4cm}
(2\#_c(\partial W^c)-3 +3I(W^c))=
\sum_{W^c\in \mathcal{W}_{\rm T}^{c,nt}(\partial =1)}
\hspace{-.4cm}(2-3 +3I(W^c))
\geq 2\left| \mathcal{W}_{\rm T}^{c,nt}
(\partial =1)\right|.\label{9.19a}
\end{equation}
	
Cases (Q3) and (Q4) deal with the subset $\cW_{\rm T}^c(\partial >1)$ of
components in $\mathcal{W}_{\rm T}^c$
having more than one boundary curve. For those, we will show the
following estimate:

\begin{lemma} \label{lemma9.7}
In the situation above,
\begin{equation}  \label{9.17a}
\sum_{W^c\in \cW_{\rm T}^c(\partial >1)}(2\#_c(\partial W^c)-3)
\geq |\wh{\cS}_{\rm T}|-1.
\end{equation}
		
Let  $\cY^c$ denote the set of components
$W^c\in \cW_{\rm T}^c(\partial >1)$ which have boundary curves on at
least two different components of
$\cV_{\rm T}^c$ (defined in~\eqref{9.4a}). Then:
\begin{enumerate}
\item If $|\wh{\cS}_{\rm T}|=1$ and equality in~\eqref{9.17a} holds,
then $\cW_{\rm T}^c(\partial >1)=\varnothing$
(equivalently, $\cW_{\rm T}^c=\mathcal{W}_{\rm T}^{c,t}\cup
\mathcal{W}_{\rm T}^{c,nt}(\partial =1)$).

\item If $|\wh{\cS}_{\rm T}|>1$ and equality occurs in \eqref{9.17a},
then $\cY^c= \cW_{\rm T}^c(\partial >1)$, $W^c$ has exactly
two boundary components for each $W^c\in \cY^c$, and
$|\cY^c|=|\wh{\cS}_{\rm T}|-1$ (see Figure~\ref{figh2}).
\end{enumerate}	
\begin{figure}[h]
\begin{center}
\includegraphics[width=9cm]{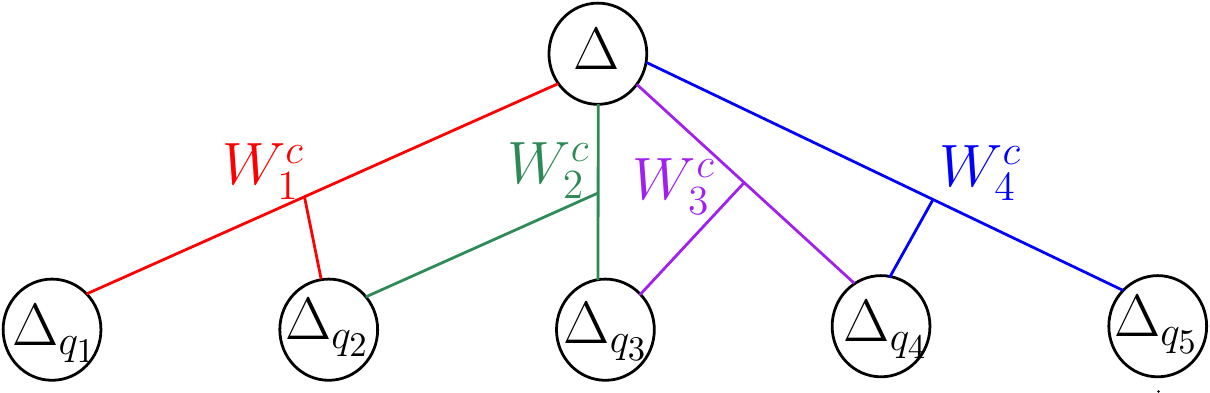}
\caption{Schematic representation of the top level of a hierarchy
	$\mathcal{H}(\Delta)$
where equality occurs in \eqref{9.17a}. Here,
$\wh{\cS}_{\rm T}=\{ q_i\ | \ i=1,\ldots ,5\} $,
$\cV_{\rm T}=\{ \Delta_{q_i}\ | \ i=1,\ldots ,5\}$,
$\cY^c=\cW_{\rm T}^c(\partial >1)=\{ W_1^c,W_2^c,W_3^c,W_4^c\}$,
$\Delta_{q_1},\Delta_{q_5}$
both have one boundary curve, while $\Delta_{q_i}$ ($i=2,3,4$)
have two boundary components each one.}
\label{figh2}
\end{center}
\end{figure}
\end{lemma}
\begin{proof}
Observe that the LHS of \eqref{9.17a} is the sum of a possibly empty set
of positive integers, where we declare this sum to be zero if this set of
positive integers is empty (equivalently,
if $\cW_{\rm T}^c(\partial >1)=\varnothing$).
Recall that $\wh{\cS}_{\rm T}\neq \varnothing$.
If $|\wh{\cS}_{\rm T}|=1$, then the RHS
of~\eqref{9.17a} is zero, hence the inequality~\eqref{9.17a}
holds in this case.
If moreover equality holds in~\eqref{9.17a},
then $\cW_{\rm T}^c(\partial >1)=\varnothing$ and so,
item~1 of the lemma holds. Hence it remains to prove
\eqref{9.17a} and item 2 of the lemma assuming that
$|\wh{\cS}_{\rm T}|>1$.
		
Let $\cY$ be the set of components $W$ of $ \S_{\rm T}$ such that
$\pi(W)$ contains at least  two  points in $\wh{\cS}_{\rm T}$.
Observe that $\cY\subset \mathcal{W}^{nt}\cap \W(\partial >1)$ and that
\[
W\in \cY \Leftrightarrow W\cap \S_{\rm T}^c\in \cY^c.
\]
Therefore,
\begin{equation}
\sum_{W^c\in \cW_{\rm T}^c(\partial >1)}(2\#_c(\partial W^c)-3)
\geq \sum_{W\in \cY}(2\#_c(\partial [W\cap \S_{\rm T}^c])-3).
\label{9.24}
\end{equation}

Since $\wh{\S}_{\rm T}$ is path-connected and
	we are assuming that $|\wh{\cS}_{\rm T}|>1$, then
for every pair of points
$q,q'\in \wh{\cS}_{\rm T}$, there exists an embedded
path $\g\colon [0,1]\to \wh{\S}_{\rm T}$ with $\g(0)=q, \g=q'$.
In particular, $\g$ contains an embedded open subarc with beginning point
$q$ and ending point $q''\in \wh{\cS}_{\rm T}\setminus\{ q\}$, such
that for some component $W(q)$ of $\cY$, we can view this open subarc as
being contained in the interior of $\pi(W(q))\setminus \wh{\cS}_{\rm T}$.
In particular, $q\in \pi (W(q))$. Since this holds for every
$q\in \wh{\cS}_{\rm T}$, then we deduce that
\[
\wh{\cS}_{\rm T}\subset \pi\left(\bigcup_{W\in \cY}W\right).
\]
Although $W(q)$ might be non-unique,
we will use the axiom of choice to assign a map
\[
q\in \wh{\cS}_{\rm T}\mapsto W(q)\in \cY \mbox{ such that }
q\in \pi (W(q)).
\]
For $q\in \wh{\cS}_{\rm T}$, let
\[
\wh{\cS}_{\rm T}(W(q))=\pi(W(q))\cap \wh{\cS}_{\rm T}.
\]
Thus, $|\wh{\cS}_{\rm T}(W(q))|\geq 2$ for each $q\in \wh{\cS}_{\rm T}$.

Notice that for each $q'\in \wh{\cS}_{\rm T}(W(q))$, $W(q)\cap
\S_{\rm T}^c$ contains at least one boundary curve in
$\partial \cD_{q'}$
(recall that $\cD_{q'}$ was defined right before~\eqref{9.4a}). Hence,
\begin{equation}\label{9.19}
\#_c(\partial [W(q)\cap \S_{\rm T}^c])\geq |\wh{\cS}_{\rm T}(W(q))|.
\end{equation}		
		
We will construct $l\in \N$ points
$q_1, q_2,\ldots, q_{l}\in \wh{\cS}_{\rm T}$, such that
$q_{i+1}\in \wh{\cS}_{\rm T}\setminus
[\cup_{j=1}^i\wh{\cS}_{\rm T}(W(q_j))]$
and $|\wh{\cS}_{\rm T}(W(q_1))\cup \ldots \cup \wh{\cS}_{\rm T}(W(q_l))|
=|\wh{\cS}_{\rm T}|$.
Choose an arbitrary $q_1\in \wh{\cS}_{\rm T}$ with a related
$W(q_1)\in \cY$. Since $|\wh{\cS}_{\rm T}(W(q_1))|\geq 2$,
\begin{eqnarray}
2\#_c(\partial [W(q_1)\cap \S_{\rm T}^c])-3
&\stackrel{\eqref{9.19}}{\geq}& 2|\wh{\cS}_{\rm T}(W(q_1))|-3
= (|\wh{\cS}_{\rm T}(W(q_1))|-1)+(|\wh{\cS}_{\rm T}(W(q_1))|-2)
\nonumber
\\
&\geq & |\wh{\cS}_{\rm T}(W(q_1))|-1.
\label{9.20}
\end{eqnarray}
		
If $|\wh{\cS}_{\rm T}(W(q_1))|=|\wh{\cS}_{\rm T}|$, then $l=1$ in our
construction of points and~\eqref{9.17a} follows from~\eqref{9.24}
and~\eqref{9.20}.
		
Suppose $|\wh{\cS}_{\rm T}(W(q_1))|<|\wh{\cS}_{\rm T}|$.
Since $\wh{\S}_{\rm T}$
is path-connected, there exists a shortest embedded arc $\a_1$ in
$\wh{\S}_{\rm T}$ from  $ \pi(W(q_1))$
to the finite set $\wh{\cS}_{\rm T}\setminus \wh{\cS}_{\rm T}(W(q_1))$
with one of its end points being  some $q_2\in \wh{\cS}_{\rm T}\setminus
\wh{\cS}_{\rm T}(W(q_1))$ and its other end point  in
$\wh{\cS}_{\rm T}(W(q_1))$.
In particular, $|\wh{\cS}_{\rm T}(W(q_1))\cap
\wh{\cS}_{\rm T}(W(q_2))|\geq 1$.
Note that
\begin{eqnarray}
\sum_{i=1}^2(2\#_c(\partial [W(q_i)\cap \S_{\rm T}^c])-3)
&\stackrel{\eqref{9.19}}{\geq}&
\sum_{i=1}^2(2|\wh{\cS}_{\rm T}(W(q_i))|-3)\nonumber
\\
&\hspace{-7.8cm}=&\hspace{-4.3cm}
\sum_{i=1}^2|\wh{\cS}_{\rm T}(W(q_i))|
+\sum_{i=1}^2(|\wh{\cS}_{\rm T}(W(q_i))|-3)
\nonumber
\\
&\hspace{-7.8cm}=&
\hspace{-4.3cm}|\wh{\cS}_{\rm T}(W(q_1))\cup \wh{\cS}_{\rm T}(W(q_2))|+
|\wh{\cS}_{\rm T}(W(q_1))\cap \wh{\cS}_{\rm T}(W(q_2))|
+\sum_{i=1}^2(|\wh{\cS}_{\rm T}(W(q_i))|-3)\nonumber
\\
&\hspace{-7.8cm}=&
\hspace{-4.3cm}
(|\wh{\cS}_{\rm T}(W(q_1))\cup \wh{\cS}_{\rm T}(W(q_2))|-1)
\nonumber
\\
& & \hspace{-4.2cm}
+(|\wh{\cS}_{\rm T}(W(q_1))\cap \wh{\cS}_{\rm T}(W(q_2))|-1)
+\sum_{i=1}^2(|\wh{\cS}_{\rm T}(W(q_i))|-2).
\label{9.27}
\end{eqnarray}
If $|\wh{\cS}_{\rm T}(W(q_1))\cup
\wh{\cS}_{\rm T}(W(q_2))|=|\wh{\cS}_{\rm T}|$, then
$l=2$ in our construction of points and~\eqref{9.17a} follows
from~\eqref{9.24} and~\eqref{9.27}.

If $|\wh{\cS}_{\rm T}(W(q_1))\cup
\wh{\cS}_{\rm T}(W(q_2))|<|\wh{\cS}_{\rm T}|$, then
there exists a shortest embedded arc $\a_2$ in $\wh{\S}_{\rm T}$ from
$\pi(W(q_1))\cup  \pi(W(q_2))$ to the finite set
$\wh{\cS}_{\rm T}\setminus [\wh{\cS}_{\rm T}(W(q_1))\cup
\wh{\cS}_{\rm T}(W(q_2))]$
with one of its end points being some  $q_3\in \wh{\cS}_{\rm T}\setminus
[\wh{\cS}_{\rm T}(W(q_1))\cup \wh{\cS}_{\rm T}(W(q_2))]$
and its other end point  in  $\wh{\cS}_{\rm T}(W(q_1))\cup
\wh{\cS}_{\rm T}(W(q_2))$.
In particular, $|[\wh{\cS}_{\rm T}(W(q_1))\cup
\wh{\cS}_{\rm T}(W(q_2))]\cap
\wh{\cS}_{\rm T}(W(q_3))|\geq 1$.
Note that
\begin{eqnarray}
\sum_{i=1}^3(2\#_c(\partial [W(q_i)\cap \S_{\rm T}^c])-3)
&\stackrel{\eqref{9.19}}{\geq}&
\sum_{i=1}^3(2|\wh{\cS}_{\rm T}(W(q_i))|-3)
\nonumber
\\
&\hspace{-7.8cm}=&\hspace{-4.3cm}
\sum_{i=1}^3|\wh{\cS}_{\rm T}(W(q_i))|
+\sum_{i=1}^3(|\wh{\cS}_{\rm T}(W(q_i))|-3)
\nonumber
\\
&\hspace{-7.8cm}=&
\hspace{-4.3cm}(|\cup_{i=1}^3\wh{\cS}_{\rm T}(W(q_i))|-1)
+(|[\wh{\cS}_{\rm T}(W(q_1))\cup \wh{\cS}_{\rm T}(W(q_2))]\cap
\wh{\cS}_{\rm T}(W(q_3))|-1)
\nonumber
\\
&\hspace{-7.8cm}+&
\hspace{-4.3cm}
(|\wh{\cS}_{\rm T}(W(q_1))\cap \wh{\cS}_{\rm T}(W(q_2)|-1))
+\sum_{i=1}^3(|\wh{\cS}_{\rm T}(W(q_i))|-2).\label{9.28}
\end{eqnarray}
If $|\cup_{i=1}^3\wh{\cS}_{\rm T}(W(q_i))|=|\wh{\cS}_{\rm T}|$,
then $l=3$ in our construction of points and~\eqref{9.17a}
follows from~\eqref{9.24} and~\eqref{9.28}.
		
If $|\cup_{i=1}^3\wh{\cS}_{\rm T}(W(q_i))|<|\wh{\cS}_{\rm T}|$,
then we repeat the above process finitely many times
(because $\wh{\cS}_{\rm T}$ is finite),
finding points $q_1,\ldots ,q_l\in \wh{\cS}_{\rm T}$ such that
$|(\cup_{i=1}^{j-1}\wh{\cS}_{\rm T}(W(q_i)))\cap
\wh{\cS}_{\rm T}(W(q_j)))|\geq 1$, for each $j=2,\ldots ,l$, and
$|\cup_{i=1}^l\wh{\cS}_{\rm T}(W(q_i))|=|\wh{\cS}_{\rm T}|$. Then,
\begin{eqnarray}
\sum_{i=1}^l(2\#_c(\partial [W(q_i)\cap \S_{\rm T}^c])-3)
&\stackrel{\eqref{9.19}}{\geq}&
\sum_{i=1}^l(2|\wh{\cS}_{\rm T}(W(q_i))|-3)
\nonumber
\\
&\hspace{-8.9cm}=&\hspace{-4.9cm}
\sum_{i=1}^l|\wh{\cS}_{\rm T}(W(q_i))|
+\sum_{i=1}^l(|\wh{\cS}_{\rm T}(W(q_i))|-3)
\nonumber
\\
&\hspace{-8.9cm}=&
\hspace{-4.9cm}
(|\cup_{i=1}^l\wh{\cS}_{\rm T}(W(q_i))|-1)
+\sum_{j=2}^l(|(\cup_{i=1}^{j-1}\wh{\cS}_{\rm T}(W(q_i)))\cap
\wh{\cS}_{\rm T}(W(q_j))|-1)
\nonumber
\\
&\hspace{-8.9cm}&
\hspace{-4.9cm}
 +\sum_{i=1}^l(|\wh{\cS}_{\rm T}(W(q_i))|-2).\label{9.29}
\end{eqnarray}
As $|\cup_{i=1}^l\wh{\cS}_{\rm T}(W(q_i))|=|\wh{\cS}_{\rm T}|$,
then~\eqref{9.17a} follows from~\eqref{9.24} and~\eqref{9.29}.

If equality in~\eqref{9.17a} occurs, then equality in~\eqref{9.24} implies
that $\cY^c=\cW_{\rm T}^c(\partial >1)$ or equivalently,
\[
\cY=\cW_{\rm T}(\partial >1)=\cW(\partial>1)\cap \S_{\rm T}:=
\mathcal{W}_{\rm T}\setminus [\mathcal{W}_{\rm T}^{t}\cup
\mathcal{W}_{\rm T}^{nt}(\partial =1)].
\]
Since the RHS of~\eqref{9.24} must be equal to the LHS of~\eqref{9.29},
we deduce that
\[
\cY=\{ W(q_i)\ | \ i=1,\ldots ,l\}
\]
and that equality holds in~\eqref{9.19} for each $i=1,\ldots,l$.
Since the third sum in the RHS of~\eqref{9.29} vanishes,
we conclude that $|\wh{\cS}_{\rm T}(W(q_i))|=2$ for each $i=1,\ldots,l$.
Finally, $|\cY^c|=|\cS_{\rm T}|-1$
because $2\#_c(\partial [W\cap \S_{\rm T}^c])-3=1$
for each $W\in \cY$. This completes the proof of Lemma~\ref{lemma9.7}.
\end{proof}
	
We continue proving Theorem~\ref{thm9.6}. We can estimate \eqref{eq:9.18}
as follows:
\begin{eqnarray}
2e(\mathcal{V}_{\rm T})-3\#_c(\S_{\rm T})
+3I(\S_{\rm T})+\#_c(\S^{f}_{\rm T})
&\stackrel{\eqref{9.16a}}{=}&\sum_{W^c\in \mathcal{W}_{\rm T}^c}
(2\#_c(\partial W^c)-3 +3I(W^c))+\#_c(\S^{f}_{\rm T})
\nonumber
\\
&\hspace{-8.5cm}\stackrel{\eqref{9.18a},\eqref{9.19a},\eqref{9.17a}}{\geq}
&\hspace{-4.5cm}| \mathcal{W}_{\rm T}^{c,nt,f}|+
2\left| \mathcal{W}_{\rm T}^{c,nt}
(\partial =1)\right| +|\wh{\cS}_{\rm T}|-1
+\sum_{W\in \cW_{\rm T}(\partial >1)}3I(W).
\label{9.30}
\end{eqnarray}

In order to bound from below the last sum in~\eqref{9.30}, note that
if $W\in \cW_{\rm T}(\partial >1)$, then either $W$ is flat (and then
$I(W)=0$), or $W$ is orientable and non-flat (in which case we estimate
$I(W)\geq 1$), or $W$ is non-orientable with $|W\cap \cS_{\rm T}|=1$
and $\#_c(\partial [W\cap \Sigma^c])>1$ (in which case we estimate
$I(W)\geq 2$ by item~2 of Lemma~\ref{lema3.6}),
or else $W$ is non-orientable
with $|W\cap \cS_{\rm T}|>1$ (in which case we estimate $I(W)\geq 0$).
Therefore, calling
\[	
\begin{array}{l}
\cW_{\rm T}^*=\{ W\in \cW_{\rm T} \ | \ W \mbox{ is non-orientable, }
|W\cap \cS_{\rm T}|=1, \ \#_c(\partial [W\cap \S_{\rm T}^c])>1\},
\\
\cW_{\rm T}^{nt,nf,or}(\partial>1)=\cW^{nt,nf,or}(\partial>1)\cap
\cW_{\rm T},
\end{array}
\]
we deduce that
\begin{equation}
\sum_{W\in \cW_{\rm T}(\partial >1)}3I(W)
\geq 6|\cW_{\rm T}^*|+3|\cW_{\rm T}^{nt,nf,or}(\partial>1)|.
\label{9.31a}
\end{equation}

Using that $|\cW_{\rm T}^*|\geq 0$, from~\eqref{9.30}
and~\eqref{9.31a} we get the following estimate from below
for~\eqref{eq:9.18}:
\begin{equation}
2e(\mathcal{V}_{\rm T})-3\#_c(\S_{\rm T})+3I(\S_{\rm T})
+\#_c(\S^{f}_{\rm T})
\geq (|\wh{\cS}_{\rm T}|-1)+| \mathcal{W}_{\rm T}^{c,nt,f}|+
2\left| \mathcal{W}_{\rm T}^{c,nt}(\partial =1)\right|
+3|\cW_{\rm T}^{nt,nf,or}(\partial>1)|.
\label{9.31}
\end{equation}
	
By the additivity in components of the correction term $C(\cH)$ defined
in~\eqref{correction}, we can write $C(\cH)$ as the sum of
$C(\cV_{\rm T}^{nm})$
plus the terms in \eqref{correction} that are added in the top level,
that is,
\begin{equation}
C(\cH)=C(\cV_{\rm T}^{nm})+\left[ 3I^*(\cV_{\rm T}^m)+
(|\wh{\cS}_{\rm T}|-1)+| \mathcal{W}_{\rm T}^{c,nt,f}|+
2\left| \mathcal{W}_{\rm T}^{c,nt}(\partial =1)\right|
+3|\cW_{\rm T}^{nt,nf,or}(\partial>1)|
\right].
\label{7.64}
\end{equation}
Thus, \eqref{9.31} and \eqref{7.64} give
\begin{equation}
2e(\mathcal{V}_{\rm T})-3\#_c(\S_{\rm T})+3I(\S_{\rm T})
+\#_c(\S^{f}_{\rm T})
\geq C(\cH)-C(\cV_{\rm T}^{nm})-3I^*(\cV_{\rm T}^m).
\label{7.65}
\end{equation}
By~\eqref{7.65}, the sum of \eqref{eq:9.18} and \eqref{eq:9.19}
is at least $C(\cH)$. Adding this last inequality with~\eqref{eq:9.17}
we obtain~\eqref{GenCMf}, as desired. This completes the
proof of Theorem~\ref{thm9.6}.
\end{proof}

\begin{definition}
{\rm
Observe that given $q\in \wh{\cS}$, the compact piece $\Delta_{q}$
has itself a hierarchy $(\wh{\cS}_q, \mathcal{V}_q,\mathcal{W}_q)$,
whose related sets are subsets of the corresponding ones for
the hierarchy of $\Delta$, i.e., $\wh{\cS}_q\subset \wh{\cS}$,
$\mathcal{V}_q\subset \mathcal{V}$,	and
$\mathcal{W}_q\subset \mathcal{W}$.
Clearly, the hierarchy of $\Delta_{q}$ has strictly less levels than the
hierarchy of $\Delta$.
We define $\mathcal{O}(\mathcal{H})\in \N\cup \{ 0\}$ to be the number of
levels in $\cH$ which consist of one $\Delta_{q}$ piece (equivalently, the
number of levels in $\cH$ whose singular set is unitary) if	$\cH$ is
non-trivial. If $\cH$ is trivial, we let $\mathcal{O}(\cH)=0$.
}
\end{definition}
For instance, the hierarchy given in item~2 of Example~\ref{example7.19}
has $\mathcal{O}(\mathcal{H})=1$, and the one in item~3 of
Example~\ref{example7.19}
(given by Figure~\ref{figh1}) has $\mathcal{O}(\mathcal{H})=2$.

\begin{corollary}
\label{corol9.12}
Let $\Delta,\cH$ be as in Theorem~\ref{thm9.6}, with $\cH$ non-trivial.
If $\Delta$ is  non-orientable, then
inequality~\eqref{GenCMf} holds after replacing $C(\cH)$ by the following
correction term:
\begin{equation} \label{9.34}
C^{no}(\mathcal{H})=:C(\cH)+6|\cW^*|\geq 3I^*(\mathcal{H})+
|\wh{\cS}|-L+2\mathcal{O}(\mathcal{H})\geq L,
\end{equation}
where $\cW^*$ is the set of components $W\in \cW$ which are non-orientable
with $|W\cap \cS|=1$ and $\#_c(\partial (W\setminus \cV^c))>1$; here
$\cV^c=\cup_{q\in \cS}\cD_q$ and $\cD_q$ was defined just
before~\eqref{9.4a}.
\end{corollary}
\begin{proof}
In passing from  \eqref{9.31a} to \eqref{9.31} in the derivation of
the correction term $C(\cH)$ of \eqref{GenCMf}, we neglected to keep the
term $6|\cW^*_{\rm T}|$ of \eqref{9.31a}. If we include this term
(which can only be non-zero provided that $\Delta$ is
non-orientable), then previous
calculations in the derivation of $C(\cH)$ imply that
inequality~\eqref{GenCMf} holds after replacing $C(\cH)$ by
$C(\cH)+6|\cW^*|$.
	
Next we prove both inequalities in~\eqref{9.34}. Both inequalities are
additive in the levels of the hierarchy, so it suffices to prove that
each level $\cH'$ of $\cH$ satisfies
\begin{equation}
C^{no}(\mathcal{H}')\geq 3I^*(\mathcal{H}')
+|\wh{\cS}(\mathcal{H}')|-1+2\mathcal{O}(\mathcal{H}')\geq 1,
\label{9.35}
\end{equation}	
\newline
where $C^{no}(\mathcal{H}')$, $I^*(\mathcal{H}')$,
$|\wh{\cS}(\mathcal{H}')|$, $\mathcal{O}(\mathcal{H}')$ denote the
related numbers referred just to the level $\mathcal{H}'$, for instance
$\wh{\cS}(\mathcal{H}')\neq \varnothing$ is the singular set of
the level $\cH'$, $C^{no}(\cH')$ is given by
\begin{equation}
\label{9.36}
C^{no}(\mathcal{H}')=3I^*(\mathcal{H}')+|\wh{\cS}(\mathcal{H}')|
-1+|\cW_{\cH'}^{nt,f}|+2|\mathcal{W}_{\cH'}^{nt,nf}(\partial=1)|
+3|\mathcal{W}_{\cH'}^{nt,nf,or}(\partial>1)|
+6|\cW^*(\cH')|,
\end{equation}
and $\mathcal{O}(\mathcal{H}')$ takes the value $1$ if
$|\wh{\cS}(\mathcal{H}')|=1$
and $0$ if $|\wh{\cS}(\mathcal{H}')|\geq 2$.

We will prove that \eqref{9.35} holds by considering two mutually
exclusive cases.
\begin{enumerate}[(a)]
\item Suppose that $|\wh{\cS}(\mathcal{H}')|\geq 2$.
In this case, the second inequality in \eqref{9.35} clearly holds. Since
$\mathcal{O}(\mathcal{H}')=0$, then the first inequality also
holds.
\item Suppose now that $|\wh{\cS}(\mathcal{H}')|=1$. Thus,
$\mathcal{O}(\mathcal{H}')=1$ and at least one of the terms
$|\cW^*(\cH')|$, $|\mathcal{W}_{\cH'}^{nt,nf}(\partial=1)|$,
$|\mathcal{W}_{\cH'}^{nt,nf,or}(\partial>1)|$
is positive, which proves that the first inequality
in \eqref{9.35} holds. The second inequality also holds since
\[
3I^*(\mathcal{H}')+ |\wh{\cS}(\mathcal{H}')|-1
+2\mathcal{O}(\mathcal{H}')\geq 2\mathcal{O}(\mathcal{H}')=2.
\]
\end{enumerate}
Hence, \eqref{9.35} holds and the corollary is proved.
\end{proof}

In order to state and prove the orientable version of
Corollary~\ref{corol9.12} we will
need the following lemma (compare to Lemma~\ref{lemma9.7}).

\begin{lemma}
\label{lemma9.11}
Let $\Delta,\cH$ be as in Theorem~\ref{thm9.6}, with $\cH$ non-trivial.	
If $\Delta $ is orientable, then,
\begin{equation}
\label{9.37}
\sum_{W^c\in \cW_{\rm T}^c(\partial >1)}(2\#_c(\partial W^c)-3 +3I(W^c))
+| \mathcal{W}_{\rm T}^{c,nt,f}|
\geq 2(|\wh{\cS}_{\rm T}|-1).
\end{equation}
Let  $\cY^c$ be as defined in Lemma~\ref{lemma9.7}. Then:
\begin{enumerate}
\item If $|\wh{\cS}_{\rm T}|=1$ and equality in~\eqref{9.37} holds, then
$\cW_{\rm T}^c(\partial >1)=\varnothing$ (equivalently, $\cW_{\rm T}^c=
\mathcal{W}_{\rm T}^{c,t}\cup \mathcal{W}_{\rm T}^{c,nt}(\partial =1)$).

\item If $|\wh{\cS}_{\rm T}|>1$ and equality occurs in \eqref{9.37},
then $\cY^c=\cW_{\rm T}^c(\partial >1)$, $W^c$ has exactly
two boundary components for each $W^c\in \cY^c$, $|\cY^c|
=|\wh{\cS}_{\rm T}|-1$ and every component in $\cY^c$ is flat.
\end{enumerate}	
\end{lemma}
\begin{proof}
If $|\wh{\cS}_{\rm T}|=1$, then \eqref{9.37} clearly holds as well as
item 1, by the same reason as in the proof of Lemma~\ref{lemma9.7}.
Assume that $|\wh{\cS}_{\rm T}|>1$. Since
$\cY^c\subset \cW_{\rm T}^c(\partial >1)$,
\begin{equation}
\sum_{W^c\in \cW_{\rm T}^c(\partial >1)}
\hspace{-.5cm}(2\#_c(\partial W^c)-3
+3I(W^c))+| \mathcal{W}_{\rm T}^{c,nt,f}|
\geq \sum_{W^c\in \cY^c}(2\#_c(\partial W^c)-3 +3I(W^c))
+| \mathcal{W}_{\rm T}^{c,nt,f}\cap \cY^c|
\label{9.38}
\end{equation}
with equality if and only if $ \mathcal{W}_{\rm T}^{c,nt,f}\subset
\cY^c=\cW_{\rm T}^c(\partial >1)$.
	
Suppose that $W^c\in \cY^c$ has $l\geq 2$ boundary curves.  If $W^c$ is
non-flat, then it makes a contribution of at least $2l$ to the RHS of
\eqref{9.38} (note that $I(W^c)\geq 1$ because $W^c$ is orientable and
non-flat); on the other hand, if $W^c$ is flat then it makes
a contribution of at least $2l-2$ to the RHS of \eqref{9.38}. Thus,
the RHS of \eqref{9.38} takes on its smallest possible value precisely
when every component of $\cY^c$ is flat; in this case, we get the next
lower estimate for the RHS of \eqref{9.38} with equality if and only if
every component of $\cY^c$ is flat:
\begin{equation}
\label{9.39}
\sum_{W^c\in \cY^c}(2\#_c(\partial W^c)-3 +3I(W^c))
+| \mathcal{W}_{\rm T}^{c,nt,f}\cap \cY^c| \geq
\sum_{W^c\in \cY^c}(2\#_c(\partial W^c)-3)+\left| \cY^c\right|.
\end{equation}
Finally, a calculation similar to the one used to prove
Lemma~\ref{lemma9.7} demonstrates that the minimum value of the RHS
of~\eqref{9.39} occurs precisely when
$\cY^c$ satisfies the second statement in Lemma~\ref{lemma9.7}; in
particular, $|\cY^c|=|\wh{\cS}_{\rm T}|-1$ in this case.
Applying~\eqref{9.17a}, we have
\begin{equation}
\label{9.40}
\sum_{W^c\in \cY^c}(2\#_c(\partial W^c)-3)+\left| \cY^c\right|
\geq \left|\wh{\cS}_{\rm T}\right|-1+|\cY^c|,
\end{equation}
with equality if and only if $|\cY^c|=|\wh{\cS}_{\rm T}|-1$ by item~2 of
Lemma~\ref{lemma9.7}, in which case	the RHS of~\eqref{9.40} equals
$2(|\wh{\cS}_{\rm T}|-1)$. This completes the proof of~\eqref{9.37}.
Item 2 of Lemma~\ref{lemma9.11} concerning $\cY^c$ follows as well from
the above discussion.  Now the proof of Lemma~\ref{lemma9.11} is finished.
\end{proof}

\begin{corollary}
\label{corol7.30}
Let $\Delta,\cH$ be as in Theorem~\ref{thm9.6}.	If $\Delta$ is
orientable, then inequality~\eqref{GenCMf} holds after replacing
$C(\cH)$ by the following correction term:
\begin{equation}
C^{or}(\mathcal{H})=3I^*(\mathcal{H})+2(|\wh{\cS}|-L)
+2|\mathcal{W}^{nt,nf}(\partial=1)|+3|\mathcal{W}^{nt,nf,or}(\partial>1)|.
\label{9.42}
\end{equation}
Furthermore, the new correction term satisfies
\begin{equation}
C^{or}(\mathcal{H})\geq
3I^*(\mathcal{H})+2(|\wh{\cS}|-L)+2\mathcal{O}(\mathcal{H})\geq 2L.
\label{9.42a}
\end{equation}
\end{corollary}
\begin{proof}
The argument is very similar to the one for proving
Corollary~\ref{corol9.12}, so
we will only focus on the differences and use the same notation.
We first check that
\begin{equation}
\label{GenCMfb}
6\,I(\Delta)\geq -\chi(\Delta)+2S(\Delta)+e(\Delta) + C^{or}(\mathcal{H}),
\end{equation}
where $	C^{or}(\mathcal{H})$ is given by equation~\eqref{9.42}.
The proof of this fact proceeds exactly as in the proof of \eqref{GenCMf}
for the correction term $C(\cH)$, except in the estimate in \eqref{9.31}
one uses Lemma~\ref{lemma9.11} to obtain
\[
2e(\mathcal{V}_{\rm T})-3\#_c(\S_{\rm T})+3I(\S_{\rm T})
+\#_c(\S^{f}_{\rm T})
\geq 2(|\wh{\cS}_{\rm T}|-1)	
+2|\mathcal{W}_{\rm T}^{nt,nf}(\partial=1)|
+3|\mathcal{W}_{\rm T}^{nt,nf,or}(\partial>1)|.
\]
This completes the proof that for $\Delta$ orientable, \eqref{GenCMfb}
holds.
	
We next prove \eqref{9.42a} holds. If $\cH$ is trivial, then $|\wh{\cS}|
=L=\mathcal{O}(\mathcal{H})=0$ and
\[
C^{or}(\cH)=3I^*(\cH)\stackrel{\eqref{I*}}{=}3I(\Delta)-3.
\]
Consequently, equality holds in the first inequality
of~\eqref{9.42a}, while the second inequality reduces to $3I(\Delta)
-3\geq 0$, which holds since $I(\Delta)\geq 1$.
Suppose in the sequel that $\cH$ is non-trivial.
By additivity, we can reduce the proof to the case that each level $\cH'$ of $\cH$
satisfies
\begin{equation}
C^{or}(\mathcal{H}')\geq 3I^*(\mathcal{H}')
+2(|\wh{\cS}(\mathcal{H}')|-1)+2\mathcal{O}(\mathcal{H}')\geq 2,
\label{9.41}
\end{equation}
where
\[
C^{or}(\mathcal{H}')=3I^*(\mathcal{H}')
+2(|\wh{\cS}(\mathcal{H}')|-1)
+2|\mathcal{W}_{\cH'}^{nt,nf}(\partial=1)|
+3|\mathcal{W}_{\cH'}^{nt,nf}(\partial>1)|.
\]	

First suppose that $|\wh{\cS}(\mathcal{H}')|\geq 2$.
In this case, the second inequality in \eqref{9.41} clearly holds. Since
$\mathcal{O}(\mathcal{H}')=0$ in this case, then the first inequality
in~\eqref{9.41} also holds.
	
Suppose now that $|\wh{\cS}(\mathcal{H}')|=1$, and so
$\mathcal{O}(\mathcal{H}')=1$. The second
inequality in~\eqref{9.41} holds because $I^*(\cH')\geq 0$ and
$2(|\wh{\cS}(\mathcal{H}')|-1)+2\mathcal{O}(\mathcal{H}')=2$.
The first inequality also holds because in this case,
$2|\mathcal{W}_{\cH'}^{nt,nf}(\partial=1)|
+2|\mathcal{W}_{\cH'}^{nt,nf}(\partial>1)|\geq 2
= 2\mathcal{O}(\mathcal{H}')$.
Hence, \eqref{9.41} holds and the corollary is proved.
\end{proof}

\begin{proposition}
	\label{propos7.31}
Let $\Delta,\cH$ be as in Theorem~\ref{thm9.6}, with $\Delta$ connected.
\begin{enumerate}
\item \label{5.31(1)}
If $I(\Delta )=1$, then $\cH$ is trivial, $\Delta$ is
orientable, $g(\Delta)=0$, and $(e(\Delta),S(\Delta))\in
\{ (2,2),(1,3)\}$.
In particular, equality in~\eqref{GenCMf} holds.

\item \label{5.31(2)}
If $\cH$ is trivial and $\Delta$ is orientable, then
$2g(\Delta)\leq 3I(\Delta)-3$, $2e(\Delta)\leq 3I(\Delta)+1$ and
$2S(\Delta)\leq 3I(\Delta)+3$.

\item \label{5.31(3)}
If $\cH$ is trivial and $\Delta$ is non-orientable, then
$I(\Delta)\geq 2$, $S(\Delta)\geq 3$,
$g(\Delta)\leq 3I(\Delta)-4$, $2e(\Delta)\leq 3I(\Delta)-2$
and $2S(\Delta)\leq 3I(\Delta)+2$.

\item \label{5.31(4a)}
If $\cH$ is non-trivial with $L>0$ levels, then $S(\Delta)\geq 2$ and
$I(\Delta)\geq 1+L$.

\item \label{5.31(4)}
If $\cH$ is non-trivial with $L>0$ levels
and $\Delta$ is orientable, then
$g(\Delta)\leq 3I(\Delta)-L-3$, $e(\Delta)\leq 3I(\Delta)-L-1$
and $S(\Delta)\leq 3I(\Delta)-L$.

\item \label{5.31(5)}
If $\cH$ is non-trivial  with $L>0$ levels and $\Delta$ is
non-orientable, then $g(\Delta)\leq 6I(\Delta)-L-7$,
$2e\leq 6I-L-3$ and $2S(\Delta)\leq 6I(\Delta)-L-1$.
\end{enumerate}
\end{proposition}
\begin{proof}
Suppose $I(\Delta)=1$. Then, the non-flat limit minimal immersion
$f_1\colon \S_1\la \R^3$ found in Section~\ref{subseclocpic1} has index 1,
and Proposition~\ref{ass3.5} implies that the hierarchy $\cH$ of
$\Delta$ is trivial. Furthermore, Theorem~1.8 in~\cite{ChMa2} ensures
that $f_1$ must be two-sided, and since the index of $f_1$ is one,
$f_1(\S_1)$ is either a catenoid or an Enneper minimal
surface~\cite{lor2}. In particular, $g(\Delta)=0$ and
$(e(\Delta),S(\Delta))\in \{ (2,2),(1,3)\}$. This proves
item~\ref{5.31(1)}.

To prove items~\ref{5.31(2)} and~\ref{5.31(3)},
suppose that $\cH$ is trivial.
By Theorem~\ref{thm9.6}, \eqref{GenCMf} can be
written as $3I(\Delta)\geq -\chi(\Delta)+2S(\Delta)+e(\Delta)-3$.
After replacing $\chi(\Delta)$ by $2-2g(\Delta)-e(\Delta)$
provided that $\Delta$ is orientable (resp. by $1-g(\Delta)-e(\Delta)$ if
$\Delta$ is non-orientable), we get
\[
\begin{array}{ccrl}
3I(\Delta)&\geq &2g(\Delta)+2e(\Delta)+2S(\Delta)-5 &
\mbox{if $\Delta$ is orientable,}
\\
3I(\Delta) & \geq & g(\Delta)+2e(\Delta)+2S(\Delta)-4  &
\mbox{if $\Delta$ is non-orientable.}
\end{array}
\]
We next discuss on the orientability character of $\Delta$.
If $\Delta$ is orientable, the estimates from above for each of
$g(\Delta),e(\Delta),S(\Delta)$ in
item~\ref{5.31(2)} of the proposition follow from a straightforward
computation using two of the inequalities $g(\Delta)\geq 0$,
$e(\Delta)\geq 1$, $S(\Delta)\geq 2$, $e(\Delta)+S(\Delta)\geq 4$.
If $\Delta$ is non-orientable (in particular, $I(\Delta)\geq 2$
by item 1 of this
proposition) and we additionally suppose that $S(\Delta)=2$, then the area
growth at infinity of $f_1$ is that of two planes, which prevents
$f_1$ from having self-intersections by the monotonicity formula for area;
therefore, $f_1$ is properly embedded in $\R^3$, which contradicts
that $\Delta$ is non-orientable. Therefore, $S(\Delta)\geq 3$ provided
that $\Delta$ is non-orientable. Now similar arguments as those
in the orientable case show that the upper estimates for
$g(\Delta),e(\Delta),S(\Delta)$ in item~\ref{5.31(3)} of the proposition
hold.

Next suppose that $\cH(\Delta)$ is non-trivial with $L>0$ levels.
This implies that we can find $L+1$ blow-up limits
\[
f_i\colon \S_i\la \R^3,\quad i=1,\ldots ,L+1
\]
of suitable rescalings $\{ \l_{i,n}F_n\}_n$ of the original sequence
$\{F_n\}_n$ as in item (S2) above (centered at possibly different
points where the second
fundamental form of $F_n$ blows-up). Since the index increases each time
we add a level (by item~8(f) of Proposition~\ref{prop7.12}), then we
deduce that  $I(\Delta)\geq L+1$. Since the total spinning of $f_1$
is at least two, the monotonicity formula implies that $S(\Delta)\geq 2$.
This completes the proof of item~\ref{5.31(4a)}.

We finish by proving item~\ref{5.31(4)} and~\ref{5.31(5)}, so continue
assuming that $\cH(\Delta)$ is non-trivial with $L>0$ levels,
and suppose $\Delta$ is connected.
In the case that $\Delta$ is orientable, we apply
Corollary~\ref{corol7.30} with the
estimate for the correction term $C^{or}(\mathcal{H})$ given
in~\eqref{9.42a}, obtaining
\begin{eqnarray}
3\,I(\Delta)&\geq &
\textstyle{-\frac{1}{2}\chi (\Delta)+S(\Delta)
+\frac{1}{2}e(\Delta)+L} \nonumber
\\
&\stackrel{(\star)}{=}&g(\Delta)+S(\Delta)+e(\Delta)-1+ L,\label{5.78'}
\end{eqnarray}
where in $(\star)$ we have used that
$\chi(\Delta)=2-2g(\Delta)-e(\Delta)$.

In the case $\Delta$ is non-orientable, we apply
Corollary~\ref{corol9.12} with the
estimate for the correction term $C^{no}(\mathcal{H})$
given in~\eqref{9.34},
obtaining
\begin{eqnarray}
6\,I(\Delta)&\geq &-\chi(\Delta)+2S(\Delta)+e(\Delta) + L \nonumber
\\
&\stackrel{(\star)'}{=}&g(\Delta)+2S(\Delta)+2e(\Delta)-1 + L,
\label{5.78} 
\end{eqnarray}
where in $(\star)'$ we have used that
$\chi(\Delta)=1-g(\Delta)-e(\Delta)$.

With inequalities~\eqref{5.78'}, \eqref{5.78} at hand, each of the
estimates from above for of $g(\Delta),e(\Delta),S(\Delta)$ in
items~\ref{5.31(4)} and~\ref{5.31(5)} of the proposition
follow from a straightforward computation using
two of the inequalities
$g(\Delta)\geq 0$, $e(\Delta)\geq 1$, $S(\Delta)\geq 2$ (which holds
by item~\ref{5.31(4a)}) and $e(\Delta)+S(\Delta)\geq 4$.
This completes the proof of the proposition.
\end{proof}

\subsection{Proofs of items (\ref{it2})-\ldots-(\ref{it8})
of Theorem~\ref{mainStructure}}
\label{sec2.9}

Next we will focus on the second step in our strategy of proving
Theorem~\ref{mainStructure}, see Section~\ref{sec3.4}.

Item~\ref{it2} of Theorem~\ref{mainStructure} follows from the fact that
$\Delta_1,\ldots ,\Delta _k$ are pairwise disjoint (by the already
proven item 1(c) of Theorem~\ref{mainStructure}).

We next prove item~\ref{it4}. The inequality $2\leq m=S(\Delta)$
for the total spinning of the boundary of $\Delta=\Delta_i$ follows since
each local picture of any element $F\in \Lambda$ has at least
either two embedded ends, or one immersed end of Enneper type,
with spinning number at least $3$; also see item~\ref{5.31(4a)} of
Proposition~\ref{propos7.31}. Item~\ref{it4}(a)
was proven in item~\ref{5.31(1)} of Proposition~\ref{propos7.31}.

Now assume $\Delta $ is orientable and $I(\Delta)\geq 2$. Then,
items~\ref{5.31(2)} and \ref{5.31(4)} of Proposition~\ref{propos7.31}
give that
\begin{eqnarray}
S(\Delta)&\leq&
\max \{ \textstyle{\frac{1}{2}(3I(\Delta)+3)},3I(\Delta)-L\}
\leq \max \{ \frac{1}{2}(3I(\Delta)+3),3I(\Delta)-1\}
=3I(\Delta)-1,
\nonumber
\\
e(\Delta)&\leq&
\max \{ \textstyle{\frac{1}{2}(3I(\Delta)+1)},3I(\Delta)-L-1\}
\leq \max \{ \frac{1}{2}(3I(\Delta)+1),3I(\Delta)-2\}
=3I(\Delta)-2,
\nonumber
\\
g(\Delta)&\leq&
\max \{ \textstyle{\frac{1}{2}(3I(\Delta)-3)},3I(\Delta)-L-3\}
\leq \max \{ \frac{1}{2}(3I(\Delta)-3),3I(\Delta)-4\}
=3I(\Delta)-4,
\nonumber
\end{eqnarray}
where $L\geq 1$ is the number of levels of the hierarchy of $\Delta$
provided this hierarchy is non-trivial. This proves item~\ref{it4}(b)
of Theorem~\ref{mainStructure}. Item~\ref{it4}(c) can be proven in
the same way, using items~\ref{5.31(3)} and \ref{5.31(5)} of
Proposition~\ref{propos7.31} (the fact that $I(\Delta)\geq 2$ follows from
items~\ref{5.31(3)} and~\ref{5.31(4a)} of Proposition~\ref{propos7.31}),
we leave the details to the reader.

The inequality $\chi(\Delta_i)\geq -6I(\Delta_i) +2m(i)+e(i)$
in item~B(d) of Theorem~\ref{mainStructure} follows directly
from~\eqref{GenCMf}: observe that the multiplicity of the multi-graph
associated to each boundary component (resp. the number of boundary
components) of $\Delta=\Delta_i$ is $m(i)$ (resp. $e(i)$)
with the notation of Theorem~\ref{mainStructure}.

The inequality $|\kappa(\Delta_i)-2\pi m(i)|\leq \frac{\tau}{m(i)}$ in
Item B(e)  of Theorem~\ref{mainStructure} follows from the multi-graphical
structure proven in item~2 of Theorem~\ref{mainStructure} and from
Lemma~\ref{ass4.4'}. As for the second inequality in Item B(e)  of
Theorem~\ref{mainStructure},
\[
\left| \kappa(\wt{M})+2\pi S\right| =
\left| -\sum_{i=1}^k\kappa(\Delta_i)+2\pi \sum_{i=1}^km(i)\right|
\leq \sum_{i=1}^k\left| \kappa(\Delta_i)-2\pi m(i)\right|
\leq \sum_{i=1}^k\frac{\tau}{m(i)}
\leq \frac{\tau}{2}k.
\]
Equation~\eqref{2.3a} follows directly from the last inequality, since
$\kappa(\wt{M})=-\sum_{i=1}^k\kappa(\Delta_i)$.

To finish the proof of item~\ref{it4} of Theorem~\ref{mainStructure},
it remains to demonstrate item~\ref{it4}(f), which we do next.
Choose a minimal  element $\Delta_q$ in the hierarchy
$\cH(\Delta)=(\wh{\cS},\cV,\cW)$
of $\cH$, being $q\in \wh{\cS}$. Then, $\Delta_q=\Delta_q(n)$ is a
connected compact surface with boundary inside $M_n$, and for $n$ large
enough, a certain rescaling of $\Delta_q(n)$ resembles arbitrarily well
the intersection with a large ball of a connected, complete, non-flat
minimal immersion $f\colon \S\la \R^3$ with finite total
curvature (see property (S1) above). As the total curvature of this limit
immersion $f$ is a negative multiple of $4\pi$ when $\S$ is orientable
and it is at least $-2\pi$ if $\S$ is non-orientable with the value
$-2\pi$ implying that $f$ is stable (see item~1 of the discussion in
Section~3 of~\cite{mpe20}), and the total curvature is invariant under
rescaling, we deduce that
\[
-\int_{\Delta_q(n)}K>3\pi
\]
for $n$ large enough.
When we ascend one level in $\cH(\Delta)$ of $\cH$ passing from
$\Delta_q$ to some $\Delta_{q'}\in \cV$ with $q'\in \wh{\cS}$ and
$\Delta_{q}\preceq\Delta_{q'}$, then a similar description holds
for $\Delta_{q'}(n)$ with $n$ large, with the difference that the related
complete minimal surface $f'\colon \S'\la \R^3$ with finite total
curvature associated to $\Delta_{q'}(n)$ may be flat, finitely
disconnected and finitely branched,
and the convergence of suitably rescaled portions of $\Delta_{q'}(n)$ to
a compact portion of $f'(\S')$ is away from finitely many points of
$f(\S')$, of which at least one corresponds to $f'(q)$. Since
\[
-\int_{\Delta_{q'}(n)}K=-\int_{\Delta_{q'}(n)\setminus \Delta_q(n)}K
-\int_{\Delta_q(n)}K
\]
and the first integral is either close to zero or larger than $3\pi$ for
$n$ large, we deduce that $-\int_{\Delta_{q'}(n)}K>3\pi$
for $n$ sufficiently large. Iterating
this process finitely many times, we get that $-\int_{\Delta}K>3\pi$,
as desired.
Adding up this last inequality in $\Delta_1,\ldots ,\Delta_k$
and using the Gauss-Bonnet formula, we deduce that
inequality~\eqref{2.4} holds. Now the proof of
item~\ref{it4} of Theorem~\ref{mainStructure} is complete.

We next prove item~\ref{it6} of Theorem~\ref{mainStructure}.
Suppose that the genus $g(M)$ of $M$ is finite and that $k\geq 1$.

Elementary surface topology of orientable surfaces implies that if $\S$
is a possibly disconnected orientable
surface (possibly with boundary) and $\Delta$ is a compact, possibly
disconnected,  smooth subsurface in the interior of $\S$,
then the genus  $g(\S)$ of $\S$, the genus $g(\Delta)$ of $\Delta$ and the
genus $g(\wt{\S})$ of $\wt{\S}=\S\setminus \Delta$ satisfy the following
inequality:
\begin{equation} \label{7.82}
g(\S)\leq g(\wt{\S}) +g(\Delta)+\#_c(\partial \Delta)-\#_c(\Delta),
\end{equation}
with equality if and only if each component of $\Delta$ does not
disconnect the component of $\S$ that contains it.

Applying \eqref{7.82} to $M$ with $\Delta=\cup_{i=1}^k\Delta_i$, gives
\begin{equation} \label{7.83}
g(M)\leq g(\wt{M}) +g(\cup_{i=1}^k \Delta_i)+\#_c(\cup_{i=1}^k
\partial \Delta_i)-k.
\end{equation}
Hence,
\begin{equation}
g(M)-g(\wt{M})\leq
\sum_{i=1}^k[g(\Delta_i)+e(\Delta_i)-1].
\label{7.84}
\end{equation}
If a domain $\Delta_i$ has
trivial hierarchy, then~\eqref{GenCMf} reduces to~\eqref{eq:CM1} and
thus,
\[
3I(\Delta_i)\geq -\chi(\Delta_i)+2S(\Delta_i)+e(\Delta_i)-3
\stackrel{(\star)}{=}2g(\Delta_i)+2e(\Delta_i)+2S(\Delta_i)-5,
\]
where in $(\star)$ we have used that
$\chi(\Delta_i)=2-2g(\Delta_i)-e(\Delta_i)$ as $\Delta_i$ must be
orientable. Therefore, in this case,
\begin{equation}
g(\Delta_i)+e(\Delta_i)-1\leq
\textstyle{ \frac{3}{2}I(\Delta_i)-S(\Delta_i)+\frac{3}{2}}.
\label{7.85}
\end{equation}
If $\Delta_i$ has non-trivial hierarchy with $L_i\geq 1$ levels,
then~\eqref{GenCMfb} and~\eqref{9.42a}   imply
\[
6I(\Delta_i)\geq -\chi(\Delta_i)+2S(\Delta_i)+e(\Delta_i)+2L_i
=2g(\Delta_i)+2e(\Delta_i)+2S(\Delta_i)+2L_i-2.
\]
Thus, in this case,
\begin{equation}
g(\Delta_i)+e(\Delta_i)-1\leq 3I(\Delta_i)-S(\Delta_i)-L_i
\leq 3I(\Delta_i)-S(\Delta_i)-1.
\label{7.86}
\end{equation}
Now \eqref{7.85} and \eqref{7.86} give the common upper bound estimate
\begin{equation}
g(\Delta_i)+e(\Delta_i)-1\leq \max\left\{
\textstyle{\frac{3}{2}I(\Delta_i)+\frac{3}{2}},
3I(\Delta_i)-1\right\} -S(\Delta_i).
\label{7.87}
\end{equation}
The function  $\max\left\{
\textstyle{\frac{3}{2}I(\Delta_i)+\frac{3}{2}},
3I(\Delta_i)-1\right\}$ has the value 3 if $I(\Delta_i)=1$ and the value
$3I(\Delta_i)-1$ if $I(\Delta_i)\geq 2$;
hence,
$\max\left\{
\textstyle{\frac{3}{2}I(\Delta_i)+\frac{3}{2}},
3I(\Delta_i)-1\right\}\leq 3I(\Delta_i)$ in all cases.
Therefore, since  it also holds that $S(\Delta_i) \geq 2$
for
all $i$, then \eqref{7.87} gives
\begin{equation}
g(\Delta_i)+e(\Delta_i)-1\leq 3I(\Delta_i)-S(\Delta_i)\leq
3I(\Delta_i)-2\qquad \mbox{for all $i=1,\ldots ,k$.}
\label{7.88}
\end{equation}
From \eqref{7.84} and~\eqref{7.88} we deduce that
\begin{equation}
g(M)-g(\wt{M})\leq \sum_{i=1}^k(3I(\Delta_i)-2)\leq 3I-2k\leq 3I-2
\label{7.89}
\end{equation}
which gives the desired inequality in item~\ref{it6} of
Theorem~\ref{mainStructure}. 

To finish the proof of Theorem~\ref{mainStructure}, it remains to
demonstrate item~\ref{it8}, which we do next. Suppose $k\geq 1$.
Item~\ref{it8} will be proven in three steps.

\begin{enumerate}[(R1)]
\label{Q1}
\item $\mbox{\rm Area}(\Delta_i)\leq 2\pi m(i)r_F(i)^2$ provided that
the constant $A_1\in [A_0,\infty)$ given in the main statement of
Theorem~\ref{mainStructure} is sufficiently large.

We will assume $i=1$ in order to use the notation introduced in
Section~\ref{sec2.7}; the cases $i\in \{ 2,\ldots ,k\} $ are similar.

Recall from property~(P1) above (and with the notation there)
that the intersection of
$F(\wt{\Delta}_1)$ between the extrinsic spheres $\partial B_X
(F(p_1),\frac{R_{s_0}}{2t})$ and $\partial B_X(F(p_1),\de_4)$
consists of $e_{s_0}$ multi-graphical annuli $\wh{G}_{s_0}(1),\ldots ,
\wh{G}_{s_0}(e_{s_0})$.
Also recall (first paragraph
after property (K2)') that $\Delta_1$ was defined as the component of
$F^{-1}(\ov{B}_X(F(p_1),r_F(1))$ that contains $p_1$, where
$r_F(1)=\de_1=\de_4/4$ and $\de_4$ is given by
Proposition~\ref{ass3.5'}.

For $j=1,\ldots, e_{s_0}$, define
\[
\wh{G}_{s_0} (j,R_{s_0}/t,r_F(1))
\]
to be portion of $F(\Delta_1)\cap
\wh{G}_{s_0}(j)$ between $\partial B_X
(F(p_1),R_{s_0}/t)$ and $\partial B_X(F(p_1),r_F(1))$. Thus,
\[
\bigcup_{j=1}^{e_{s_0}}\wh{G}_{s_0}
(j,R_{s_0}/t,r_F(1))=F(\Delta_1)\setminus
\ov{B}_X(F(p_1),R_{s_0}/t).
\]
Therefore,
\begin{equation}
\label{eq:2.21}
\frac{\mbox{\rm Area}(\Delta_1)}{\pi m(1)r_F(1)^2}=
\frac{\mbox{\rm Area}[\Delta_1\cap
F^{-1}(\ov{B}_X(F(p_1),R_{s_0}/t))]}{\pi m(1)r_F(1)^2}+
\sum_{j=1}^{e_{s_0}}
\frac{\mbox{\rm Area}(\wh{G}_{s_0} (j,R_{s_0}/t,r_F(1))}
{\pi m(1)r_F(1)^2}.
\end{equation}

Observe that for $t$ sufficiently large (equivalently, for $A_1$
sufficiently large, see equation~(\ref{eq:2.14})),
the extrinsic radius
$R_{s_0}/t$ becomes arbitrarily small (because $R_{s_0}$
is independent of $t$) and so, the first term of the
RHS of (\ref{eq:2.21}) also becomes arbitrarily small
for $A_1$ sufficiently large. Regarding  the second term of the
RHS of (\ref{eq:2.21}), observe that
\begin{equation}
\label{eq:2.22}
\sum_{j=1}^{e_{s_0}}\frac{\mbox{\rm Area}
(\wh{G}_{s_0} (j,R_{s_0}/t,r_F(1))}
{\pi m(1)r_F(1)^2}
\approx
\frac{\mbox{\rm Area}[f_{s_0}(\Sigma_{s_0})\cap \B (\vec{0},t\, r_F(1)]}
{m(1)\, \mbox{\rm Area}(\D (\vec{0},t\, r_F(1)))},
\end{equation}
where $f_{s_0}\colon \S_{s_0}\la \R^3$ is the complete, finitely branched
minimal immersion with finite total curvature
defined in the paragraph just before Proposition~\ref{ass3.5'}
(with the notation there, $\l_{s_0,n}=\frac{1}{r_{s_0,n}}$
plays the role of $t$ in our current notation),
and the symbol $\approx $ means arbitrarily close for $t$
large (to check this, re-scale the ambient metric of $X$ around
$F(p_1)$ by the factor $t$ and use either property (S2.a) or else
the adaptation of Proposition~\ref{prop7.12} after replacing $f_2$
by $f_{s_0}$).
Now, the monotonicity formula for minimal surfaces in $\R^3$ implies that
the quotient in (\ref{eq:2.22}) is less than or equal to 1
(and arbitrarily close to 1 provided that $t$ is large enough).
Therefore, (\ref{eq:2.21}) ensures that if $t$ is sufficiently large, we
have
\[
\frac{\mbox{\rm Area}(\Delta_1)}{\pi m(1)r_F(1)^2}\leq 2,
\]
which proves property (R1).

\item
$\pi \de_1^2\leq \mbox{\rm Area}(\Delta_i)$ provided that
$A_1$ is sufficiently large.

Using the notation of the already proven item~\ref{it3}
of Theorem~\ref{mainStructure}, is clearly suffices to prove that
\[
\mbox{\rm Area}\left( \cup_{h=1}^{e(i)}G_{i,h}\right)\geq \pi \de_1^2
\]
provided that $A_1$ is sufficiently large. Recall from
item~\ref{it3} of Theorem~\ref{mainStructure}
that $G_{i,h}$ is an annular
multi-graph (of multiplicity $m_{i,h}$) over its projection
$\Omega _{i,h}$
to $P_{i,h}=\varphi_{F(p_i)}(\D_h)$ and the boundary of
$G_{i,h}$ consists of two curves, each one lying on one of the extrinsic
spheres $\partial B_X(F(p_i),r_F(i)/2)$ and
$\partial B_X(F(p_i),r_F(i))$. Observe that the quotient
\[
\frac{\mbox{\rm Area}(\cup _{h=1}^{e(i)}G_{i,h})}
{\mbox{\rm Area}(\cup _{h=1}^{e(i)}\Omega_{i,h})}
\]
is invariant under re-scaling of the ambient metric centered at $F(p_i)$.
Arguing similarly as in (\ref{eq:2.22}) and with the notation there, we
have that for $t$ sufficiently large
\[
\frac{\mbox{\rm Area}(\cup _{h=1}^{e(i)}G_{i,h})}{\mbox{\rm Area}
(\cup _{h=1}^{e(i)}\Omega_{i,h})}
\approx
\frac{\mbox{\rm Area}[f_{s_0}(\Sigma_{s_0})\cap
(\ov{\B}(\vec{0},t\, r_F(i))
\setminus \ov{\B}(\vec{0},t\, r_F(i)/2)]}
{e(i)\, \pi t^2[r_F(i)^2-r_F(i)^2/4]}
\approx m(i).
\]
Therefore,
\[
\begin{array}{rcl}
\mbox{\rm Area}\left( \cup_{h=1}^{e(i)}G_{i,h}\right)
&\approx &
m(i)\, \mbox{\rm Area}(\cup _{h=1}^{e(i)}\Omega_{i,h})
\\
& \approx &
m(i)\pi [r_F(i)^2-\textstyle{\frac{r_F(i)^2}{4}}]=
m(i)\frac{3\pi}{4} r_F(i)^2
\geq m(i)\frac{3\pi}{4} \de_1^2\geq \pi \de_1^2,
\end{array}
\]
where in the last equality we have used that $m(i)\geq 2$.

\item $\mbox{\rm Area}(\widetilde{M})\geq 14\pi \sum_{i=1}^k
 m(i)r_F(i)^2$.

We continue using the notation of item (R1).
Recall that for $t$ sufficiently large,
$F(M)$ contains $e_{s_0}$ annular multi-graphs
$\wh{G}_{s_0} (1),\dots ,\wh{G}_{s_0} (e_{s_0})$ in
$\ov{B}_X(F(p_1),\de_4)\setminus
B_X(F(p_1),\frac{R_{s_0}}{2t})$ and
$e_{s_0}$ is the number of ends of $f_{s_0}$,	
and $[\wh{G}_{s_0} (1)\cup\dots \cup \wh{G}_{s_0} (e_{s_0})]\cap
B_X(F(p_1),r_F(1))$ is contained in $\Delta _1$.
Observe that the disjoint union
\[
[\wh{G}_{s_0} (1)\cup\ldots \cup \wh{G}_{s_0} (e_{s_0})]
\setminus B_X(F(p_1),r_F(1))
\]
is contained in $\wt{M}$.
A similar situation holds around each of the relative
maxima $p_2,\ldots ,p_k\in \mathcal{P}_F$ of $|A_M|$
(in the sense of item~\ref{it1}(d) of Theorem~\ref{mainStructure}), which
produces corresponding annular multi-graphs inside
$\wt{M}$ that will be denoted by
\begin{center}
\begin{tabular}{cc}
$[\wh{G}_{s_0} (1,1)\cup\dots \cup \wh{G}_{s_0} (1,e_{s_{0,1}})]
\setminus B_X(F(p_1),r_F(1)) $ & `around' $p_1$,
\\
$\cdots $ & $\cdots $
\\
$[\wh{G}_{s_0} (k,1)\cup\dots \cup \wh{G}_{s_0} (k,e_{s_{0,k}})]
\setminus B_X(F(p_k),r_F(k)) $ & `around' $p_k$,
\end{tabular}
\end{center}
all pairwise disjoint. Therefore,
\begin{equation}
\label{eq:2.23}
\mbox{\rm Area}(\wt{M})\geq \sum_{i=1}^k\mbox{Area}\left[
\left( \wh{G}_{s_0} (i,1)\cup\dots \cup
\wh{G}_{s_0} (i,e_{s_{0,i}})\right)  \setminus B_X(F(p_i),r_F(i))
\right]
\end{equation}
Given $i\in \{ 1,\ldots ,k\} $ and $h\in \{ 1,\ldots ,e_{s_0,i}\} $,
we call $\Omega'_{i,h}$ to the projection of $\wh{G}_{s_0} (i,h)\setminus
B_X(F(p_i),r_F(i)) $ to the corresponding `disk' $P_{i,h}$ defined as in
item~\ref{it3} of Theorem~\ref{mainStructure}.
Arguing as in item (R2), we have
\[
\frac{\mbox{\rm Area}\left[
(\cup _{h=1}^{e_{s_{0,i}}}\wh{G}_{s_0}(i,h))\setminus
B_X(F(p_i),r_F(i))\right]}{\mbox{\rm Area}\left(
\cup _{h=1}^{e_{s_{0,i}}}\Omega'_{i,h}\right)}
\approx
\frac{\mbox{\rm Area}[f_{s_{0,i}}(\Sigma_{s_{0,i}})\cap
(\ov{\B}(\vec{0},t\, \de_4)
\setminus \ov{\B}(\vec{0},t\, r_F(i))]}
{e_{s_0,i}\, \pi t^2[\de_4^2-r_F(i)^2]}
\approx
m(i),
\]
where $f_{s_{0,i}}\colon \Sigma_{s_{0,i}}\la \R^3$
is the corresponding complete, finitely branched minimal surface of
finite total curvature obtained as a local picture
around $F(p_i)$, and $e_{s_{0,i}}$ is the number of its ends.

Therefore,
\[
\begin{array}{rcl}
\mbox{\rm Area}\left[
(\cup _{h=1}^{e_{s_{0,i}}}\wh{G}_{s_0}(i,h))\setminus
B_X(F(p_i),r_F(i))\right]
&\approx &
m(i)\, \mbox{\rm Area}\left(
\cup _{h=1}^{e_{s_{0,i}}}\Omega'_{i,h}\right)
\\
& \approx &
m(i)\pi [\de_4^2-r_F(i)^2]=15 \, m(i)\pi r_F(i)^2.
\end{array}
\]
From here and (\ref{eq:2.23}) one concludes directly
the inequality (R3), which completes the proof of item~\ref{it8}
of Theorem~\ref{mainStructure}.
\end{enumerate}

\section{Sequential compactness results in $\L$ for $X$ fixed}
\label{sectionCompactness}
Fix $I\in \N\cup \{0\}$. An important consequence of the
statement and proof of the Structure
Theorem~\ref{mainStructure}  is that
certain sequences of immersions in $\L=\L(I, H_0,\ve_0,A_0,K_0)$
have natural limits
that are finitely branched $H$-surfaces for some $H\in [0,H_0]$.
A special case of this behavior applies to the following situation.
Suppose that $\{ F_n\colon M_n \la X\}_n$
is a sequence in $\L$ with common ambient space $X$,
the $M_n$ are connected with empty boundary,
and the norm of the second fundamental forms of $F_n$
are sufficiently large so that the points $p_1(n)\in M_n$ defined in
Theorem~\ref{mainStructure} exist and the sequence of points
$F_n(p_1(n))=x_n$ converges to $x\in X$. If in addition the norms of
the second fundamental forms of the $F_n$ are uniformly bounded,
then a subsequence of the $F_n$ converges smoothly on compact balls
in $M_n$ centered at $p_1(n)$ to a complete immersed
surface $F_\infty\colon \S\la X\in \L$
of constant mean curvature with a special point
$p_1(\infty)\in \S$ with $F_\infty(p_1(\infty))=x$.
The next theorem proves a similar result
holds when the norms of the second fundamental forms of the $F_n$ at
$p_1(n)$ diverge to $\infty$ as $n\to \infty$.
However, while the complete limit mapping $F_\infty \colon \S \la X$
in this case is smooth and defined on a limit
Riemann surface $\S$, the convergence is not smooth at a non-empty finite
set $\cB_\S\subset \S$ of points and $F_\infty$ may have
a finite set   of branch points that form a subset of $\cB_\S$, where the
total branching order is at most $3I$ and the index of $F_\infty$ is at
most $I-1$.

\begin{theorem}
\label{compactness:lemma}
Given $I\in \N\cup \{0\}$, $\tau\in (0,\pi/10]$,
let $\L=\L(I, H_0,\ve_0,A_0,K_0)$ be as given in Definition~\ref{def:L}.
Let $F_n\colon M_n \la X$ be  a sequence of $H_n$-immersions
in $\L$ with $M_n$ connected with empty boundary, and with
the supremum of the norms of their
second fundamental forms $|A_{F_n}|$ greater than the constant	
$A_1=A_1(\L)$ given in Theorem~\ref{mainStructure} and let
$\cP_{Fn}=\{p_1(n),\ldots,p_{k(n)}(n)\}$ be the associated
non-empty set of (distinct) points	given in the
statement of the same theorem, $k(n)\leq I$.
Without loss of generality and after passing to a subsequence, we can
assume that both $k(n)=k$ and $\mbox{\rm Index}(F_n)=I'\leq I$
do not depend on $n$, and that $\lim_{n\to \infty}H_n=H_{\infty}
\in [0,H_0]$.

Suppose that the points $F_n(p_1(n))$ converge as $n\to \infty$ to
a point $x_1\in X$ and the norms of the second fundamental forms of $F_n$
at the points $p_1(n)$ are unbounded. Let $k'\in \{ 1,\ldots ,k\}$ be the
cardinality of the set of points in $\cP_{F_n}$ which do not diverge
intrinsically from $p_1(n)$, i.e.  after replacing by a
further subsequence and possibly re-indexing,
\[
\lim_{n\to \infty} d_{M_n}(p_1(n),p_j(n))
=\left\{
\begin{array}{cl}
d_j\in [\frac{14}{5}\de_1,\infty) & \mbox{if $2\leq j\leq k'$,}
\\
\infty & \mbox{if $k'+1\leq j\leq k$,}
\end{array}
\right.
\]
where $\de_1>0$ is defined in Theorem~\ref{mainStructure}
and $d_2\leq \ldots \leq d_{k'}$.
For each
$i\in \{ 1,\ldots ,k'\}$, let $\Delta_i(n)\subset M_n$
be the compact subdomain described in item~\ref{it1} of
Theorem~\ref{mainStructure} that contains the point $p_i(n)$.
Then, after replacing by a further subsequence:
\begin{enumerate}
\item \label{it1thm6.1}
For each $i\in \{1,\ldots,k'\} $, the points
$F_n(p_i(n))$ converge
as $n\to \infty$ to a point $x_i\in X$, where $x_1$ is
previously defined in the hypotheses of this theorem,
and the numbers $r_{F_n}(p_i)$ converge to some
$r_i\in [\de_1,\de/2]$, where $\de\geq 2\de_1$ is defined in
Theorem~\ref{mainStructure}.

\item \label{it2thm6.1}
For each $i\in \{ 1,\ldots ,k'\}$,
the $H_n$-multi-graphical immersions
\[
F_n|_{\Delta_i(n)\setminus F_n^{-1}(B_{X}(p_i(n),r_{F_n}(i)/2))}
\]
converge, as $n\to \infty$, to a finite collection of
$e_i$ immersed compact $H_{\infty}$-annular
multi-graphs in $\ov{B}_X(x_i,r_i)\setminus B_X(x_i,r_i/2)$,
where $e_i\in \N$ is the number of boundary components
of $\Delta_i(n)$, and the multiplicity of each of these multi-graphs
is at most $3\, \mbox{\rm Index}(\Delta_i(n))\leq 3I'$;
let
\[
F_\infty^{\cA}\colon \cA\la \ov{B}_X(x_i,r_i)\setminus B_X(x_i,r_i/2)
\]
denote these explicit limit immersions, where
$\cA$ is a finite number of compact annular Riemannian surfaces.

\item \label{it3thm6.1}
There exists a partition $\{ 1,\ldots ,k'\}=\cB\cup \cU$ such that
$\{ |A_{F_n}|(p_i(n))\}_n$ is bounded (resp. unbounded)
if $i\in \cB$ (resp. $i\in \cU$); thus, we may assume
that after replacing by a further subsequence,
for each $i\in \cU$, $|A_{F_n}|(p_i(n))>n$.

\item  \label{it4thm6.1}
For each $i\in \cB$,  the restrictions $F_n|_{\Delta_i(n)}$
converge as $n\to \infty$ to an $H_{\infty}$-immersion
\[
F^i_{\infty}\colon \S_i\la \ov{B}_X(x_i,r_i)
\]
for some compact Riemannian surface $\S_i$ with boundary
diffeomorphic to $\Delta_i(n)$ for $n$ sufficiently large.
In this case, $F_{\infty}^i$ has its image boundary in
$\partial B_X(x_i,r_i)$ and its image in
$\ov{B}_X(x_i,r_i)\setminus B_X(x_i,r_i/2)$
consists of the $e_i$ multi-graphs described in
item~\ref{it2thm6.1}.

\item  \label{it5thm6.1}
For each $i\in \cU$, there exists a finitely
connected, finitely  branched $H_{\infty}$-immersion
\[
F_\infty^i\colon \S_i\la \ov{B}_X(x_i,r_i),
\]
where as in the previous case, $\S_i$ is compact with
smooth non-empty boundary and such that
we can identify $F_\infty^i$ restricted to
$(F_\infty^i)^{-1}[\ov{B}_X(x_i,r_i)
\setminus B_X(x_i,r_i/2)]$ with the multi-graphs in
item~\ref{it2thm6.1}.
Furthermore, there is a finite set 
$\cB_{\S_i}\subset (F_{\infty}^i)^{-1}[B_X(x_i,r_1/2)]$
satisfying the following properties:
\ben[(a)]
\item The set of branch points of $F^i_\infty$ is contained in
 $\cB_{\S_i}$.

\item There exist a positive integer $J(i)\leq\Ind(\Delta_i(n))\leq I'$,
a finite set of points
\[
Q_i(n)=\{q_1(i,n), \ldots, q_{J(i)}(i,n)\} \subset \Int(\Delta_i(n))
\]
with  $q_1(i,n)=p_i(n)$ and such that for each $j\in \{1,\ldots,J(i)\}$,
$|A_{F_n}|(q_j(i,n))>n$ for all $n\in \N$.
\item For any $\ve>0$ sufficiently small,
the restrictions of $ F_n$ to $\Delta_i(n)\setminus \bigcup_{q\in Q_i(n)}B_{M_n}(q,\ve) $
converge smoothly as $n\to \infty$ to
$F_{\infty}^i$ restricted to
$\S_i\setminus \bigcup_{b\in \cB_{\S_i}}
B_{\S_i}(b,\ve) $, and
\begin{itemize}
\item For $n$ sufficiently large, the number of boundary curves
of $\bigcup_{q\in Q_i(n)} B_{M_n}(q,\ve)$
coincides with the cardinality of $\cB_{\S_i}$.
\item The restriction of $F_{\infty}^i$ to
$\bigcup_{b\in \cB_{\S_i}} B_{\S_i}(b,\ve) $
is a finite collection of branched
$H_\infty$-disks, each of which can viewed as a
multi-graph in $X$ with associated finite multiplicity
$S_{\infty}(b)\in \N$ and branch point image at
$F_{\infty}^i(b)$. Hence, the  branching order
of $F^i_{\infty}$ at a given point
$b\in \cB_{\S_i}$ is equal to  $S_{\infty}(b)-1$.
\end{itemize}
\item
(Quotient space after collapsing of some points in $\cB_{\S_i}$).
For each $j\in \{1,\ldots,J(i)\}$,
there exists a non-empty subset
$\cB_{\S_i}(j)\subset \cB_{\S_i}$ which arises from the limits
of points in $\partial B_{M_n}(q_j(i,n),\ve)$
as $n\to \infty $ and
$\ve\to 0$. After identifying all the points in
$\cB_{\S_i}(j)$ to a single point, and identifying every point of
$\;\cup_{j=1}^{J(i)}(\S_i\setminus \cB_{\S_i}(j))$	with itself, we
define a quotient space $\wh{\S}_i$ and a related quotient map
$\pi_{i,j}\colon {\S}_i\to \wh{\S}_i$.
Then, the map $F_{\infty}^i$ induces a continuous map
$F^i_{\infty}\colon \wh{\S}_i\to \ov{B}_X(x_i,r_i)$,
so that the immersions $F_n|_{\Delta_i(n)}$ converge
to  $F_{\infty}^i\colon \wh{\S}_i \to \ov{B}_X(x_i,r_i)$.
\een

\item \label{it6thm6.1}
There exists a Riemann surface $\S$ and a conformal
branched $H_{\infty}$-immersion $F_\infty\colon \S\la X$
satisfying the following properties:
\ben
\item There is a conformal  embedding
$f\colon \cup_{i=1}^{k'}\S_i\to \S$ of the disjoint union $\cup_{i=1}^{k'}\S_i$
such that for any $i\in \{1,\ldots,k'\}$,
$F_\infty^i=F_\infty\circ (f|_{\S_i})$, where the mappings $F_\infty^i$
are defined in items~\ref{it4thm6.1} and~\ref{it5thm6.1}
above. Under conformal identification via $f$, henceforth consider
$\cup_{i=1}^{k'}\S_i$ to be contained in $\S$.
\item The set of branch points of $F_\infty$ is
contained in $\cup_{b\in \cB_{\S_i}} \cB_{\S_i} \subset \cup_{i=1}^{k'}\S_i$, and
so it is described in item~\ref{it5thm6.1} above.

\item $F_\infty$ can be viewed to be the limit of the immersions
$F_n$ in the following sense.  $F_\infty$ restricted to
$\S\setminus \cup_{i=1}^{k'}\S_i$ is the limit
in balls of $M_n$ centered at the points $p_1(n)$ of the
immersions $F_n\colon M_n\setminus \cup_{i=1}^k \Delta_i(n)\la X$,
and $F_\infty$ restricted to $\cup_{i=1}^{k'}\S_i$ is the limit of  $F_n$
restricted to $ \cup_{i=1}^k \Delta_i(n)$, as described
in items~\ref{it4thm6.1} and~\ref{it5thm6.1}
above.

\item The norm of the second fundamental of $F_\infty$ restricted
to $\S\setminus \cup_{i=1}^{k'}\S_i$ is bounded by $A_1$,
where $A_1$ is described in the first paragraph of the statement
of this theorem.
 \een
\end{enumerate}
\end{theorem}
\begin{proof}
Assume that Theorem~\ref{mainStructure} holds for $I$ with associated constants
$\de_1,\delta, A_1$. The fact that $k(n)$, Index$(F_n)$ are independent of $n$
after passing to a subsequence, follows trivially since they are bounded
positive integers. Similar arguments give the convergence of $H_n$ to $H_{\infty}
\in [0,H_0]$ and also item~\ref{it1thm6.1}. The convergence of $H_n$-multi-graphs
is item~\ref{it2thm6.1} is also standard, as they have uniform curvature estimates
coming from stability. Item~\ref{it3thm6.1} is also standard by an
induction argument in $k'$ and a diagonal argument. Items~\ref{it4thm6.1},
\ref{it5thm6.1} and~\ref{it6thm6.1} follow from an adaptation of the proof
Proposition~\ref{prop7.12}.
\end{proof}

\begin{corollary}
Given $I\in \N\cup \{0\}$, $\tau\in (0,\pi/10]$,
let $\L=\L(I, H_0,\ve_0,A_0,K_0)$ be as given in Definition~\ref{def:L}.
Let $F_n\colon M_n \la X$ be  a sequence of $H_n$-immersions
in $\L$ where all of the $M_n$ are connected and $X$ is compact.
Then, given  base points $q_n\in M_n$, a subsequence of the $F_n$
converges to a branched $H$-immersion
$F_\infty \colon \S\la X$ of index at most $I$,
where the convergence  as $n\to \infty$
takes place in the intrinsic balls $B_{M_n}(q(n),i)$, $ i\in \N$,
and this convergence is described in Theorem~\ref{compactness:lemma}.
\end{corollary}
\begin{remark}
{\rm
Consider a sequence $F_n\colon M_n \la X$ of complete $H_n$-immersions
in the space $\L$ as described in the statement of Theorem~\ref{compactness:lemma},
with limit branched $H_\infty$-immersion
$F_\infty\colon \S\la X$ described in item~6 of the theorem.
\begin{enumerate}
\item If $F_\infty$ has a branch point at some $q\in \S$ of branching order
$l\in \N$, then item~6(b) implies
$q\in \cup_{b\in \cB_{\S_i}} \cB_{\S_i} \subset \cup_{i=1}^{k'}\S_i$.
The proof of the theorem gives that there are blow-up points
$q(n)\in M_n$ that yield, under blowing-up, a
limit complete, possibly finitely branched minimal surface $M$ in $\rth$
with finite total curvature and such that one of the ends $E$ of $M$
has multiplicity $l+1$; such an end is not embedded, and there are
portions of the $F_n$ converging to $E$ which fail to be injective.
Hence, the existence of branch points for the limit branched
immersion  $F_\infty$ implies that for $n$ large,
the sequence $F_n$ restricted to $\cup_{i=1}^{k'}\Delta_i(n)$
is not injective.  In particular if $F_n$ is
injective for all $n\in \N$, then any limit $F_\infty\colon \S\la X$
given by the theorem has no branch points.
\item Assume $F_{\infty}$ has at least one branch point. By item~\ref{it5thm6.1}
of Theorem~\ref{compactness:lemma}, every branch point $b$ of $F_{\infty}$ lies
in some set $\cB_{\S_i}$ for some $i\in \cU$, and the branch order of $F_{\infty}$
at $b$ is equal to $S_{\infty}(b)-1$. Adding this along the set $\cB_{F_{\infty}}$
of branch points of $F_{\infty}$, we get that the total branching order of $F_{\infty}$
is at most
\[
\sum_{b\in \cB_{F_{\infty}}}[S_{\infty}(b)-1]\leq 3I-1.
\]
\end{enumerate}
}
\end{remark}
\section{Appendix A: Curvature estimates for stable $H$-surfaces}
\label{sec:curvatureEst}
Rosenberg, Toubiana and Souam~\cite[Main Theorem]{rst1}
proved that  there exists a universal constant
$C'_s>0$  such that for any $K_0\geq 0$ and any complete Riemannian
3-manifold $(Y,g)$ of absolute sectional curvature at most $K_0$, every
stable two-sided $H$-immersion $F\colon M\la  Y$ in satisfies
\begin{equation}
	\label{eqstablecurvestim}
	|A_M|(p)\leq \frac{C'_s}{\min \{ d_{M}(p,\partial M),\frac{\pi }{2\sqrt{K_0}}\} }.
\end{equation}
Observe that the above curvature estimate fails to hold when the
$H$-immersion is minimal and one-sided; a counterexample can be
constructed whenever a complete flat three-manifold $Y$ admits a
complete, non-totally geodesic,
stable one-sided minimal surface without boundary, see Remark~\ref{rem:stable} for examples.
The  next theorem is an adaptation of  \eqref{eqstablecurvestim}
that includes curvature estimates for the case of one-sided minimal surfaces in $Y$,
see also Corollaries 9 and 10 in~\cite{ros9}.

\begin{theorem}[Curvature estimate for stable $H$-surfaces]
	\label{stableestim1s}
	There exists $C''_s\geq 2\pi$ such that given $K_0>0$ and
	a complete Riemannian 3-manifold $(Y,g)$ of bounded sectional curvature
	$|K|\leq K_0$, then for any connected, immersed,
	one-sided, stable minimal surface $M\la Y$ and for any $p\in M$,
	\begin{equation}
		\label{eqcurvestim}
		|A_{M}|(p)\leq \frac{C''_s}{\min \{ \Inj_Y(p),d_{M}(p,\partial M),\frac{\pi}{2\sqrt{K_0}}\} }.
	\end{equation}
	
	Let $C_s:=\max\{C'_s,C''_s\}$, where $C'_s$ is defined by
	\eqref{eqstablecurvestim}. Given $\ve_0>0$, $K_0\geq 0$, if $X$ is a
	complete Riemannian 3-manifold with injectivity radius at
	least $\ve_0$ and bounded sectional curvature $|K|\leq K_0$,
	and $F\colon M\la X$ is a stable $H$-immersion,
	then
	\begin{equation}
		\label{eqcurvestim2}
		|A_{M}|(p)\leq \frac{C_s}{\min \{ \ve_0,d_{M}(p,\partial M),
			\frac{\pi}{2\sqrt{K_0}}\} }.
	\end{equation}
\end{theorem}

\begin{proof}
	Clearly the validity of (\ref{eqcurvestim}) implies that
	(\ref{eqcurvestim2}) holds. Also observe that by
	Remark~\ref{rem:stable}, any $C''_s>0$ that satisfies
	\eqref{eqcurvestim}	must be at least $2\pi$; in particular, $C_s\geq 2\pi$
	(in fact, $C_s\geq C_s'>2\pi$, see Remark~\ref{rem:stable}).
	
	We next prove the existence of a universal constant $C_s''$
	satisfying (\ref{eqcurvestim}) by contradiction. Since
	(\ref{eqcurvestim}) is invariant under re-scaling, by
	scaling the ambient	Riemannian metric by
	$\frac{\sqrt{K_0}}{\pi}$ we may assume that there exists a sequence
	$\{ M_n\la Y_n\}_n$ of one-sided, stable
	minimal surfaces with boundary, immersed in complete
	Riemannian 3-manifolds $(Y_n,g_n)$ with absolute
	sectional curvature $|K_{Y_n}|\leq \pi^2 $, and points
	$p_n\in M_n$ such that for all $n\in \N$,
	\begin{equation}
		\label{eq:thm6.6a}
		|A_{M_n}|(p_n)\cdot \min \{ \Inj_{Y_n}(p_n),
		d_{M_n}(p_n,\partial M_n),1/2\} \geq n.
	\end{equation}
	
	Consider the open geodesic disk $D_n\subset M_n$ of	center
	$p_n$ and radius $d_{M_n}(p_n,\partial M_n)$. Let
	$p_n^*\in D_n$ be a maximum of the continuous function
	\[
	f_n\colon D_n\to \R, \quad f_n(x)=|A_{M_n}|(x)\cdot
	\min \{ \Inj_{Y_n}(x),d_{D_n}(x,\partial D_n),1/2 \}.
	\]
	After passing to a subsequence, we can assume that one of
	the following three cases occurs for all $n\in \N$.
	\begin{enumerate}[(A)]
		\item $\min \{ \Inj_{Y_n}(p_n^*),d_{D_n}(p_n^*,\partial
		D_n),1/2 \}=\Inj_{Y_n}(p_n^*)$.
		\item $\min \{ \Inj_{Y_n}(p_n^*),d_{D_n}(p_n^*,\partial
		D_n),1/2 \} =d_{D_n}(p_n^*,\partial D_n)$.
		\item $\min \{ \Inj_{Y_n}(p_n^*),d_{D_n}(p_n^*,\partial
		D_n),1/2 \} =1/2$.
	\end{enumerate}
	
	Suppose that case (C) holds. Since $\Inj_{Y_n}(p_n^*)\geq
	1/2$, then Lemma~2.2 in~\cite{rst1} implies:
	\begin{equation}
		\label{eq:thm7.1A}
		\mbox{\em The injectivity radius function of
			$B_{Y_n}(p_n^*,1/2)$ restricted to $B_{Y_n}(p_n^*,1/8)$
			is at least $1/8$.
		}
	\end{equation}
	Applying Theorem~2.1 in~\cite{rst1} to the choices
	$M=B_{Y_n}(p_n^*,1/2)$, $\Lambda=\pi^2$, $\Omega
	=B_{Y_n}(p_n^*,1/10)$, $\Omega(\de) =B_{Y_n}(p_n^*,1/8)$,
	$i=1/8$, we conclude that every point $x\in
	B_{Y_n}(p_n^*,1/10)$ admits harmonic coordinates centered
	at $x$ and defined on the geodesic ball $B_{Y_n}(x,\ve_0)$
	for some $\ve_0>0$ independent of $x$ and $n$, and the metric
	$g_n$ is $C^{1,\a}$-controlled in the sense of
	Definition~\ref{defharm} in terms of a constant $Q>1$ which
	is also independent of $n\in \N$.
	
	Let $\l_n=|A_{M_n}|(p_n^*)$, which tends to $\infty $ as
	$n\to \infty $ because
	\begin{equation}
		\label{eq:thm7.1B'}
		\frac{1}{2}|A_{M_n}|(p_n^*)=f_n(p_n^*)\geq f_n(p_n)
		\stackrel{(\ref{eq:thm6.6a})}{\geq}n.
	\end{equation}
	
	Define $B_n'=\left( B_{Y_n}(p_n^*,1/10),\l_n^2\, g_n
	\right) $. The sequence of $3$-manifolds $\{B_n'\}_n$ converges $C^{1,\a}$
	to $\R^3$ with its standard metric and
	the harmonic coordinates in $B'_n$ centered at $p_n^*$
	converge as $n\to \infty$ to the usual harmonic coordinates
	centered at the origin.
	
	Consider the sequence of immersed, one-sided, stable minimal
	surfaces
	\[
	\Delta_n=\left( B_{M_n}(p_n^*,1/10),\l_n^2\, g_n\right)
	\la B'_n.
	\]
	Observe that the intrinsic distances in $\Delta_n$ from
	$p_n^*$ to the boundary of $\Delta _n$ diverge to
	$\infty $. We claim that the $\Delta_n$ have uniformly
	bounded second fundamental form: Take $x\in
	B_{M_n}(p_n^*,1/10)$. Since $x\in D_n$ because we are in
	case (C), and thus
	\[
	|A_{M_n}|(x)\cdot
	\min \{ \Inj_{Y_n}(x),d_{D_n}(x,\partial D_n),1/2 \}
	=f_n(x)\leq f_n(p_n^*)=\frac{\l_n}{2},
	\]
	or equivalently,
	\begin{equation}
		\label{eq:thm7.1B}
		|A_{\Delta _n}|(x)\cdot
		\min \{ \Inj_{Y_n}(x),d_{D_n}(x,\partial D_n),1/2 \}
		\leq 1/2.
	\end{equation}
	Observe that $\Inj_{Y_n}(x)\geq 1/8$ by
	(\ref{eq:thm7.1A}). Also, $d_{D_n}(x,\partial D_n)\geq 2/5$
	because $x\in B_{M_n}(p_n^*,1/10)$,
	$B_{M_n}(p_n^*,1/2)\subset D_n$ and by the triangle
	inequality. Hence, the minimum in the LHS of
	(\ref{eq:thm7.1B}) is at least $1/8$, from where we deduce
	that $|A_{\Delta _n}|(x)\leq 4$, and our claim is proved.
	
	Therefore, after passing to a subsequence, the $\Delta_n$
	converge to a complete minimal surface $S$ immersed in
	$\R^3$ with bounded second fundamental form, see the
	arguments at the beginning of Section~\ref{sec2.6}
	for details.
	
	We claim that $S$ is stable. If $S$ is two-sided, this is
	standard, see e.g.~\cite[page 636]{rst1}.  We next give a
	different argument that is valid regardless of whether $S$
	is one- or two-sided. Stability of $S$ in the one-sided
	case amounts to show that
	\begin{equation}
		\label{eq:thm7.1C}
		\int_{\wt{S}}|A_{\wt{S}}|^2\phi^2\leq \int_{\wt{S}}|\nabla
		\phi |^2,
	\end{equation}
	for every compactly supported smooth function $\phi\in
	C_0^{\infty }(\wt{S})$ defined on the two-sided cover
	$\wt{S}$ of $S$ that is anti-invariant, see
	Definition~\ref{DefIndexNO}. Given such a function $\phi $,
	we can view $\phi $ for $n$ sufficiently large as a
	compactly supported smooth function $\phi_n$ defined on the
	two-sided cover $\wt{\Delta}_n$ of $\Delta_n$
	that is	anti-invariant, and thus, by stability of
	$\Delta_n$, we have
	\begin{equation}
		\label{eq:thm7.1D}
		\int_{\wt{\Delta}_n}(|A_{\wt{\Delta}_n}|^2+
		\mbox{Ric}_{B'_n}(N_n,N_n))\phi_n^2\leq
		\int_{\wt{\Delta}_n}|\nabla \phi_n|^2,
	\end{equation}
	where Ric$_{B'_n}$ denotes the Ricci curvature of $B'_n$
	and $N_n$ is a unit normal vector to $\wt{\Delta}_n$ in
	$B'_n$. The $C^{1,\a}$ convergence of the metrics $\l_n\,
	g_n$ to the flat metric on $\R^3$ allows us to take limits
	in (\ref{eq:thm7.1D}) as $n\to \infty $ to obtain
	(\ref{eq:thm7.1C}), and thus, $S$ is stable.
	
	The desired contradiction (that proves \eqref{eqcurvestim}
	in the case that (C) holds) comes from the fact that there
	are no complete stable non-flat minimal surfaces in
	$\R^3$, see Ros~\cite[Theorem 8]{ros9}.
	
	Next we will explain how to reduce case (A) to case (C).
	If case (A) holds, we have $\Inj_{Y_n}(p_n^*)\leq 1/2$.
	Let $\mu_n=1/\Inj_{Y_n}(p_n^*)$. Define $Y_n'=\left(
	Y_n,\mu_n^2\, g_n\right) $ and $M'_n=(M_n,\mu_n^2\, g_n)$.
	Note that $\Inj_{Y'_n}(p_n^*)=1$, the absolute sectional
	curvature of $Y'_n$ is less than or equal to $\frac{\pi^2}
	{\mu_n^2}\leq \pi ^2$, which implies we may use the upper
	estimate $K_0=\pi^2$ (in other words, $(M'_n,Y_n')$ is a
	possible counterexample to (\ref{eqcurvestim}) under the
	normalization introduced in the second paragraph of this
	proof), and so $\frac{\pi}{2\sqrt{K_{0}}}=1/2$. Observe
	that $(M'_n,Y'_n)$ lies in case (C) because 	
	$d_{M'_n}(p_n^*,\partial M'_n)=\mu_n\, d_{M_n}(p_n^*,
	\partial M_n)\geq 1$, and so, $\min \{ \Inj_{Y'_n}(p_n^*),
	d_{M'_n}(p_n^*,\partial M'_n),1/2)\}=1/2$. If we check that
	\begin{equation}
		\label{eq:thm7.1E}
		|A_{M'_n}|(p_n^*)\cdot \min \{ \Inj_{Y'_n}(p_n^*),
		d_{M'_n}(p_n^*,\partial M'_n),1/2)\} \to \infty
	\end{equation}
	then we will find a contradiction as we did in case (C).
	To see this, observe that two times the LHS of
	(\ref{eq:thm7.1E}) can be written as
	\[
	|A_{M'_n}|(p_n^*)=|A_{M'_n}|(p_n^*) \cdot \Inj_{Y'_n}(p_n^*)
	=|A_{M_n}|(p_n^*)\cdot \Inj_{Y_n}(p_n^*)
	=f_n(p_n^*)\geq n \to \infty,
	\]
	which finishes the proof in case (A) occurs.
Similar reasoning reduces case (B) to case (C), which completes
	the proof of Theorem~\ref{stableestim1s}.
\end{proof}

\begin{remark}[Lower bound estimates for $C'_s$ and $C_s''$]
	\label{rem:stable}
	{\rm
		We claim that $\pi$, $2\pi$ are lower bounds for $C'_s$,
		$C''_s$, respectively. To see this, consider the Scherk's doubly periodic
		minimal surface  $M(\t)$ in $\rth$, $\t\in (0,\pi/2]$, and
		its non-orientable, embedded quotient surface $\wh{M}(\t)$
		with total curvature $-2\pi$ in the flat quotient manifold
		$Y(\t)=T_\t^2\times \R$ where $T_\t=\R^2/\mbox{Span}
		\{ w_1(\t),w_2(\t)\}$, where
		\[
		w_1(\t)=\frac{\pi}{2}\left( \frac{1}{\cos (\t /2)},
		0,0\right) ,\quad
		w_2(\t)=\frac{\pi}{2}\left( 0, \frac{1}{\sin (\t /2)},
		0\right) .
		\]
		Here, the oriented cover $\wt{M}(\t)$ of $\wh{M}(\t)$ is
		conformally $(\C\cup\{\infty\})\setminus\{e^{\pm i\t/2}\}$
		with Weierstrass data
		\[
		\left( g(z)=z, \quad \omega=\frac{i\, dz}
		{\Pi(z\pm e^{\pm i\t/2})}\right).
		\]
		Straightforward calculations show that at $z=0$ in
		$(\C\cup\{\infty\})\setminus\{e^{\pm i\pi/4}\}$ viewed as a point
		of $\wh{M}(\t)$, the absolute Gaussian curvature is given
		by $|K|(0)=16$ and this point is the unique maximum of
		$|K|$ on $\wh{M}(\t)$. On the other hand, the injectivity
		radius of $Y(\t)$ (at every point) equals $\frac{\pi}
		{4\cos (\t/2)}$ which has a maximum value of $\frac{\pi}{2\sqrt{2}}$
		at $\t=\pi/2$. Therefore, for any $\theta \in /0,\pi/2]$ we have
		\[
		|A_{\wh{M}(\t)}|\cdot \Inj_{Y(\t)}\leq
		|A_{\wh{M}(\pi/2)}|(0)\cdot \Inj_{Y(\pi/2)}
		=|4\sqrt{2}|\frac{\pi}{2\sqrt{2}}=2\pi .
		\]
		Hence the constant $C''_s$ in the above theorem must be at least $2\pi$.
		
		The standard fundamental region $Q$ for $\wh{M}(\pi/4)$  in
		$\rth$ is a vertical graph bounded by 4
		vertical lines and
		\[
		|A_Q|(0)\cdot d_{Q}(0,\partial Q)=
		4\sqrt{2}\frac{\pi}{2\sqrt{2}}=2\pi,
		\]
		so the constant $C'_s$ in  \eqref{eqstablecurvestim} also
		must be at least $2\pi$. In fact, $C'_s$ can be seen to be strictly
		greater than $2\pi$ by consideration of the intersection of
		$M(\theta)$ with a ball of radius slightly larger than
		$\frac{\pi}{2\sqrt{2}}$. Therefore,
			the constant $C_s$ given in given in the above
		 theorem also must be greater than $2\pi$.

		Next consider the translational quotient of $H$ of a
		helicoid in $\rth$ such that $H$ is an embedded, one-sided,
		stable minimal surface in $Y=\rth/(\pi \Z)$ with finite
		total curvature $-2\pi$. Let $p\in H$ be any point on the
		axis of $H$. Then,
		\[
		|A_H|\cdot \Inj_Y\leq |A_H|(p)|\cdot \Inj_Y(p)
		=|\sqrt{2}|\frac{\pi}{2} =\frac{\pi}{\sqrt{2}}.
		\]
		The slab-type region $W$ of $H$ bounded by
		two straight lines inside $H$ at distance $\pi $ apart is
		stable, and the function $p\in W\mapsto
		|A_W|(p)d_{W}(p,\partial W)$ has a maximum value at  the
		mid point of the segment obtained by intersecting the axis
		of $H$ with $W$, hence
		\[
		|A_W|(p)\cdot d_{W}(p,\partial W)\leq
		|A_W|(0)\cdot d_{W}(0,\partial W)=|\sqrt{2}|\frac{\pi}{2}
		=\frac{\pi}{\sqrt{2}}.
		\]
		The above curvature  estimates for $\wh{M}(\pi/2)$, $Q$,
		$H$ and $W$ lead us to ask the following question.
		
		\begin{question} If $M$ is a complete, one-sided, stable
			minimal surface  in a complete flat 3-manifold $Y$,
			does  the following inequality  hold?
			\[
			\mbox{For all } p\in M,\quad |A_M|(p)|\cdot \Inj_Y(p)\leq 2\pi.
			\]
More generally, does setting $C_s''=2\pi$ work in Theorem~\ref{stableestim1s}?
		\end{question}
		These questions are also motivated by the result by
		Ros~\cite{ros9} that the only complete non-flat stable
		minimal surface in a quotient of $\R^3$ by a rank one
		(resp. two) group of translations is a quotient of the
		Helicoid (resp. quotients of the Scherk doubly periodic
		minimal surfaces) with total curvature $-2\pi$.
	}
\end{remark}

\section{Appendix B: Some results in reference~\cite{mpe20} used in this paper}
\label{sec:summary}

In this section we state, for the  readers convenience, some results
from~\cite{mpe20} that we frequently apply in the proofs of the
present paper. 

\begin{proposition}[Intrinsic monotonicity of area formula, Proposition~2.4 in~\cite{mpe20}]
\label{lemma8.4}
Let $\ov{B}_X(x_0,R_1)$ denote a closed geodesic ball in an
$m$-dimensional manifold $(X,g)$, where $0<R_1\leq \Inj_X(x_0)$,
and suppose that $K_{\mbox{\rm \footnotesize sec}}\leq a$ on $B_X(x_0,R_1)$
for some $a\in \R$.
Given $H_0\geq 0$, define
\begin{equation} \label{eq:R0}
R_0(a,H_0)=\left\{ \begin{array}{cl}
\frac{1}{\sqrt{a}}\mbox{\rm arc cot}\left( \frac{H_0}{\sqrt{a}}\right) &  \mbox{if } a>0,
\\
1/H_0 & \mbox{if } a=0 \quad \mbox{(if $H_0=0$ we take $R_0(0,0)=\infty$)}
\\
\frac{1}{\sqrt{-a}}\mbox{\rm arc coth}\left( \frac{H_0}{\sqrt{-a}}\right), &  \mbox{if } a<0
\quad \mbox{(if $\frac{H_0}{\sqrt{-a}}\geq 1$ we take $R_0(a,H_0)=\infty$),}
\end{array}\right.
\end{equation}
and let $r_1=r_1(R_1,a,H_0)=\min \{ R_1,R_0(a,H_0)\}$.
	
	Suppose   $M$ is a complete, immersed, connected  $n$-dimensional submanifold of $X$
	and $x_0\in M$ is a point such that when $\partial M\neq \varnothing$,
	$d_M(x_0,\partial M)\geq R_1$ and the length of the mean curvature vector $\vec{H}$ of $M$
	restricted to $\ov{B}_X(x_0,R_1)$ is bounded from above by  $H_0$.
	Then:
	\ben
	\item If $M$ is compact without boundary, then  there exists $y\in M$
	such that the extrinsic distance from $x_0$ to $y$ is greater than or equal to $r_1$.
	\item The $n$-dimensional volume $A(r)$ of  $B_M(x_0,r)$ is
	a strictly increasing function of $r\in (0,r_1]$.
	\item For all $r\in (0,r_1]$ when $r_1\neq \infty$ or  otherwise, for all $r\in (0,\infty)$:
	\begin{equation}
		\label{eq:lemma8.4}
		A(r)\geq \left\{ \begin{array}{cc}
			\omega_n\, r^n e^{-nH_0r} & \mbox{ if $a\leq 0$},
			\\
			\omega_n\, r^n e^{-nr(H_0+\frac{1}{2}f_a(r_1)r)} & \mbox{ if $a>0$},
		\end{array}
		\right.
	\end{equation}
\een
where $\omega _n$ is the volume of the unit ball in $\R^n$ and
given $a>0$, the function $f_a\colon [0,\pi/\sqrt{a})\to \R$ is defined
by $f_a(t)=\frac{1}{t^2}\left[1-t\sqrt{a}\cot(\sqrt{a}t)\right]$, $t\in [0,\pi/\sqrt{a})$.
\end{proposition}

\begin{corollary}[Corollary~2.6 in~\cite{mpe20}]
\label{corol11.6}
Let $R_1>0$, $a\in \R$ and $H_0\geq 0$,	and suppose  that $X$
is a complete Riemannian $m$-dimensional manifold with
injectivity radius at least $R_1>0$ and  $K_{sec}\leq a$.
If $M\la X$ is a complete, non-compact immersed
$n$-dimensional submanifold with empty boundary and the mean curvature
vector $\vec{H}$ of $M$ satisfies $|\vec{H}|\leq H_0$, then $M$ has infinite
volume.
\end{corollary}

\begin{proposition}[Proposition~2.7 in~\cite{mpe20}]
\label{yau}
Given $R_1>0$, $a\in \R$ and $H_0\geq 0$, there exists
$r_2=r_2(R_1,a,H_0)\in (0,r_1]$
(here $r_1$ is given by Proposition~\ref{lemma8.4})
such that if $X$ is a complete Riemannian 3-manifold with
injectivity radius at least $R_1>0$ and
$K_{sec}\leq a$, and if $M\la X$ is
a complete, connected immersed surface with boundary,
whose mean curvature vector $\vec{H}$ satisfies
$|\vec{H}|\leq H_0$, then for all $p\in \Int(M)$ we have
\begin{equation} \label{yaulemma0}
\mathrm{Area}[B_M(p,r)]\geq 3 r^{2}, \quad
\mbox{whenever $0<r\leq \min \{r_2,d_M(p,\partial M)\} $.}
\end{equation}

Furthermore, given $\ve_0>0$ define $\ds C_A
=\min \{ \ve_0,\frac{r_2^2}{\ve_0}\} $. If $p\in M$ satisfies
$d_M(p,\partial M)\geq \ve_0$, then
\begin{equation}
\label{yaulemma2}
\mathrm{Area}[B_M(p,d_M(p,\partial M))]\geq C_A\, d_M(p,\partial M)
\end{equation}
and
\begin{equation}
\label{yaulemma1}
\mathrm{Area}[B_M(p,\ve_0)]\geq C_A\, \ve_0,
\end{equation}
\end{proposition}

We finish this summary of auxiliary results taken
from~\cite{mpe20} with the following scale-invariant weak chord-arc type
estimates for branched minimal surfaces of finite index in $\R^3$.
	
\begin{proposition}[Proposition~4.1 in~\cite{mpe20}]
\label{propos5.5}
Given $I,B\in \N\cup \{ 0\} $, let $f\colon(\S, p_0)\la (\R^3,\vec{0})$
be a complete, connected, pointed branched
minimal surface with index at most $I$ and total branching order at most $B$.
Given $R>0$, let $\Omega _R$ denote  the component
of $f^{-1}(\ov{\B}(R))$ that contains $p_0$. Then, the following
scale-invariant	estimates hold and depend only on $I$, $B$:
\ben
\item For any $p\in\Omega_R$,
\begin{equation}
\label{eq:lemma5.50}
d_{\Omega_R}(p,\partial \Omega _R)<\wh{L}R,
\end{equation}
where $\wh{L}=\sqrt{\frac{1}{2}(3I+2B+3)}$.
\item If $f $ is  injective with image  a plane, then the
distance between any two points of $\Omega_R$ is less than
or equal to $2R$. Otherwise, given points $p,q$ in $\Omega_R$,
\begin{equation}
\label{eq:lemma5.53}
d_{\Omega_{2R}}(p,q)< \wh{C}R,
\end{equation}
where $\wh{C}=\wh{C}(I,B)=8\wh{L}^3+2\pi
\wh{L}^2-20\wh{L} -\frac{\pi}{2}$. In particular,
$\Omega_{R}\subset B_{\S}(p,\wh{C}R)$ for every $p\in \Omega_R$.
\een
\end{proposition}

\center{William H. Meeks, III at  profmeeks@gmail.com\\
Mathematics Department, University of Massachusetts, Amherst, MA 01003}
\center{Joaqu\'\i n P\'{e}rez at jperez@ugr.es\\
Department of Geometry and Topology and Institute of Mathematics
(IMAG), University of Granada, 18071, Granada, Spain
\bibliographystyle{plain}
\bibliography{bill}
\end{document}